\documentclass[twoside,11pt]{article}
\usepackage[preprint]{jmlr2e}
\usepackage{amssymb,amsbsy,amsmath,amsfonts,amscd}
\usepackage{hyperref}
\usepackage{bbm}
\usepackage{latexsym}
\usepackage{listings}
\usepackage{multirow}
\usepackage{enumitem}
\usepackage{siunitx}
\usepackage{subfiles}
\usepackage{amssymb,amsbsy,amsmath,amsfonts,amscd}
\usepackage{epsfig, graphicx,subcaption,tikz}
\usepackage{nccmath,mathtools,mathrsfs}
\usepackage{algorithm,algpseudocode,xcolor,makecell}

\DeclareMathOperator*{\argmin}{arg\,min}
\newcommand{\blue}[1]{{\color{black}#1}}
\newcommand{\ts}{\tilde{s}}

\newcommand{\logks}{\log\frac{\rho_k}{\rho^*}}
\newcommand{\Cov}{\mathrm{Cov}}
\newcommand{\Var}{\mathrm{Var}}

\newcommand{\logs}{\log\frac{\rho}{\rho^*}}
\newcommand{\tr}{\tilde{r}}
\graphicspath{{"Figs/"}}
\newcommand{\KL}{\mathrm{D}_\mathrm{KL}} 
\newcommand{\FI}{\mathcal{I}}

\newcommand{\R}{\mathbb{R}}
\newcommand{\E}{\mathbb{E}}
\newcommand{\G}{\mathcal{G}}

\newcommand{\K}{\mathcal{K}}

\newtheorem{assumption}{Assumption}
\newtheorem{kernelassumption}{Assumption}[section]
\newcommand{\supp}{\text{supp}}
\newcommand{\dist}{\text{dist}}
\usepackage{array}
\newcolumntype{C}{>{\centering\arraybackslash}m{.35\linewidth}}
\newcolumntype{D}{>{\centering\arraybackslash}m{.34\linewidth}}
\newcolumntype{E}{>{\centering\arraybackslash}m{.28\linewidth}}
\newcommand{\tP}[1]{{\tilde{P}}}
\DeclareMathOperator{\Tr}{Tr}

\usepackage{lastpage}
\ShortHeadings{Convergence analysis of BRWP}{Han, Osher, and Li}

\usepackage{needspace}
\newcommand{\T}{\mathbb T}
\newcommand{\cR}{\mathcal R}
\newcommand{\cL}{\mathcal L}

\newcommand{\dd}{\,\mathrm d}
\newcommand{\norm}[1]{\left\|#1\right\|}

\newcommand{\oldresult}[1]{\setcounter{theorem}{\numexpr#1-1\relax}}
\hypersetup{colorlinks=true,linkcolor=black,citecolor=black,urlcolor=black}
\renewenvironment{proof}[1][Proof]{\par\noindent\textbf{#1.}\ }{\unskip\nobreak\hfill\BlackBox\par\medskip}
\begin{document}
\title{Convergence of Noise-Free Sampling Algorithms with Regularized Wasserstein Proximals} 
\date{}

\author{\name Fuqun Han \email fqhan@math.ucla.edu \\
       \addr Department of Mathematics\\
       University of California, Los Angeles\\
       Los Angeles, CA, USA
       \AND \name Stanley Osher \email sjo@math.ucla.edu \\
       \addr Department of Mathematics\\
       University of California, Los Angeles\\
       Los Angeles, CA, USA
       \AND \name Wuchen Li \email wuchen@mailbox.sc.edu \\
       \addr Department of Mathematics \\University of South Carolina\\ Columbia, SC, USA}
\editor{}
\maketitle

\begin{abstract}
In this work, we investigate the convergence properties of the backward regularized Wasserstein proximal (BRWP) method for sampling a target distribution. The BRWP approach is a semi-implicit time discretization of a probability flow ODE whose density satisfies the Fokker--Planck equation of the overdamped Langevin dynamics. Specifically, the evolution of the density, and hence the score function, is approximated via a kernel representation derived from the regularized Wasserstein proximal operator. We analyze the kernel formula using its dual representation and a local Taylor expansion.
\blue{Under explicit smoothness and positivity assumptions, we obtain one-step density estimates and control the score error after integration against smooth test functions.} \blue{Under additional uniform bounds on the score remainder and spatial derivatives, we establish geometric decay of the Kullback--Leibler divergence up to an $O(h^2)$ error for the BRWP update with exact kernel evaluation and a strongly log-concave target distribution, where $h$ is the stepsize.} Our analysis also identifies \blue{sufficient stepsizes for a prescribed accuracy.} Numerical experiments compare the deterministic BRWP scheme with the Unadjusted Langevin Algorithm (ULA) and other sampling methods.
\end{abstract}

% Preserve the original manuscript body size explicitly.
\small

\section{Introduction}
\label{sec_intro}
Sampling from complex and potentially high-dimensional distributions is increasingly crucial in data science \citep{andrieu2003introduction}, computational mathematics \citep{durmus2019high}, and engineering \citep{leimkuhler2015molecular}. Efficient sampling algorithms are central to numerous real-world applications, including identifying global optimizers for high-dimensional functions \citep{sampling_optimization_2019_MJ}, generating samples from latent spaces in generative modeling \citep{song2019generative}, and solving Bayesian inverse problems to estimate posterior distributions \citep{stuart2010inverse, stuart_Bayersian}. The success of these applications heavily relies on the efficiency, reliability, and theoretical convergence guarantees of the employed sampling algorithms.

Given the importance of sampling from known distributions, various algorithms have been developed and analyzed. Markov Chain Monte Carlo (MCMC) methods are widely used \citep{Nesterov_MCMC,random_work_Metroplis,hit_and_run,carrillo2022consensus,HMC}. A well-known example is the overdamped Langevin dynamics, which relies on a gradient drift vector field and a simple diffusion process introduced by Gaussian noise. Theoretical results show that the overdamped Langevin diffusion can converge to a stationary distribution under appropriate assumptions \citep{otto2000generalization}. In practice, however, challenges arise when discretizing the Langevin diffusion in time, particularly due to the dimension dependence caused by the random walk component and the bias introduced by the finite time stepsize. Notable discretizations include the Unadjusted Langevin Algorithm (ULA), which employs an explicit Euler update \citep{ULA_first}, the Metropolis-Adjusted Langevin Algorithm (MALA), which incorporates an acceptance–rejection step \citep{early_MALA}, the Proximal Langevin Algorithm, which uses implicit updates in the drift to handle nonsmooth potentials \citep{proximal_non_smooth}, and forward–backward splitting schemes formulated directly in the Wasserstein space \citep{salim_splitting, diao2023forward}. While these implicit or semi-implicit methods can partially mitigate discretization-induced bias, Wasserstein splitting approaches typically require solving a variational subproblem at each iteration, leading to substantial computational overhead except in special cases where closed-form updates are available.

Motivated by these limitations of Langevin-based approaches, an alternative perspective is to reformulate the diffusion process in terms of the score function and consider the corresponding probability flow ODE. Recent studies have focused on this probability flow ODE derived from Langevin dynamics \citep{song2019generative, lu2022dpm,chen2024probability,huang2024convergence}. This approach expresses the diffusion process, originally driven by Brownian motion, in terms of the score function, which is the gradient of the logarithm of the density function. The density itself satisfies the Fokker--Planck equation associated with overdamped Langevin dynamics. This reformulation also connects to the gradient flow of the Kullback–Leibler (KL) divergence in the Wasserstein space, the space of probability densities endowed with the Wasserstein-2 metric \citep{otto2000generalization}. In this context, the score function represents the Wasserstein gradient of the negative Boltzmann–Shannon entropy.

Applying the probability flow ODE to sampling introduces several significant challenges. First, the computation or approximation of the score function can be both computationally intensive and inaccurate when only finite samples are available. Second, the choice of time discretization is crucial: while a stable implicit scheme is desirable, it is often difficult to implement. The time-implicit Jordan–Kinderlehrer–Otto (JKO) \citep{JKO} scheme addresses this issue by iteratively evaluating the proximal operator of the KL divergence in Wasserstein space. However, the JKO update is computationally demanding due to the absence of a closed-form formula for the Wasserstein proximal operator, which typically requires solving an additional optimization problem at each time step.

To mitigate these difficulties, a regularization term can be added to the Wasserstein proximal operator, as demonstrated in \citep{kernel_proximal}, yielding a system that admits a closed-form solution via the Hopf–Cole transformation. This solution, known as the kernel formula, allows a semi-implicit time discretization of the evolution of the score function and forms the basis of the Backward Regularized Wasserstein Proximal (BRWP) scheme. The BRWP scheme, initially introduced in \citep{BRWP_2023} using empirical particle approximations for the density, was later enhanced in \citep{TT_BRWP} through tensor-train representations to accommodate high-dimensional settings. For Gaussian target distributions, earlier work gives explicit convergence and bias estimates for BRWP \citep{TT_BRWP, BRWP_2023}. These results suggest its potential applicability to a broader class of non-Gaussian distributions, motivating the more general convergence analysis.

In this work, we explore the convergence properties of the BRWP scheme for sampling from strongly log-concave distributions. The BRWP algorithm is a semi-implicit discretization of the probability flow ODE, where the score at the next time point is approximated by the kernel formula and the update is free from Brownian motion. By applying the dual formulation and a localized Taylor expansion to the kernel formula, we obtain one-step density and weak score estimates. \blue{Under Assumption~\ref{assump:SO-weak}, entropy estimates then give geometric decay of the KL divergence up to an $O(h^2)$ error and sufficient stepsizes for a prescribed accuracy.}

We summarize the main results informally as follows. We aim to sample a target distribution $\rho^* :=\frac{1}{Z}\exp(-\beta V)$ with a given potential function $V:\R^d\rightarrow\R$, a normalization constant $Z=\int_{\R^d}\exp(-\beta V(x))dx<\infty$, a constant inverse temperature $\beta$, and an initial density $\rho_0$. In particular, in Section~\ref{sec_kernel_approx}, we show that the following kernel formula from the regularized Wasserstein proximal operator (RWPO)
\begin{equation}
\label{rho_T_BRWP1}
\K_V^h\rho_{{0}}(x)  =  \int_{\R^d} \frac{\exp\big[-\frac{\beta}{2} \big(V(x) + \frac{\|x - y\|_2^2}{2h}\big)\big]}{\int_{\R^d} \exp\big[-\frac{\beta}{2} \big(V(z) + \frac{\|z - y\|_2^2}{2h}\big)\big] dz} \rho_0(y) dy\,,
\end{equation}
provides a one-step approximation to the evolution governed by the Fokker--Planck equation:
\begin{equation*}
    \frac{\partial \rho}{\partial t} = \nabla \cdot (\nabla V \rho) + \beta^{-1} \Delta \rho \,.
\end{equation*}
Here, the stepsize is given as $h>0$.

\Needspace{8\baselineskip}
\begin{theorem}[Informal, see Lemma~\ref{lem:PV-weak-local}]
\blue{Under Assumption~\ref{assump:kernel-hypotheses}}, let $\rho(t)$
solve the Fokker--Planck equation with $\rho(t_0)=\rho_0$.
For $\varphi\in C_c^{2,1}(U)$ and sufficiently small $h>0$,
\begin{equation}
\left|\left\langle
\varphi,\K_V^h\rho_0-\rho_0-h\partial_t\rho(t_0)
\right\rangle\right|
\le Ch^2\|\varphi\|_{C^{2,1}(U)}.
\end{equation}
\blue{Here $C$ depends only on the fixed data in
Assumption~\ref{assump:kernel-hypotheses}}.
\end{theorem}

Using exact kernel evaluation, we define the following semi-implicit discretization of the probability flow ODE with stepsize $h>0$:
\begin{equation}
x_{k+1} = x_k - h(\nabla V(x_k) + \beta^{-1} \nabla\log\K_V^h\rho_k(x_k))\,,
\end{equation}
where $x_k\sim\rho_k$ and $\K_V^h\rho_k$ provides {a one-step }$\mathcal{O}(h^2)$ weak approximation to  {the first-order expansion of the Fokker--Planck evolution under the preceding hypotheses}. In Section~\ref{sec_conv_KL}, we analyze the evolution of $\rho_k$ and establish convergence bounds for the proposed method in terms of the KL divergence, where $\KL(\rho_k\|\rho^*) = \int \rho_k \logks dx$.
\begin{theorem}[Informal, see Theorem~\ref{Thm_KL_decay}]
Assume $\alpha I_d\preceq\nabla^2V\preceq LI_d$ and
$D_0:=\KL(\rho_0\|\rho^*)<\infty$.
\blue{If Assumption~\ref{assump:SO-weak} holds uniformly through
the required iterations}, then, for sufficiently small $\delta>0$,
BRWP achieves $\KL(\rho_k\|\rho^*)\le\delta$ using
\[
h=c_*\sqrt{\delta},
\qquad
k=\mathcal O\!\left(
\frac{\log((1+D_0)/\delta)}{\alpha\sqrt{\delta}}
\right).
\]
Here $c_*>0$ is sufficiently small and depends only on
the fixed problem parameters and uniform assumption bounds.
\end{theorem} 

The convergence of Langevin–dynamics–based sampling methods has been extensively studied under various assumptions; see, for example, \citet{balasubramanian2022towards,durmus2019high,MALA_2019_converge,chewi2024analysis}. Table~\ref{table_complexity} compares this bound with those for ULA, proximal Langevin, and other Langevin-based methods under their respective assumptions. We further outline in Section~\ref{sec_score_implementation} practical approaches to evaluating the score.

The deterministic structure of our method also allows for explicit bounds on the KL divergence, which in turn lead to sufficient stepsize conditions. The convergence properties of related probability flow algorithms have been analyzed in \citep{gao2024convergence,chen2023improved,chen2022sampling,yang2022convergence}. In our setting, the score information is obtained through the RWPO kernel, and the theoretical analysis in Section~\ref{sec_kernel_approx} gives one-step density and score approximation estimates. The stepsize conditions in Section~\ref{sec_conv_KL} follow by balancing the KL contraction under strong convexity with the one-step entropy error controlled by Assumption~\ref{assump:SO-weak}.

Finally, in Section~\ref{sec_NE}, we present numerical experiments comparing our method against ULA and proximal-type sampling algorithms. We conclude with a brief discussion of several promising directions for future research.

\section{Review of Probability Flow ODE with Score Function, Regularized Wasserstein Proximal Operator, and BRWP Algorithm}
\label{sec_review}
In this section, we begin by outlining the sampling problem for a given target distribution and reviewing classical sampling algorithms derived from overdamped Langevin dynamics, along with their convergence analysis in Wasserstein space. We then discuss various sampling algorithms based on discrete-time approximations of the Fokker--Planck equation, including the ULA, the JKO scheme, and BRWP. To provide a comprehensive comparison, we summarize the assumptions required in the upcoming analysis,  {conditional iteration bounds, and sufficient stepsize restrictions }for BRWP and several popular existing methods in Section~\ref{sec_BRWP_algo} and Table~\ref{table_complexity}.

\subsection{Sampling problem}
We aim to generate samples \(x_{k,j}\in\mathbb{R}^d\), where \(k\in\mathbb{N}\) denotes the iteration index and \(j=1,\dots,N\) indexes the particles. The target distribution is a Gibbs measure with density\[
\rho^*(x) = \frac{1}{Z}\exp(-\beta V(x))\,,
\]
where $Z = \int_{\R^d} \exp(-\beta V(x)) dx < +\infty$ is the normalization constant, $\beta > 0$ is the inverse temperature, and $V $ represents the potential function. The collection $\{x_{0,j}\}_{j=1}^N$ is an arbitrary set of initial particles, and the objective of the sampling algorithm is to ensure that, as $k$ increases, the distribution of $\{x_{k,j}\}_{j=1}^N$ approximates the density function $\rho^*$.

We first recall the classical overdamped Langevin dynamics to motivate the BRWP method. The dynamics are described by a stochastic differential equation
\begin{equation}
\label{def:Langevin}
    dX_t = -\nabla V(X_t) dt + \sqrt{2\beta^{-1}} dB_t\,,
\end{equation}
where $B_t$ denotes the standard Brownian motion in $\R^d$. The density function of $X_t$ evolves according to the Kolmogorov forward equation, also called the Fokker--Planck equation:
\begin{equation}
\label{eqn_FP}
\begin{split}
\frac{\partial \rho}{\partial t} =& \nabla \cdot (\rho\nabla V ) + \beta^{-1} \Delta \rho \\
=&\nabla\cdot(\rho\nabla V )+\beta^{-1}\nabla\cdot(\rho \nabla\log\rho)\\
=& \beta^{-1}\nabla\cdot\left(\rho\nabla \logs \right)\,,
\end{split}
\end{equation}
where the second equality is based on the fact that $\nabla\log\rho=\nabla\rho/{\rho}$. Clearly, $\rho^*$ is an equilibrium of the Fokker--Planck equation \eqref{eqn_FP}.  {Under the strong-convexity bounds adopted below}, $\rho^*$ is the unique invariant distribution of the Langevin dynamics.

A known fact is that the Fokker--Planck equation corresponds to the Wasserstein-2 gradient flow of {$\beta^{-1}$ times }the KL divergence
\[
\mathrm{D}_{\mathrm{KL}}(\rho\|\rho^*) := \int \rho \log\frac{\rho}{\rho^*} dx\,.
\]
This means that the Fokker--Planck equation \eqref{eqn_FP} can be formulated as
\begin{equation*}
  \frac{\partial\rho}{\partial t}=\beta^{-1}\nabla\cdot\left(\rho\nabla \left(\frac{\delta}{\delta\rho}\KL(\rho\|\rho^*)\right)\right)\,,
\end{equation*}
where $\frac{\delta}{\delta\rho}$ is the $L^2$ first variation with respect to the density function $\rho$. Furthermore, in the Wasserstein space, the squared norm of the gradient of the KL divergence, also known as the relative Fisher information, is written as
\[
\FI(\rho\|\rho^*) = \int \left\|\nabla \log\frac{\rho}{\rho^*}\right\|_2^2 \rho dx\,.
\]

 {\color{black}For smooth positive solutions with finite initial entropy and the differentiation and boundary conditions of Appendix~\ref{sec:app_KL}}, we have
 \[
{\frac{d}{dt} \KL(\rho \| \rho^*) = -\beta^{-1}\FI(\rho\| \rho^*)\le0\,.
}\]
Moreover, when $\rho^*$ satisfies the log-Sobolev inequality (LSI) with constant  {$\beta\alpha$ in the convention of Appendix~\ref{sec:app_KL}}, we can establish the relationship between the KL divergence and Fisher information
\begin{equation}
\label{PL_ineq}
    \KL(\rho\|\rho^*) \leq \frac{1}{2\beta\alpha}\FI(\rho\|\rho^*)\,.
\end{equation}
This inequality can be viewed as the gradient-dominated condition, commonly known as the Polyak-Lojasiewicz (PL) inequality in the Wasserstein space. From now on, we assume that $\nabla^2 V\succeq \alpha I$, so that $\rho^*$ is strongly log-concave and the LSI holds \citep{gross1975logarithmic,bakry2006diffusions}.

Utilizing the PL inequality, one can show that the KL divergence converges exponentially fast along the Wasserstein gradient flow (see Lemma \ref{Lemma_KL_derivative} in the appendix for more details)
\[
\KL(\rho_t \| \rho^*) \leq \exp(-2\alpha t) \KL(\rho_0 \| \rho^*)\,,
\]
where $\rho_t(x) = \rho(t,x)$.

Although the continuous-time convergence is well understood, a numerical analysis of the particle trajectories arising from the discretized Langevin scheme shows that the method induces a systematic bias and has a dependence on the dimensionality of the problem.  {The size of the discretization error can depend on dimension and on the properties of the drift; the total variance of the Gaussian increment }is proportional to  dimension. More details will be provided in the next subsection.

Given these difficulties, in this work, we instead consider a probability flow ODE, where the diffusion is generated by the score function:
\begin{equation}
\label{particle_evolution}
dX_t = -\nabla V(X_t) dt - \beta^{-1} \nabla \log \rho_{t}(X_t) dt\,.
\end{equation}
The trajectories of \eqref{particle_evolution} differ from those of the Langevin dynamics \eqref{def:Langevin}. However, since the drift field satisfies $\nabla\log \rho = \nabla \rho /\rho$, the Liouville equation coincides with the Fokker--Planck equation \eqref{eqn_FP}. Consequently, the corresponding probability flow continues to be the Wasserstein gradient flow for  {$\beta^{-1}\KL(\rho\|\rho^*)$}. In continuous time, the KL divergence still converges exponentially fast along the flow given by \eqref{particle_evolution}.

Moreover, since the ODE \eqref{particle_evolution} is deterministic, it is feasible to perform higher-order numerical analysis at finite stepsizes to explore the higher-order error terms introduced by discretization and derive sufficient stepsize restrictions. This analysis will be the main focus of Section~\ref{sec_conv_KL}, where we  {derive conditional convergence and accuracy bounds}.

\subsection{Discrete time approximation}

Although the KL divergence converges exponentially fast in continuous time along the flow given by the Fokker--Planck equation, discretizing Langevin dynamics  {can alter the convergence rate and introduce }a bias term due to  discretization error. For example, applying explicit Euler discretization with stepsize $h$ to \eqref{def:Langevin} results in the ULA
\begin{equation}
\label{def_ULA}
    x_{k+1} = x_k - h\nabla V(x_k) + \sqrt{2\beta^{-1}h} z_k\,,
\end{equation}
where $z_k \sim \mathcal{N}(0,I)$. The convergence of ULA is well studied, and \citet{ULA_convergence} show
\[
\KL(\rho_k \| \rho^*) \leq \exp(-\alpha h k) \KL(\rho_0 \| \rho^*) + \frac{8 h d L^2}{\alpha}\,,
\]
where $\rho_k$ is the ULA law, $\beta=1$, $h\le\alpha/(4L^2)$, and $\alpha I\preceq\nabla^2V\preceq LI$. The discretization induces an error term of order $h$,  {with an explicit dimension factor in this bound}. This makes high-dimensional cases particularly challenging and requires a small stepsize to control the bias.

We next study the discretization of the probability flow ODE \eqref{particle_evolution} and note that one of the primary challenges is the approximation of the score function. It becomes increasingly difficult and inaccurate in high-dimensional spaces. As a result, an explicit forward Euler discretization can substantially misestimate the transport field induced by $\nabla\log\rho$. This amplification of discretization error leads to numerical instability and may cause the particle system to collapse or diverge. These difficulties motivate studying implicit and semi-implicit discretizations of the probability flow ODE.

Implicit methods for solving Langevin dynamics have also been extensively studied, including proximal Langevin \citep{wibisono2019proximal}, the proximal sampler with a restricted Gaussian oracle \citep{proximal_non_smooth}, and implicit Langevin \citep{JMLR_Euler}. In particular, at the density level, the implicit method usually corresponds to a certain splitting \citep{bernton2018langevin} of the classical JKO scheme for the Fokker--Planck equation
\begin{equation}
\label{JKO}
   \rho_{k+1} = \argmin_{\rho\in \mathcal{P}_2(\R^d)}\,\,\beta^{-1}\KL(\rho\|\rho^*) + \frac{1}{2h}W_2(\rho,\rho_k)^2\,,
\end{equation}
where $W_2(\rho,\rho_k)$ is the Wasserstein-2 distance between densities $\rho$ and $\rho_k$.

However, the aforementioned implicit implementation requires evaluating proximal operators or solving a system of nonlinear equations, which could be time-consuming and difficult for general families of density functions. In light of this, a formal semi-implicit Euler discretization method was proposed in \citep{BRWP_2023, TT_BRWP} for the particle evolution equation \eqref{particle_evolution} at time $t_k$:
\begin{equation}
\label{Euler_backward}
x_{k+1} = x_k - h\left(\nabla V(x_k) + \beta^{-1} \nabla \log \rho_{\rm next}(x_k)\right)\,.
\end{equation}
 {Here $\rho_{\rm next}$ denotes the prospective density used for the next-time score. Algorithm~\ref{algo_BRWP} evaluates that score using $\widetilde\rho_{k+1}=\K_V^h\rho_k$. The BRWP density $\rho_{k+1}$ is the pushforward of $\rho_k$ by the resulting particle map}.

\subsection{Regularized Wasserstein proximal operator and kernel formula}
Concerning \eqref{Euler_backward}, one may note that approximating the terminal density \(\rho(t_k+h,\cdot)\) is challenging due to the high-dimensional nature of \(\rho\). A natural approach for discretizing the Fokker--Planck equation is the classical JKO scheme \eqref{JKO}, but each iteration requires evaluating a Wasserstein proximal operator, which in general involves an expensive optimization problem.

To address this difficulty, we recall that the JKO update can be interpreted as a Wasserstein proximal operator associated with the KL divergence. This perspective suggests that, if a tractable surrogate for this proximal operator were available, one could design a semi-implicit discretization of the density evolution. Motivated by this idea, we consider a {regularized} Wasserstein proximal operator \citep{kernel_proximal}, which admits a closed-form solution and can therefore be used to approximate the Fokker--Planck dynamics with a small stepsize.

Let $\rho_0$ be the initial density. We first recall the Wasserstein proximal operator associated with a linear energy functional:
\begin{equation}
\label{Wass_linear_energy}
   \arg\min_{q} \bigg[ \frac{1}{2h} W_2(\rho_0, q)^2 + \int_{\R^d} V(x)\, q(x)\, dx \bigg] .
\end{equation}
Note that this differs from the Wasserstein proximal operator with KL divergence used in the JKO scheme; this connection will become clear once a Laplacian regularization is added.

By the Benamou-Brenier formula \citep{BB_formula}, the Wasserstein-2 distance can be written as an optimal transport problem:
\[
\frac{W_2(\rho_0, q)^2}{2h} = \inf_{\rho,\,v} \int_0^h \int_{\R^d} \frac{1}{2} \|v(t,x)\|_2^2\, \rho(t,x)\, dx\, dt\, ,
\]
where the minimization is taken over vector fields \(v : [0,h] \times \R^d \to \R^d\) and densities \(\rho : [0,h] \times \R^d \to \R\) satisfying
\[
\frac{\partial \rho}{\partial t} + \nabla \cdot (\rho v) = 0\,,\qquad \rho(0,x) = \rho_0(x)\,,\qquad \rho(h,x) = q(x)\,.
\]

Solving \eqref{Wass_linear_energy} directly is typically a challenging variational problem. In \citep{kernel_proximal}, motivated by Schr\"odinger bridge systems, a regularized Wasserstein proximal operator is introduced by adding a Laplacian term to the dynamics, leading to
\begin{equation}
\label{Reg_Was}
   \K_V^h\rho_0 := \argmin_{q} \inf_{v,\rho} \int_0^h \int_{\R^d} \frac{1}{2} \|v(t,x)\|_2^2\, \rho(t,x)\, dx\, dt + \int_{\R^d} V(x)\, q(x)\, dx ,
\end{equation}
subject to the modified continuity equation
\begin{equation}
   \frac{\partial \rho}{\partial t} + \nabla \cdot (\rho v) = \beta^{-1}\Delta \rho,\qquad \rho(0,x)=\rho_0(x),\qquad \rho(h,x)=q(x).
\end{equation}
We remark that this regularization is the key step: it enables a closed-form solution via the Hopf-Cole transformation, leading to the kernel formula that forms the basis of the BRWP scheme.

Introducing a Lagrange multiplier function $\Phi$, we find that evaluating $\K_V^h\rho_0$ is equivalent to computing the solution of the coupled PDEs
\begin{align}
\label{regu_PDE}
\begin{cases}
    &\partial_t \rho + \nabla_x\cdot (\rho \nabla_x\Phi) = \beta^{-1} \Delta_x\rho\,, \\
    &\partial_t \Phi + \frac{1}{2}\|\nabla_x\Phi\|_2^2 = -\beta^{-1} \Delta_x\Phi\,, \\
    &\rho(0,x) = \rho_0(x)\,, \quad \Phi(h,x) = -V(x)\,.
\end{cases}
\end{align}
Comparing the first equation in \eqref{regu_PDE} with the Fokker--Planck equation defined in \eqref{eqn_FP}, we observe that the solution to the regularized Wasserstein proximal operator \eqref{Reg_Was} approximates the terminal density \(\rho\) when \(h\) is small. In other words, by solving the coupled PDEs \eqref{regu_PDE}, we can approximate the evolution governed by the Fokker--Planck equation. This approximation will be justified rigorously in Section \ref{sec_kernel_approx}.

The primary motivation for considering the regularized Wasserstein proximal operator as an approximate solution lies in the fact that the coupled PDEs \eqref{regu_PDE} can be solved using a Hopf-Cole type transformation \citep[Sec.~4.4]{evans2022partial} by taking $\eta = \exp(\beta/2\Phi)$ and $\hat\eta = \rho/\eta$. Then \eqref{regu_PDE} is equivalent to a system of backward-forward heat equations
\begin{equation}
\label{Heat_rho}
    \rho(t,x) = \eta(t,x)\hat{\eta}(t,x)\,,
\end{equation}
where $\hat{\eta}$ and $\eta$ satisfy
\begin{align}
&\partial_t\hat{\eta}(t,x) = \beta^{-1} \Delta_x\hat{\eta}(t,x)\,,  \notag\\
&\partial_t\eta(t,x) = -\beta^{-1} \Delta_x\eta(t,x)\,, \label{forward_heat}\\
&\eta(0,x)\hat{\eta}(0,x) = \rho_0(x), \quad \eta(h,x) = \exp\left(-\beta/2V(x)\right)\notag\,.
\end{align}

The heat equations in \eqref{forward_heat}, together with the coupling relation \eqref{Heat_rho}, can be solved using the heat kernel
\[
G_h(x-y) = \frac{1}{(4\pi h/\beta)^{d/2}}\exp\Big(-\beta\frac{\|x-y\|_2^2}{4h}\Big)\,,
\]
corresponding to a diffusion process with thermal diffusivity~ {$\beta^{-1}$}. For instance, the  {backward }component can be written as
 \[
{\eta(t,x) = \int_{\R^d} G_{h-t}(x-y)e^{-\frac{\beta}{2}V(y)}dy,
}\]
and an analogous expression holds for $\hat{\eta}$. Through the coupling relation in~\eqref{Heat_rho}, the final solution to the regularized Wasserstein proximal operator~\eqref{Reg_Was} can then be obtained in closed form:
 \[
{\begin{cases}
    \rho(t, x) = \left(G_{h - t} * \exp(-\frac{\beta V}{2})\right)(x) \cdot \left(G_{t} * \frac{\rho_0}{G_h * \exp(-\frac{\beta V}{2})}\right)(x)\,, \\
    v(t, x) = 2 \beta^{-1} \nabla \log\left(G_{h - t} * \exp(-\frac{\beta V}{2})\right)(x)\,, \\
    \rho(h,x) = \exp\left(-\frac{\beta V(x)}{2}\right)   \cdot \left(G_h * \frac{\rho_0}{G_h * \exp(-\frac{\beta V}{2})}\right)(x)\,.
\end{cases}
}\]
In particular, we focus on the terminal density \(\rho(h,x)\), which serves as an approximation to the solution of the Fokker--Planck equation at time \(h\) with initial density \(\rho_0\). For this purpose, we write down the explicit kernel formula of \(\rho(h,x)\), in which the common Gaussian prefactor cancels between the numerator and denominator:
\begin{equation}
\label{rho_T_BRWP}
 \K_V^h(\rho_0)(x): =\rho(h,x)=  \int_{\R^d} \frac{\exp\big[-\frac{\beta}{2} \big(V(x) + \frac{\|x - y\|_2^2}{2h}\big)\big]}{\int_{\R^d} \exp\big[-\frac{\beta}{2} \big(V(z) + \frac{\|z - y\|_2^2}{2h}\big)\big] dz} \rho_0(y) dy\,.
\end{equation}

\subsection{BRWP algorithm}
\label{sec_BRWP_algo}
With the aid of \eqref{rho_T_BRWP}, we approximate $\rho_{\rm next}$ in \eqref{Euler_backward} by $\widetilde\rho_{k+1}=\K_V^h\rho_k$. This approximation leads to the BRWP sampling algorithm, which is derived from a semi-implicit discretization of the probability flow ODE and whose  {score density admits an explicit integral representation}.

\begin{algorithm}[H]
\caption{BRWP for Sampling from $\rho^*= \frac{1}{Z}\exp(-\beta V)$} \label{algo_BRWP}
\begin{algorithmic}[1]
\State \textbf{Input:} Initial particle set $\{x_{0,j}\}_{j=1}^N$

 \For{$k = 0, 1, 2, \dots$}
    \State  {Use }$\rho_k$ as the  {current density; for the particle implementation, use the empirical measure. }\For{each particle $j = 1, 2, \dots, N$}
        \State Update particle:
        \[
        x_{k+1,j} = x_{k,j} - h\left(\nabla V(x_{k,j}) + \beta^{-1} \nabla \log \widetilde\rho_{k+1}(x_{k,j})\right)
        \]
        where $\widetilde\rho_{k+1} = \K_V^h(\rho_k)$, defined in \eqref{rho_T_BRWP}, {is the score density}.
    \EndFor
\EndFor
\end{algorithmic}
\end{algorithm}

{In the analysis below, $\rho_k$ denotes the density transported by the BRWP update using the score obtained from the exact kernel formula. Section~\ref{sec_score_implementation} discusses numerical evaluation of this score.}

{We use separate hypotheses for the approximation estimates and the convergence analysis. Assumptions~\ref{assump:growth}--\ref{assump:moment} give initial-data and growth conditions for the auxiliary estimates. Assumption~\ref{assump:kernel-hypotheses} specifies the density class used in Section~\ref{sec_kernel_approx}. The convergence theorem uses the Hessian bounds~\eqref{eq:strong-convexity}, finite initial KL divergence, and Assumption~\ref{assump:SO-weak} along the BRWP iteration.
}\begin{assumption}
\label{assump:Regularity}
\textbf{(Regularity)} The potential \(V\in C_{\mathrm{loc}}^{4,1}(\R^d)\) is \(\alpha\)–strongly convex and $L$-smooth; that is,
 \begin{equation}{\label{eq:strong-convexity}
\alpha I_d \preceq \nabla^2 V(x) \preceq L I_d \quad \text{for some } \alpha, L >0\,,\quad \forall x \in \R^d\,. }\end{equation}
\end{assumption}

\begin{assumption}
\label{assump:growth}
\textbf{(Polynomial growth)} There exist constants \(a_1,a_2>0\) and \(m\ge1\) such that
\[
V(x)\ge a_1\|x\|^m - a_2.
\]
Moreover, there exist constants \(C_V,q_V\ge0\) such that, for all multi-indices \(\lambda\) with \(|\lambda|\le3\),
\[
|\partial^\lambda V(x)| \le C_V (1+\|x\|)^{q_V}.
\]
\end{assumption}

\begin{assumption}
\label{assump:moment}
\textbf{(Finite  moment and regularity of initial density)} The initial density $\rho_0 \in \mathcal P_2(\R^d)$ has a finite $p$th moment for some $p>d+q_V+3$, belongs to $C^{3,1}(U)$, and is strictly positive on a neighborhood of a bounded set $U$ (see Section~\ref{sec_kernel_approx} for the precise specification of $U$).
\end{assumption}

Assumption~\ref{assump:Regularity}  {provides the smoothness of \(V\) required for our local weak error analysis and implies the }log-Sobolev inequality~ {\eqref{PL_ineq}. Additional derivative assumptions are specified with the corresponding results. Assumption~\ref{assump:growth} ensures that $V$ and its derivatives exhibit at most polynomial growth. Its lower-growth bound, with $m=2$, and its derivative bounds through order two already follow from~\eqref{eq:strong-convexity}; the convergence theorem uses the bounded third derivative in Assumption~\ref{assump:SO-weak}}.

Assumption~ {\ref{assump:moment} describes a class of initial data with polynomial moments; its exponent $p>d+q_V+3$ is not required by the convergence theorem. }Lemma~\ref{lem:moment-p-discrete}  {states the moment actually needed for its recurrence. The kernel }is well defined  {for probability inputs; the separate Gaussian result supplies the uniform regularity of BRWP iterates used in the corresponding results. The reported experiments in }Section~\ref{sec_NE}  {also include targets beyond the strongly convex setting. }

{Table~\ref{table_complexity} compares sufficient iteration bounds and stepsize restrictions with fixed $\beta=1$.  For the underdamped and proximal rows, assume also $V\in C^3$ and that $\nabla^2V$ is globally $K$-Lipschitz in operator norm. This implies a Frobenius Lipschitz bound $\sqrt d K$ for the underdamped reference. Use the Gaussian initialization prescribed in each cited theorem, after translating the minimizer to zero. $\widetilde O$ suppresses the logarithmic factors from those theorems; the dependence on fixed problem data and the initial KL divergence in the BRWP row is as in Theorem~2. The constant in the BRWP entropy estimate may depend on dimension and regularity. In the BRWP row, $h_\varepsilon$ is the threshold for the entropy comparison with the kernel formula in Lemma~\ref{lem:analytic-defect}. Each accuracy-dependent choice must satisfy the threshold of the bound used. }

 \begin{table}[ht!]\footnotesize
 \centering
\renewcommand{\arraystretch}{2} \setlength{\tabcolsep}{4pt}
\begin{tabular}{|c|c|c|}
\hline
\textbf{Algorithm} &  \makecell{\textbf{Number of iterations required}\\ \textbf{to achieve} $\KL(\rho_k\|\rho^*)\le\delta$} &  \textbf{Sufficient stepsize} \\ \hline
 BRWP (this paper) &  {\color{black}$\mathcal O\!\left(\frac{\log((1+D_0)/\delta)}{\alpha\delta^{1/2}}\right)$ }&  {$h\le\min\{h_\varepsilon,\alpha^{-1}\}$} \\ \hline
 ULA \citep{ULA_convergence} &  {$\widetilde{\mathcal O}\!\left(\frac{dL^2}{\alpha^2\delta}\right)$ } & $\frac{\alpha}{4L^2}$ \\ \hline
 \makecell{Underdamped Langevin\\ \citep{Nesterov_MCMC}} &  {$\widetilde{\mathcal O}\!\left(\frac{d^{1/2}(L^{3/2}+d^{1/2}K)}{\min\{\alpha,1\}^2\delta^{1/2}}\right)$ } &  \makecell{{Accuracy-dependent choice}\\ {in the cited Theorem~1}} \\ \hline
 \makecell{Proximal Langevin\\ \citep{wibisono2019proximal}} &  {$\widetilde{\mathcal O}\!\left(\frac{d^{1/2}(L^{3/2}+dK)}{\alpha^{3/2}\delta^{1/2}}\right)$ } &  {$\min\!\left\{\frac{1}{8L},\frac{1}K,\frac{3\alpha}{32L^2}\right\}$ } \\ \hline
\end{tabular}
\caption{Iteration  {bounds }for sampling algorithms under the hypotheses stated in the text. The BRWP row concerns the BRWP iteration under Assumption~\ref{assump:SO-weak}. Here $D_0=\KL(\rho_0\|\rho^*)$. The last column gives sufficient stepsize restrictions. {These iteration counts do not include the cost of numerical score evaluation.}}\label{table_complexity}
\end{table}

 {The ULA row uses Theorem~2 of \citet{ULA_convergence}, with $h=\alpha\delta/(16dL^2)$ for sufficiently small $\delta$. The proximal row uses Theorem~1 and Corollary~1 of \citet{wibisono2019proximal}, with $h=c\sqrt{\alpha\delta/[d(L^3+d^2K^2)]}$ for a sufficiently small constant $c$. The underdamped row uses the algorithm, momentum initialization, and parameter choices of Theorem~1 of \citet{Nesterov_MCMC}. The BRWP row uses $h=c_*\sqrt\delta$, with $c_*>0$ sufficiently small, subject to the stated stepsize bound and uniform hypotheses through the requested horizon. In the proximal stepsize restriction, $1/K=+\infty$ when $K=0$. }

In the next section, we shall focus on investigating the approximate density given by the kernel formula \eqref{rho_T_BRWP} and the regularity properties of the discrete density from the BRWP iteration under the stated hypotheses.

\section{Kernel Approximation and Regularity of the BRWP Update}
\label{sec_kernel_approx}
{We study the accuracy of the kernel formula used to evaluate the score in the BRWP update. Under Assumption~\ref{assump:kernel-hypotheses}, we first estimate its one-step density approximation and the effect of errors in the input density. We also verify preservation of the density class $\mathcal{A}_{m,M}$ for Gaussian evolution. Finally, we use the local density estimate to bound the score error.}

{The density and score estimates describe the one-step approximation supplied by the kernel formula. To study convergence of the BRWP iteration in KL divergence, Section~\ref{sec_conv_KL} uses additional global bounds on the score remainder and spatial derivatives, stated in Assumption~\ref{assump:SO-weak}.}

We begin by introducing the notation and operators used throughout this section. For a smooth test function \(\varphi:\R^d\to\R\), the Langevin generator and its adjoint are
\begin{equation}
\label{def:L}
(\mathcal L\varphi)(x) =\beta^{-1}\Delta\varphi(x)-\nabla V(x)\!\cdot\!\nabla\varphi(x)\,, \qquad (\mathcal L^{\!*}\rho)(x) =\beta^{-1}\Delta\rho(x)+\nabla\!\cdot(\rho\nabla V(x))\,.
\end{equation}

Define
\[
\psi(x):=e^{-\frac{\beta}{2}V(x)}\,, \qquad \rho^*(x) = \frac{1}{Z}\psi(x)^2\,.
\]
The introduction of \(\psi\) is motivated by the weight that appears in \(\K_V^h\) in \eqref{rho_T_BRWP}. Let \(\G_h\) denote the Gaussian heat semigroup,
\[
(\G_h f)(y)=\int_{\R^d}G_h(x-y)f(x)\,dx\,,
\]
and write  {$q_h:=\G_h\psi$}. Then, for test functions $\varphi$ and densities $\rho$, the dual kernel operator takes the form
 \begin{equation}{\label{eq:Doob}
P_V^h\varphi(y) :=\frac{\G_h(\psi\varphi)(y)}{q_h(y)}\,, \qquad \langle\varphi,\K_V^h\rho\rangle=\langle P_V^h\varphi,\rho\rangle\,, }\end{equation}
where \(\langle u,v\rangle=\int_{\R^d}u(x)v(x)\,dx\,.\) With the above notation, the kernel formula can be written as
 \begin{equation}{\label{eq:KVT-kernel}
\K_V^h f(x) =\psi(x)\int_{\R^d}\frac{G_h(x-y)f(y)}{q_h(y)}\,dy\,. }\end{equation}

Equivalently, \(P_V^h\) admits the Gibbs-type representation
\begin{equation}\label{eq:PVT-prob}
P_V^h\varphi(y)=\E_{\mu_y}[\varphi(X)]\,,
\end{equation}
where
\[
U_y(x):=\tfrac{\beta}{2}V(x)+\tfrac{\beta}{4h}\|x-y\|^2\,, \qquad \mu_y(dx)=\frac{e^{-U_y(x)}}{\int_{\R^d}e^{-U_y(x)}\,dx}\,dx\,.
\]

We will compare three families of densities, all starting from the same initial data \(\rho_0\):
\begin{itemize}
\item \(\rho(t_k)\): the exact Fokker--Planck solution at time \(t_k\)\,;
\item \(\widehat\rho_k\): the kernel iterates  \(\widehat\rho_{k+1} = \K_V^h\widehat\rho_k\)\,;
\item \(\rho_k\): the  {BRWP iterates corresponding to }Algorithm~\ref{algo_BRWP}{, written }as
  \[
  \rho_{k+1}=(F_k)_\#\rho_k\,, \qquad F_k(x)=x-h\big(\nabla V(x)+\beta^{-1}\nabla\log(\K_V^h\rho_k)(x)\big)\,.
  \]
\end{itemize}

For a finite signed measure \(\mu\) supported in a bounded set \(U\subset\R^d\), define
\begin{align}\label{def:dual-C11}
\|\mu\|_{(C^{1,1}(U))^*} &:=\sup_{\|\varphi\|_{C^{1,1}(U)}\le1}|\langle\varphi,\mu\rangle|\,, & \|\varphi\|_{C^{1,1}(U)} &:=\|\varphi\|_{L^{\infty}(U)} +\|\nabla\varphi\|_{L^{\infty}(U)}
   +[\nabla\varphi]_{\mathrm{Lip}(U)}\,, \\[2mm]
\|\mu\|_{(C_c^{1,1}(U))^*} &:=\sup_{\|\varphi\|_{C_c^{1,1}(U)}\le1}|\langle\varphi,\mu\rangle|\,, & \|\varphi\|_{C_c^{1,1}(U)} &:=\|\nabla\varphi\|_{L^{\infty}(U)}+[\nabla\varphi]_{\mathrm{Lip}(U)}\,.
\end{align}
Higher-order norms \(\|\cdot\|_{C^{m,1}(U)}\) are defined analogously.  The dual norm above appears naturally in the weak formulation that we will work with in this section. Moreover, we write the weighted $p$th moment as
 \[
{M_p(\mu):=\int_{\R^d}(1+\|y-x_*\|^p)\,d|\mu|(y),\qquad x_*:=\argmin V\,.
}\]

For the following analysis,  {fix bounded balls
}\[
{U\Subset U_0\Subset U_r\Subset\R^d, \qquad r_0:=\dist(U_0,\partial U_r)>0.
}\]
The {notation $C_c^{m,1}(U)$ denotes the space of compactly supported functions on $U$ whose zero extensions belong to $C^{m,1}(\R^d)$. In the dual norms defined using compactly supported test functions, the supremum is over such functions. All input measures are finite; $d|\mu|$ denotes total variation. Constants }may depend on  {the displayed norms, the domains and their separation, $d,\beta,\alpha,L$, and the stated bounds}, but are uniform in  {$h$ and $k$ only when those data are }uniform. 

\begin{kernelassumption}[Potential and density bounds]
\label{assump:kernel-hypotheses}
\blue{Let $V$ satisfy Assumption~\ref{assump:Regularity} with a bounded third derivative:}
\[
\blue{V\in C_{\rm loc}^{4,1}(\R^d),\qquad \alpha I_d\preceq D^2V\preceq LI_d,\qquad \|D^3V\|_\infty<\infty.}
\]
\blue{Fix $\beta,\bar h>0$, the balls $U\Subset U_0\Subset U_r$, and constants $m>0$, $M<\infty$. We assume all input densities belong to the class}
\[
\blue{\mathcal{A}_{m,M}:=\left\{u\text{ a probability density}:\quad
\begin{gathered}
u|_{U_r}\in C^{3,1}(U_r),\quad \inf_{U_r}u\ge m,\quad
\|u\|_{C^{3,1}(U_r)}+M_4(u)\le M
\end{gathered}\right\}.}
\]
\blue{For input densities with index $k$, the same $m,M$ apply for $0<h\le\bar h$ and $0\le k<N(h)$ on the specified horizon. We call a pair $\rho_k,\tilde\rho_k\in\mathcal{A}_{m,M}$ admissible if $\tilde\rho_k-\rho_k\in C_c^{2,1}(U_0)$.}
\end{kernelassumption}

{Constants are uniform over $\mathcal{A}_{m,M}$ and sufficiently small $0<h\le h_0\le\bar h$; they may depend on $m,M$, the fixed domains, $d,\beta,\alpha,L$, $\|D^3V\|_\infty$, and the local potential norms.}

{Denote the BRWP density update by $\cR_h$:}
\begin{equation}{\label{eq:raw-map}
\cR_h\rho=\left(I-h\nabla\!\left(V+\beta^{-1}\log\K_V^h\rho\right)\right)_\#\rho. }
\end{equation}
The detailed proofs of results in this section are given in Appendix~\ref{sec_appendix_kernel}.

\subsection{One-step weak kernel approximation of the Fokker--Planck flow}
We first compare the density obtained from the kernel formula with the first-order expansion of the Fokker--Planck evolution on a compact set. The remainder is of order $h^2$ in one step. The proof is based on successive Taylor expansions of smooth functions, combined with integration by parts and moment estimates for the Gaussian kernel.

\begin{lemma}[Local kernel expansion]
\label{lem:PV-weak-local}
Under \blue{Assumption~\ref{assump:kernel-hypotheses}, let $u\in\mathcal{A}_{m,M}$ and $\varphi\in C_c^{2,1}(U)$. For all sufficiently small $0<h\le h_0$},
 \begin{equation}{\label{eq:local-PV-weak}
\big|\langle \varphi,\;\K_V^hu-u-h\mathcal{L}^*u\rangle\big|=\big|\langle P_V^h\varphi-\varphi-h\mathcal L\varphi,\;u\rangle\big| \le C\,h^2\,\|\varphi\|_{C^{2,1}(U)}\,(\|u\|_{C^{0,1}(\blue{U_{r}})}\mathbin{\blue{+}}\blue{1}) \,, }\end{equation}
\blue{Moreover, the kernel density satisfies}
\[
\blue{\|\K_V^hu-u-h\mathcal L^*u\|_{C^0(U)}\le Ch^2,\qquad \inf_U\K_V^hu\ge m/2,\qquad \|\K_V^hu\|_{C^2(U)}\le C.}
\]
{The constants are uniform over $\mathcal{A}_{m,M}$. }
\end{lemma}
\begin{proof}[Sketch of proof]
{For the weak estimate, write $\psi=e^{-\beta V/2}$ and $q_h=\G_h\psi$. For expressions depending on $x$ and $y$, write}
\[
\G_h[F(x,y)]:=\int_{\R^d}G_h(y)F(x,y)\,dy,
\]
{with $x$ held fixed. The dual identity is}
\begin{equation}
{P_V^h\varphi(x) =\varphi(x) +\frac{N(x)}{q_h(x)} \,, \qquad N(x):=\G_h\!\big[\psi(x+y)\,(\varphi(x+y)-\varphi(x))\big],\quad q_h(x):=\G_h[\psi(x+y)]\,. }\end{equation}
{Taylor expansion, Gaussian integration by parts, and localization then give~\eqref{eq:local-PV-weak}. The full weak calculation and the normalized-kernel tail estimate are given in Appendix~\ref{sec_appendix_kernel}.}

{For the local density bound, choose a smooth cutoff equal to one near $\overline U$ and supported in $U_r$. The normalized-kernel tail estimate controls the complementary input by its unit mass. Localized heat expansion gives}
\[
\blue{\K_V^hu =u+\frac h\beta\left\{\psi\Delta(u/\psi)-u\frac{\Delta\psi}{\psi}\right\} +O_{C^0(U)}(h^2) =u+h\mathcal L^*u+O_{C^0(U)}(h^2).}
\]
{The $C^{3,1}$ bounds supply bounded fourth weak derivatives; mollification justifies the heat remainder uniformly, as detailed in the appendix. Since $u\ge m$ and $\mathcal L^*u$ is uniformly bounded on $U$, the kernel density is at least $m/2$ for small $h$. Localized heat contraction and the differentiated tail bound give its uniform $C^2(U)$ bound.}
\end{proof}

{The local density bound in Lemma~\ref{lem:PV-weak-local} also gives}
\[
|\langle\varphi,\K_V^hu-u-h\mathcal L^*u\rangle| \le Ch^2\|\varphi\|_{L^1(U)}
\]
{for test functions supported in $U$, including those in $C^{1,1}$ and $C^{0,1}$.}

{We next estimate the weak approximation error for test functions defined on all of $\R^d$. The bound uses a global norm of the test function and a weighted moment of the input density. Conditional Gaussian moments and Stein identities control the integral over the full input support.}

 \begin{theorem}[Global weak $O(h^2)$ expansion]
\label{thm:PV-weak}
 \blue{Under Assumption~\ref{assump:kernel-hypotheses}, let $u\in\mathcal{A}_{m,M}$. For every $\varphi\in C_b^4(\R^d)$ and all sufficiently small $0<h\le h_0$,
}\begin{align}
 \blue{\left|\langle\varphi,\K_V^hu-u-h\cL^*u\rangle\right| }&\blue{=\left|\langle P_V^h\varphi-\varphi-h\cL\varphi,u\rangle\right| \le Ch^2\norm\varphi_{C_b^4(\R^d)}M_4(u), \qquad 0<h\le h_0.\label{eq:global-PV-weak-fixedU-final}
}\end{align}
\blue{Here $C,h_0$ depend only on $d,\beta,\alpha,L$ and $\norm{D^3V}_\infty$. The divergence is understood distributionally. Since $M_4(u)\le M$, the remainder is uniform over $\mathcal{A}_{m,M}$. }\end{theorem}

We next quantify how an error in the input density affects the kernel approximation. The next result provides control over both the local truncation error and the error inherited from previous steps.

 \begin{theorem}[\blue{Weak kernel approximation with input error}]
\label{thm:oracle-C01-test}
Consider the kernel formula iteration $\widehat\rho_{k+1} = \K_V^h\widehat\rho_k$. \blue{Under Assumption~\ref{assump:kernel-hypotheses}, let $\rho(t_k),\widehat\rho_k\in\mathcal{A}_{m,M}$ and $\varphi\in C_b^4(\R^d)$. Then, for sufficiently small $h$,
}\begin{align}\blue{\label{eq:oracle-with-propagation-C01}
\Big| \big\langle \varphi,\, \widehat\rho_{k+1}-\rho(t_k)-h\mathcal{L}^{*}\rho(t_k) \big\rangle \Big|\le }&\blue{C\,h^2\|\varphi\|_{C_b^4(\R^d)}M_4(\rho(t_k)) + \frac{1}{1+\alpha h}W_1(\widehat\rho_k,\rho(t_k))\operatorname{Lip}(\varphi)\,. }\end{align}
If the kernel formula is evaluated at the exact input $\rho(t_k)$, the input-error term in~\eqref{eq:oracle-with-propagation-C01} vanishes, giving
 \[
\blue{\big|\langle \varphi,\, \K_V^h(\rho(t_k))-\rho(t_k)-h\mathcal{L}^{*}\rho(t_k) \rangle\big| \le C\,h^2\, \|\varphi\|_{C_b^4(\R^d)}M_4(\rho(t_k))\,.
}\]
\end{theorem}

\begin{proof}
Add and subtract $\K_V^h\rho(t_k)$ and apply duality:
\[
\begin{aligned}
\langle\varphi,\, \widehat\rho_{k+1}-\rho(t_k)-h\mathcal L^{\!*}\rho(t_k)\rangle &= \underbrace{\langle\varphi,\, \K_V^h\rho(t_k)-\rho(t_k)-h\mathcal L^{\!*}\rho(t_k)\rangle}_{\text{local truncation}} + \underbrace{\langle P_V^h\varphi,\, \widehat\rho_k-\rho(t_k)\rangle}_{\text{input error}}\,.
\end{aligned}
\]

By Theorem~\ref{thm:PV-weak}, the local truncation term is bounded by  {$Ch^2\|\varphi\|_{C_b^4}M_4(\rho(t_k))$. }

For the propagated error, Lemma~\ref{lem:PV-regularization} in Appendix~\ref{subsec:stability}  {gives contraction of the global Lipschitz seminorm:
}\[
{\operatorname{Lip}(P_V^h\varphi) \le(1+\alpha h)^{-1}\operatorname{Lip}(\varphi).
}\]
{Since the two laws have equal mass, Kantorovich duality bounds the pairing of their difference with $P_V^h\varphi$ by the product of this seminorm and $W_1(\widehat\rho_k,\rho(t_k))$. }Substituting this into the dual pairing gives~\eqref{eq:oracle-with-propagation-C01}.
\end{proof}

{The preceding results are uniform once the input densities remain in $\mathcal{A}_{m,M}$. We next give a setting in which this requirement can be verified from the initial law.}

\subsection[Gaussian verification]{Gaussian evolution in $\mathcal{A}_{m,M}$}
\label{sec_C31BRWP}

{For Gaussian initial data and a strongly convex quadratic potential, the exact flow, the kernel formula iteration, and the BRWP iteration all remain Gaussian. The following theorem uses uniform bounds on their means and covariance eigenvalues to establish the regularity, positivity, and fourth moments required by Assumption~\ref{assump:kernel-hypotheses}.}

 \oldresult{7}
\begin{theorem}[Gaussian verification and uniform regularity]
\label{thm:uniform-C31}
 \blue{Let
}\[
\blue{V(x)=\tfrac12(x-x_*)^{\mathsf T}A(x-x_*)+c, \quad \alpha I\preceq A\preceq LI, \quad \rho_0=\mathcal N(m_0,C_0),\quad C_0\succ0.
}\]
\blue{Then there are $h_0>0$ and $C_*<\infty$, depending on the initial Gaussian parameters and the fixed potential and domain, such that
}\begin{equation}\blue{\label{eq:C31-uniform}
\sup_{0<h\le h_0}\sup_{k\ge0} \left(\|\rho_k\|_{C^{3,1}(U)}+\|\widehat\rho_k\|_{C^{3,1}(U)}\right)\le C_*. }\end{equation}
\blue{Moreover, there are $m>0$, $M<\infty$, independent of $k,h$, such that
}\[
\blue{\rho_k,\widehat\rho_k,\rho(t_k)\in\mathcal{A}_{m,M} \qquad(k\ge0,\ 0<h\le h_0).}
\]
\blue{Thus Assumption~\ref{assump:kernel-hypotheses} holds for these inputs at all iterations.} 
\end{theorem}
\begin{proof}[Sketch of proof]
\blue{The kernel sends $\mathcal N(m,C)$ to $\mathcal N(Jm,Q)$ after translation by $x_*$, where $J=(I+hA)^{-1}$ and $Q=JCJ+2h\beta^{-1}J$. The BRWP update is affine on a Gaussian input. Proposition~\ref{prop:gaussian} keeps its means and covariance eigenvalues in a fixed compact range, uniformly in small $h$ and all $k$. The kernel chain and exact Gaussian flow have the same type of compact parameter bounds. Differentiation gives uniform local regularity. Continuity and strict positivity on the compact parameter set give a uniform positive lower bound on $U_r$, and Gaussian moment formulas give a uniform fourth moment. The appendix supplies the details. }\end{proof}

{The constants depend on the Gaussian initial state. Theorem~\ref{thm:uniform-C31} provides inputs in $\mathcal{A}_{m,M}$ for all time; the approximation results also apply to non-Gaussian inputs in this class on the stated horizon.}

\subsection{Weak error of the kernel-based score oracle}

{We now estimate the error in the score obtained from the kernel formula, using the local density bound in Lemma~\ref{lem:PV-weak-local}. Here $\tilde\rho_k$ approximates the input density $\rho_k$, and $\K_V^h\tilde\rho_k$ is the density used to evaluate the score. For an admissible pair, the error consists of the approximation error for $\rho_k$ and the effect of perturbing that input. The expression
}\[
{\nabla\log\rho_k+h\nabla\!\left(\frac{\mathcal L^*\rho_k}{\rho_k}\right)
}\]
{is the first-order expansion of the score obtained from the kernel formula}. 

\begin{theorem}[Weak score-oracle estimate]\label{thm:weak-score-drift}
 \blue{Under Assumption~\ref{assump:kernel-hypotheses}, let $\rho_k,\tilde\rho_k\in\mathcal{A}_{m,M}$ be an admissible input pair. For all sufficiently small $h$ and every $w\in C_c^{1,1}(U;\R^d)$,
}\begin{equation}{\label{eq:weak-score-drift}
\Bigg| \int_{U}w \cdot\Bigl[\nabla\log(\K_V^h\tilde\rho_k)-\nabla\log\rho_k-h\,\nabla\!\Bigl(\tfrac{\mathcal L^*\rho_k}{\rho_k}\Bigr)\Bigr]\rho_k\,dx\Bigg| \le C\Bigl(h^2+\|\tilde\rho_k-\rho_k\|_{C^{2,1}(\blue{U_{0}})}\Bigr)\|w\|_{C^{1,1}(U)}\,,}
\end{equation}
\blue{where $C$ is uniform over admissible pairs in $\mathcal{A}_{m,M}$.}
\end{theorem}

\begin{proof}[Sketch of proof]
Fix $k$ and an admissible input pair $(\rho_k,\tilde\rho_k)$, and write $r:=\rho_k$,  $r_K:=\K_V^h\rho_k$, and $r^{(1)}:=r+h\,\mathcal L^* r$. We decompose the error in approximating the score function into two parts:
\begin{align*}
&\int w\cdot\bigl[\nabla\log(\K_V^h\tilde\rho_k) -\nabla\log r - h\nabla(\tfrac{\mathcal L^*r}{r})\bigr] r dx \\= &\int w\!\cdot\!\Bigl[\nabla\log r_K-\nabla\log r-h\nabla\!\Bigl(\tfrac{\mathcal L^*r}{r}\Bigr)\Bigr]r dx +\int w\!\cdot\bigl[\nabla\log(\K_V^h\tilde\rho_k)-\nabla\log(\K_V^h\rho_k)\bigr]\rho_k dx
\end{align*}
Lemma~\ref{lem:PV-weak-local} supplies uniform positivity of the kernel densities. We denote the first and second terms above by $I^{0}$ and $I^{1}$.

For $I^{0}$, we insert and subtract $\nabla\log r^{(1)}$ and exploit positivity and  {$C^{1,1}$ regularity of the three }densities to rewrite $I^{0}$ as integrals against $r_K-r^{(1)}$.  The coefficients appearing in these integrals lie in $C^{0,1}$ with norms controlled by $\|w\|_{C^{1,1}(U)}$. Applying {the $C^0$ remainder estimate in Lemma~\ref{lem:PV-weak-local}} yields $I^{0}=O(h^{2}\|w\|_{C^{1,1}(U)})$.

For $I^{1}$,  {heat-semigroup contraction applied to the supported perturbation divided by $q_h$ implies \( \|\K_V^h\tilde\rho_k-\K_V^h\rho_k\|_{C^1(U)} \leq C\|\tilde\rho_k-\rho_k\|_{C^{2,1}(U_0)}. \) Since $f\mapsto\nabla\log f$ is Lipschitz as a map from $C^1$ to $C^0$ on sets with a fixed positive lower bound and a fixed $C^1$ upper bound}, we obtain
 \[
{|I^{1}| \leq C\|w\|_{C^{1,1}(U)}\, \|\tilde\rho_k-\rho_k\|_{C^{2,1}(U_{0})}.
}\]

Combining the bounds on $I^{0}$ and $I^{1}$ gives the weak score estimate with its input-perturbation term.
\end{proof}

{For $\rho_k = \tilde{\rho}_k$, the perturbation term vanishes. Theorem~\ref{thm:weak-score-drift} bounds the remainder after integration against compactly supported test functions on $U$. To estimate the KL divergence in Section~\ref{sec_conv_KL}, we need a bound after integration against the relative score $\nabla\log(\rho_k/\rho^*)$ on all of $\R^d$. Assumption~\ref{assump:SO-weak} states this bound and the derivative estimates used in the entropy calculation.}

\section{Convergence Analysis of the BRWP Update in KL Divergence}\label{sec_conv_KL}

{We now study the convergence of the BRWP update $\rho_{k+1}=\cR_h\rho_k$ in KL divergence under Assumption~\ref{assump:SO-weak}, with $V\in C^2$ satisfying~\eqref{eq:strong-convexity}. Lemmas~\ref{Lemma_rho_k_iteration}--\ref{Lemma_KL_decay_pre} give the one-step expansions. Lemma~\ref{lemma_KL_one_step} gives the one-step entropy bounds, and Theorem~\ref{Thm_KL_decay} yields exponential decay of the KL divergence up to an $O(h^2)$ error on the stated iteration horizon.}

\subsection{Assumptions and weak density expansion}
\label{subsec:one-step-expansion}

\begin{assumption}[Relative-score oracle and analytic conditions]\label{assump:SO-weak}
{\blue{Fix $h_0>0$ and consider the family of BRWP iterates $\{\rho_k^{(h)}:0<h\le h_0,\ 0\le k<N(h)\}$. Assume there exist constants $C,M<\infty$, independent of $h$ and $k$, such that the following conditions hold at every state in this family.} At each state, write $\rho_k=\rho_k^{(h)}$, set $\widetilde\rho_{k+1}=\K_V^h\rho_k$ and
}\[
{\Phi_k=\beta^{-1}\log(\rho_k/\rho^*),\qquad \widetilde\Phi_k=\beta^{-1}\log(\widetilde\rho_{k+1}/\rho^*).}
\]
{When the iteration index is omitted, write $\rho=\rho_k$, $\Phi=\Phi_k$, and $\widetilde\Phi=\widetilde\Phi_k$. The remainder in the score expansion is
}\[
\blue{e_{k,h}:=\nabla\log\widetilde\rho_{k+1}-\nabla\log\rho_k -h\nabla\!\left(\frac{\mathcal L^*\rho_k}{\rho_k}\right).}
\]

\noindent\textbf{Relative-score pairing.} {The following Lebesgue integral satisfies}
\begin{equation}\blue{\label{eq:weak-oracle-entropy}
 \left|\int_{\R^d}\nabla\log(\rho_k/\rho^*)\cdot e_{k,h}\,\rho_k\dd x\right| \le Ch^2. }\end{equation}
\noindent\textbf{Analytic conditions.} \blue{The density $\rho_k$ is positive and $C^4$, $V\in C^3$, and
}\begin{align}
 &\blue{D^2\Phi_k\preceq MI_d, \qquad \norm{D^3V}_\infty+\norm{\nabla\widetilde\Phi_k}_{W^{1,3}(\rho_k)}\le M, \qquad\norm{\nabla(\widetilde\Phi_k-\Phi_k)}_{H^1(\rho_k)}\le Mh. \label{eq:strong-taylor-assumption}
}\end{align}
{Here $H^1(\rho)=W^{1,2}(\rho)$ and $\|u\|_{W^{1,p}(\rho)}^p=\int(|u|^p+\|Du\|_F^p)\rho\dd x$. The Hessian inequality is an upper bound.} {The number $N(h)$ counts iterations, so the time interval has length $hN(h)$. An accuracy bound requires these uniform constants through its full time horizon.  }
\end{assumption}

\blue{The relative-score pairing bounds the contribution of the score remainder to the entropy change. The upper Hessian bound also bounds $D^2\widetilde\Phi_k$ for small $h$ (Lemma~\ref{lem:analytic-descent}), ensuring the invertibility used in the entropy calculation. The weighted $W^{1,3}$ bound controls cubic Taylor remainders, while the $H^1$ bound controls the difference between $\nabla\widetilde\Phi$ and $\nabla\Phi$ and between their derivatives.}

\blue{For strongly convex quadratic potentials and nondegenerate Gaussian initial laws, Assumption~\ref{assump:SO-weak} holds uniformly over all iterations at sufficiently small $h$; see Proposition~\ref{prop:gaussian} in Appendix~\ref{sec:gaussian}. For periodic potentials, Proposition~\ref{prop:torus} in Appendix~\ref{sec:torus} verifies the periodic version of the assumption on a uniformly positive $C^4$ family, over any horizon on which these bounds hold. Both verifications also cover the additional condition for the weak density expansion in Lemma~\ref{Lemma_rho_k_iteration}.}

{The weighted gradient bounds \eqref{eq:strong-taylor-assumption} imply a uniform bound on the relative Fisher information defined as $\FI(\rho\|\rho^*)=\int|\nabla\log(\rho/\rho^*)|^2\rho\,dx$.} The $W^{1,3}$ and $H^1$ bounds give
\begin{gather*}
\blue{\sqrt {\FI(\rho\|\rho^*)}\le\beta\bigl(\|\nabla\widetilde\Phi\|_{L^2(\rho)}+Mh\bigr) \le\beta(M+Mh).}
\end{gather*}
{Under Assumption~\ref{assump:SO-weak}, the relative-score pairing and the definition of $e_{k,h}$ give, by the cutoff argument in Lemma~\ref{lem:analytic-descent},
}\[
\blue{\int \frac{|\mathcal L^*\rho|^2}{\rho} =-\frac\beta h\int\nabla\Phi\cdot\nabla(\widetilde\Phi-\Phi)\rho +h^{-1}\int\nabla\Phi\cdot e_{k,h}\rho \le M\sqrt{\FI(\rho\|\rho^*)}+Ch.}
\]
{These estimates ensure that the integrals in the entropy calculation are finite. The logarithmic Sobolev inequality~\eqref{PL_ineq} and the estimate following Lemma~\ref{lem:moment-p-discrete} also give uniform bounds on the KL divergence and second moment. We next record the weak density expansion, with its additional condition on the score remainder.}

\begin{lemma}[Weak expansion of the BRWP density update]
\label{Lemma_rho_k_iteration}
Let $(\rho_k)_{k\ge0}$ be the BRWP iterates. {\blue{Use the notation and horizon of Assumption~\ref{assump:SO-weak}, and assume its analytic conditions. In addition, suppose the following bounded-test pairing is well defined as a Lebesgue integral and satisfies, uniformly in $k$ and $h$,}}
\begin{equation}\blue{\label{eq:weak-oracle-global}
 \left|\int_{\R^d}w\cdot e_{k,h}\,\rho_k\dd x\right| \le Ch^2\norm w_{C_b^{1,1}(\R^d)}, \qquad w\in C_b^{1,1}(\R^d;\R^d). }\end{equation}
Let  \blue{$u\in C_b^3(\R^d)$. Then
}\begin{align}
{\label{eq:weak-rho-k-final} \langle u,\rho_{k+1}-\rho_k\rangle }&{=h\,\Big\langle  u \,,\nabla\cdot\left(\rho_k\,\nabla\!\Bigl(\beta^{-1}\log\!\frac{\rho_k}{\rho^*}\Bigr)\right)\Big\rangle}\\
&{\quad\blue{+ h^2\beta^{-1} \Big\langle u \,,\, \nabla\cdot\left( \rho_k\nabla\left(\tfrac{\mathcal L^{*}\rho_k}{\rho_k} \right)\right)\Big\rangle} +\frac{h^2}{2}\, \big\langle u \,,\, \nabla\cdot(\rho_k\zeta_k)\big\rangle + \blue{\mathcal{O}\!\bigl(h^3\|u\|_{C_b^3(\R^d)}\bigr)} \,,\notag }\end{align}
where  {the constant is uniform under these hypotheses }and
\[
\zeta_k := \nabla\!\cdot\!\bigl((\nabla\Phi_k)(\nabla\Phi_k)^{\!\top}\bigr) + (\nabla\Phi_k\cdot\nabla\log\rho_k)\,\nabla\Phi_k\,.
\]
Here $\nabla\!\cdot\![A]$ denotes the vector field whose $i$th component is $\sum_{j} \partial_{x_j} A_{ij}$.
\end{lemma}
\begin{proof}[Proof sketch]
The full proof is given in Appendix~\ref{sec:app_KL}. We outline the main steps.

 {Apply Taylor's theorem to $u(x-h\nabla\widetilde\Phi(x))$. The integrated third-order remainder is bounded by
}\[{Ch^3\|D^3u\|_\infty\int|\nabla\widetilde\Phi|^3\rho_k.}\] Estimate~\eqref{eq:weak-oracle-global}{, tested against $w=\nabla u$, permits replacing $\nabla\widetilde\Phi$ by $\nabla\Phi+h\beta^{-1}\nabla(\mathcal L^*\rho/\rho)$ in the linear term. The $O(h)$ bound in $L^2(\rho_k)$ supplied by~\eqref{eq:strong-taylor-assumption} permits replacing $\nabla\widetilde\Phi$ by $\nabla\Phi$ in the quadratic term. Integration by parts gives the stated divergence form}.
\end{proof}

 {The global norm of the test function and the moment bound control the Taylor remainder directly. Divergences are distributional: the quadratic pairing acts on $D^2u$ through the integrable tensor $\rho(\nabla\Phi)(\nabla\Phi)^{\mathsf T}$}.

To express the expansion more compactly, introduce the weighted elliptic operator
\begin{equation}
\label{def_Dk}
\mathcal D_k^\beta(u) :=\frac{\beta^{-1}}{\rho_k}\,\nabla\!\cdot(\rho_k\nabla u)\,, \qquad \mathcal D_k^\beta(u)\,\rho_k =\beta^{-1}\nabla\!\cdot(\rho_k\nabla u)\,.
\end{equation}
Using the identity
\[
\frac{\mathcal L^{\!*}\rho_k}{\rho_k} = \mathcal D_k^\beta\left(\logks\right)\,,
\]
we can rewrite the second-order weak expansion \eqref{eq:weak-rho-k-final} as
\begin{equation}
\label{eqn:rho_k_backward}
\rho_{k+1} =\rho_k +h\,\mathcal D_k^\beta\left(\logks\right)\,\rho_k +h^2\,(\mathcal D_k^\beta\!\circ\!\mathcal D_k^\beta)\left(\logks\right)\,\rho_k +\frac{h^2}{2}\,\nabla\!\cdot\!\big(\zeta_k\,\rho_k\big) +\mathcal O(h^3)\,.
\end{equation}
 {Equation~\eqref{eqn:rho_k_backward} is understood after integration against the test functions in Lemma~\ref{Lemma_rho_k_iteration}. We next use the change-of-variables formula for the transport map to expand the KL divergence.}

\subsection{Second-order expansion of the KL divergence}

{Under Assumption~\ref{assump:SO-weak}, the change-of-variables formula gives the first two terms of the change in KL divergence over one BRWP step.}

 \begin{lemma}[\blue{Second-order KL identity}]\label{Lemma_KL_decay_pre}
{\blue{Under Assumption~\ref{assump:SO-weak}, take $h$ sufficiently small.} }Along the update  \blue{\eqref{eq:raw-map} }from $t_k$ to $t_{k+1}=t_k+h$,
 \begin{align}
{\label{eqn:decay_KL_1} \KL(\rho_k\|\rho^*)-}&{\KL(\rho_{k+1}\|\rho^*) = h\,\beta^{-1}\,\FI(\rho_k\|\rho^*) -\frac{3h^2}{2}\int \Big|\mathcal D_k^\beta\!\Big(\log\frac{\rho_k}{\rho^*}\Big)\Big|^2 \rho_k\,dx
}\\[-2pt]
&{\quad-\frac{h^2}{2}\,\beta^{-2} \int \Big\langle \nabla\log\frac{\rho_k}{\rho^*},\,\nabla^2\log\frac{\rho_k}{\rho^*}\,\nabla\log\frac{\rho_k}{\rho^*}\Big\rangle \rho_k\,dx +\blue{O(h^3)}\,.\notag }\end{align}
\blue{The }$O(h^3)$  \blue{remainder is uniform under these hypotheses. }\end{lemma}

\begin{proof}
 {For sufficiently small $h$, $F=I-h\nabla\widetilde\Phi$ is the gradient of a uniformly convex potential and is a global orientation-preserving diffeomorphism, as proved in Appendix~\ref{App_sec_KL}. With $W=-\log\rho^*$, change of variables gives
}\[
{\blue{\KL(F_\#\rho\|\rho^*)-\KL(\rho\|\rho^*) =\int[W(x-h\nabla\widetilde\Phi)-W(x)-\log\det(I-hD^2\widetilde\Phi)]\rho\,dx.}
}\]
{Expand this exact identity and use the stated analytic conditions and~\eqref{eq:weak-oracle-entropy} to control its global remainder. Weighted integration by parts gives the coefficients in~\eqref{eqn:decay_KL_1}}. {Appendix~\ref{App_sec_KL} gives the complete calculation. }\end{proof}

{\emph{Comparison with the continuous flow.} The spatial terms in the discrete expansion of Lemma~\ref{Lemma_KL_decay_pre} also appear in the second derivative of the KL divergence along the Fokker--Planck flow. Let $(\rho_t)_{t\ge0}$ solve the Fokker--Planck equation with $\rho_0=\rho_k$. Under the smoothness and integration-by-parts conditions in Lemmas~\ref{Lemma_FI_derivative}--\ref{lemma_second_order}, those lemmas give}
\begin{align}
\label{dt2_KL}
&\frac{1}{2}\frac{d^2}{dt^2}\KL(\rho_t\|\rho^*)\Big|_{t=0} = \beta^{-1}\!\int\! \left\langle \nabla\log\frac{\rho_k}{\rho^*}, \,\nabla^2 V\,\nabla\log\frac{\rho_k}{\rho^*}\right\rangle \rho_k\,dx +\, \beta^{-2}\!\int\! \left\|\nabla^2\log\frac{\rho_k}{\rho^*}\right\|_{\!F}^2 \rho_k\,dx\notag \\
&= \int \left|\mathcal{D}_k^{\beta}\!\left(\log\frac{\rho_k}{\rho^*}\right)\right|^2 \rho_k \, dx \,+\, \beta^{-2}\!\int\! \left\langle \nabla \log\frac{\rho_k}{\rho^*}, \,\nabla^2\log\frac{\rho_k}{\rho^*}\, \nabla \log\frac{\rho_k}{\rho^*}\right\rangle \rho_k \, dx .
\end{align}

\subsection{Convergence of KL divergence and the mixing time}
\label{sec_mixing}

{The direct transport estimate in Lemma~\ref{lem:analytic-descent} gives a one-step bound relative to the target $\rho^*$. Iteration yields the following convergence estimate. The auxiliary comparison with the kernel invariant density is given in Appendix~\ref{sec:certificates}.}

\oldresult{11}
\begin{lemma}[\blue{One-step entropy recurrence}]\label{lemma_KL_one_step}
\blue{Under~\eqref{eq:strong-convexity} and Assumption~\ref{assump:SO-weak}, for sufficiently small $0<h<\alpha^{-1}$, the KL divergence satisfies
}\begin{equation}\blue{\label{eqn:KL_one_step_sec_4}
\KL(\rho_{k+1}\|\rho^*)\le(1-\alpha h)\KL(\rho_k\|\rho^*)+Ch^3, }\end{equation}
\blue{uniformly for $0\le k<N(h)$. The constant $C$ and the small-stepsize threshold depend only on the fixed problem parameters and the uniform bounds in Assumption~\ref{assump:SO-weak}. }\end{lemma}
\begin{proof}
{Lemma~\ref{lem:analytic-descent} gives
}\[
\blue{\KL(\rho_{k+1}\|\rho^*)-\KL(\rho_k\|\rho^*) \le-\frac{h}{2\beta}\FI(\rho_k\|\rho^*)+Ch^3.}
\]
{The target logarithmic Sobolev inequality~\eqref{PL_ineq}, $\FI(\rho_k\|\rho^*)\ge2\beta\alpha\KL(\rho_k\|\rho^*)$, proves the claimed contraction.}\end{proof}

\oldresult{12}
\begin{theorem}[\blue{Quantitative KL bound}]
\label{thm:conditional}\label{Thm_KL_decay}
\blue{Assume~\eqref{eq:strong-convexity}, $\KL(\rho_0\|\rho^*)<\infty$, and Assumption~\ref{assump:SO-weak} uniformly for $0\le j<N(h)$. For sufficiently small $0<h<\alpha^{-1}$ and $0\le k\le N(h)$,
}\begin{align}\blue{\label{eq:original-gaussian-transient}
\KL(\rho_k\|\rho^*) }&\blue{\le e^{-\alpha kh}\KL(\rho_0\|\rho^*)+\frac{C}{\alpha}h^2. }\end{align}
\blue{Here $C$ is the constant in~\eqref{eqn:KL_one_step_sec_4}, independent of $k$ and $h$ on the stated horizon. The iteration count and stepsize choice in Informal Theorem~2 hold whenever the uniform hypotheses cover the requested schedule. }\end{theorem}
\begin{proof}
{Iterate~\eqref{eqn:KL_one_step_sec_4} to obtain
}\[
\blue{\KL(\rho_k\|\rho^*) \le(1-\alpha h)^k\KL(\rho_0\|\rho^*) +Ch^3\sum_{j=0}^{k-1}(1-\alpha h)^j.}
\]
{The sum is at most $(\alpha h)^{-1}$ and $(1-\alpha h)^k\le e^{-\alpha kh}$, which proves~\eqref{eq:original-gaussian-transient}. Put $D_0=\KL(\rho_0\|\rho^*)$. Choose $h=c_*\sqrt\delta$ with $(C/\alpha)c_*^2\le1/2$ and $k\ge(\alpha h)^{-1}\log(2(1+D_0)/\delta)$. For sufficiently small $\delta>0$, both terms in~\eqref{eq:original-gaussian-transient} are at most $\delta/2$, proving Informal Theorem~2 on the stated horizon. }\end{proof}

\oldresult{13}
\begin{corollary}[\blue{Sufficient stepsize}]
\label{cor:sufficient-stepsize}\label{cor:step_size_BRWP}
\blue{Under the hypotheses of Theorem~\ref{Thm_KL_decay}, let $h_0>0$ be the stepsize threshold in Lemma~\ref{lemma_KL_one_step}. The condition}
\[
\blue{0<h\le h_0,\qquad h<\alpha^{-1}}
\]
\blue{is sufficient for geometric decay of the KL divergence up to an $O(h^2)$ error on the stated iteration horizon.}
\end{corollary}
\begin{proof}
Since $0<1-\alpha h<1$, iterating~\eqref{eqn:KL_one_step_sec_4} gives~\eqref{eq:original-gaussian-transient}.
\end{proof}

Some other quantities are also often useful for measuring the efficiency of a sampling algorithm, including the Wasserstein-2 distance between \( \rho_k \) and \( \rho^* \) and the mixing time. In particular, the mixing time is defined as
 \begin{equation}
{t_{\rm mix}(\eta,\rho_0) = \inf\{k\in\mathbb N_0 \mid d_{\rm TV}(\rho_k,\rho^*)\leq\eta\}\,, }\end{equation}
{Here $\inf\varnothing=+\infty$, and $t_{\rm mix}(\eta,\rho_0)$ is the first iteration at which the total-variation distance to the target is at most $\eta$. The iterates start from $\rho_0$, and the total-variation distance is defined by}
\begin{equation}
d_{\rm TV}(\rho_k, \rho^*) = \frac{1}{2}\int_{\R^d} |\rho_k(x) - \rho^*(x)| \, dx\,.
\end{equation}

Recalling Pinsker's inequality
\[
d_{\rm TV}(\rho_k, \rho^*)^2 \leq \frac{1}{2} \KL(\rho_k \| \rho^*)\,,
\]
and Talagrand's inequality
 \[
{\frac{\beta\alpha}{2} W_2(\rho_k, \rho^*)^2 \leq \KL(\rho_k \| \rho^*)\,,
}\]
we obtain the following direct corollary of Theorem \ref{Thm_KL_decay}.

\oldresult{14}
\begin{corollary}[\blue{Wasserstein error and mixing time}]
\label{cor:metric}
\blue{Let $\{\rho_k\}$ be the BRWP iterates with $0<h\le h_0$ and $h<\alpha^{-1}$, as in Corollary~\ref{cor:step_size_BRWP}. Suppose Assumption~\ref{assump:SO-weak} holds uniformly for $0\le j<N(h)$ and $\KL(\rho_0\|\rho^*)<\infty$. Then the following bounds hold for $0\le k\le N(h)$:
}\begin{align*}
\blue{W_2(\rho_k,\rho^*)^2} &\blue{\le\frac{2e^{-\alpha kh}}{\beta\alpha} \KL(\rho_0\|\rho^*)+\frac{2C}{\beta\alpha^2}h^2,}\\
\blue{d_{\rm TV}(\rho_k,\rho^*)^2} &\blue{\le\frac{e^{-\alpha kh}}{2} \KL(\rho_0\|\rho^*)+\frac{C}{2\alpha}h^2,}
\end{align*}
\blue{where $C$ is the constant in~\eqref{eqn:KL_one_step_sec_4}. In particular, for sufficiently small $\delta>0$, choose $h=c_W\sqrt\delta$. Then $W_2(\rho_k,\rho^*)^2\le\delta$ is achieved with
}\[
\blue{k=\mathcal O\!\left( \frac{1}{\alpha\sqrt\delta} \log\frac{1+\KL(\rho_0\|\rho^*)}{\delta}\right).}
\]
\blue{For total-variation accuracy $\eta>0$ sufficiently small, choose $h=c_{\rm TV}\eta$. The mixing time satisfies
}\[
\blue{t_{\rm mix}(\eta,\rho_0) =\mathcal O\!\left( \frac{1}{\alpha\eta} \log\frac{1+\KL(\rho_0\|\rho^*)}{\eta^2}\right).}
\]
\blue{Here $c_W,c_{\rm TV}>0$ are sufficiently small constants depending on the fixed problem parameters and the uniform assumption bounds. The stepsizes must lie in the stated range and the iteration horizon must cover the indicated schedules. }
\end{corollary}
\begin{proof}
{Apply Talagrand's and Pinsker's inequalities to \eqref{eq:original-gaussian-transient}. Choose $c_W$ and $c_{\rm TV}$ so that the respective residual terms are at most $\delta/2$ and $\eta^2/2$. Making the exponential terms at most the same amounts gives the stated iteration counts.}
\end{proof}

\section{Estimation of the Score Function and Practical Considerations}
\label{sec_score_implementation}

The numerical accuracy of the score approximation used in the BRWP update is a critical factor in the performance of the algorithm. Any error in the evaluation of the score enters the iteration directly and may propagate across steps as an additional perturbation. In this section, we discuss practical score approximation schemes.

{The following strategies can be used to evaluate the score}:
\begin{enumerate}
    \item \emph{(Kernel density estimation, KDE).} A natural approach, adopted in~\citep{BRWP_2023,TT_BRWP}, is to estimate the score using kernel density estimation based on the particle ensemble $\{x_{k,j}\}_{j=1}^N$; see also~\citep{Wibisono2024OptimalSE,jiang2009general} for score estimation. {A positive Gaussian bandwidth gives a smooth positive density and a well-defined score. Error depends on regularity, bandwidth, sampling dependence, dimension, and the norm in which it is measured. Once the density representation is fixed, score evaluations at the particles may be performed independently. }

    \item \emph{(Richardson extrapolation).} One may propagate the auxiliary density $\widehat\rho_k^{\mathrm{rich}}$ and use its score, as in Algorithm~\ref{algo_BRWP_richardson}.

    \item \emph{(Laplace approximation).} Since the kernel formula involves a heat kernel with a small parameter $h$, the integral is dominated, as $h\to0$, by a neighborhood of the minimizer of the phase. This observation can be made rigorous via a Laplace approximation. For example,
     \[
    \begin{aligned}&{\int_{\R^d} \exp\!\Big[-\frac{\beta}{2}\Big(V(z) + \frac{\|z - y\|_2^2}{2h}\Big)\Big]\,dz}
    \\ &{\approx \frac{(4\pi h/\beta)^{d/2}}{\det(I+h\nabla^2V(y^*))^{1/2}}\exp\!\Big[-\frac{\beta}{2}\Big(V(y^*) + \frac{\|y^* - y\|_2^2}{2h}\Big)\Big]\,,}\end{aligned}
    \]
where
    \[
    y^* \;=\; \arg\min_{z\in\R^d}\Big(V(z) + \frac{\|z - y\|_2^2}{2h}\Big) \;=\; \mathrm{prox}_V^h(y)\,.
    \]
    The approximation error admits an explicit expansion, and higher-order corrections are available within the same framework. This Laplace approximation reduces the original high-dimensional integral to a Gaussian integral around $y^*$. {The Gaussian factor has a closed form once the minimizer and Hessian are computed. Lemma~\ref{lemma_lap_deno} gives the remainder estimate on compact sets under its stated hypotheses. }

    \item \emph{(Low-dimensional or parametric structure).} When the evolving density admits a low-dimensional or parametric representation (e.g., tensor trains, Gaussian mixtures, or normalizing flows), the score or the kernel operator $\K_V^h$ can often be evaluated directly from the representation, without relying on kernel density estimation. {A low-rank representation may reduce evaluation costs}.
\end{enumerate}

As an alternative to kernel density estimation, one may propagate an auxiliary density $\widehat\rho_k$ by iterating the kernel formula \eqref{rho_T_BRWP} and evaluate the score from this density. This avoids constructing an empirical density at each iteration and instead produces a sequence of deterministic densities
\begin{equation}
\label{prox_rho_k_sec_5}
    \widehat\rho_{k+1} = \mathcal{K}_V^{\,h}(\widehat\rho_k)\,, \qquad \widehat\rho_0 = \rho_0\,,
\end{equation}
which can be evaluated without particle sampling, for example, via tensor-based approximations. Because the Gaussian factors in the kernel formula factorize across dimensions,  structured quadrature schemes {may be used for the associated integrals}. As a result, computation shifts to the evaluation of high-dimensional, tensorized integrals, whose complexity depends on the structural properties of the target distribution and the achievable tensor ranks. {For the exact kernel iteration, Corollary~\ref{cor:KV_non_refreshed} gives an $O(h)$ comparison with the BRWP density in $W_1$, under its stated one-step error bound.}

{To improve the approximation of the Fokker--Planck evolution, we combine two half steps and one full step of the kernel formula by Richardson extrapolation. Starting from $\widehat\rho_0^{\mathrm{rich}}=\rho_0$, define
}\[
{\widehat\rho_{k+1}^{\mathrm{rich}}=2\K_V^{h/2}\bigl(\K_V^{h/2}(\widehat\rho_k^{\mathrm{rich}})\bigr)-\K_V^h(\widehat\rho_k^{\mathrm{rich}}).
}\]
We use the score of $\widehat\rho_{k+1}^{\mathrm{rich}}$ to advance the particles from step $k$ to $k+1$ in Algorithm~\ref{algo_BRWP_richardson}.

\blue{Algorithm~2 uses the score of an independently evolved Richardson approximation, which is assumed to be differentiable and positive at the particle positions where its score is evaluated. Convergence of the resulting particle scheme is not established by the analysis in Section~\ref{sec_conv_KL}.}

\begin{algorithm}[H]
\caption{BRWP with Richardson-corrected regularized Wasserstein proximal operator} \label{algo_BRWP_richardson}
\begin{algorithmic}[1]
\State Initialize $\widehat\rho_0^{\mathrm{rich}} = \rho_0$.
\For{$k = 0,1,2,\dots$}
    \State Compute $\K_V^{h/2}(\widehat\rho_k^{\mathrm{rich}})$, $\K_V^{h/2}\!\left(\K_V^{h/2}(\widehat\rho_k^{\mathrm{rich}})\right)$, and $\K_V^h(\widehat\rho_k^{\mathrm{rich}})$ by tensor computation.
\State Form the Richardson approximation
\[
    \widehat\rho_{k+1}^{\mathrm{rich}} = 2\, \K_V^{h/2}\!\left(\K_V^{h/2}(\widehat\rho_k^{\mathrm{rich}})\right) - \K_V^h(\widehat\rho_k^{\mathrm{rich}}),
\]
 \State {Evaluate }$\nabla\log\widehat\rho_{k+1}^{\mathrm{rich}}$ {at the particles}.

    \For{each particle $j = 1,\dots,N$}
        \State
        \[
        x_{k+1,j} = x_{k,j} - h\Big( \nabla V(x_{k,j}) + \beta^{-1}\nabla\log \widehat\rho_{k+1}^{\mathrm{rich}}(x_{k,j}) \Big).
        \]
    \EndFor
\EndFor
\end{algorithmic}
\end{algorithm}
For exact kernel evaluations, the Richardson approximation has mass one but may not be nonnegative. Subsequent kernel evaluations are understood through the linear integral formula \eqref{rho_T_BRWP}, extended to signed inputs when necessary.

\section{Numerical Experiments}
\label{sec_NE}

In this section, we present several numerical experiments to illustrate the behavior of the algorithms. All numerical experiments are conducted in a $10$-dimensional sample space, i.e., \( \R^{d} \) with $d=10$, and we only plot the result in the first dimension for the purpose of presentation.

\noindent\textbf{Example 1.} In this example, we explore the evolution of density functions over several iterations using the kernel formula \eqref{rho_T_BRWP}. We numerically illustrate how the computed density approaches the target density over the reported iterations under an appropriately chosen stepsize. To efficiently manage the high-dimensional integrations involved, we employ tensor-train approximation for the density functions. Further details on this approach can be found in our previous work \citep{TT_BRWP}.

The first distribution we consider is a mixed Gaussian distribution defined as
\begin{equation}
\label{mixed_Gaussian}
    \rho^*(x) = \frac{1}{2(2\pi\sigma^2)^{d/2}} \left[\exp\left(-\frac{\|x-a\|_2^2}{2\sigma^2}\right) + \exp\left(-\frac{\|x+a\|_2^2}{2\sigma^2}\right)\right]\,,
\end{equation}
where \( a = (2,\cdots,2) \).

\begin{figure}[H]
    \centering
    \begin{subfigure}[t]{.3\linewidth}
        \includegraphics[width=\linewidth]{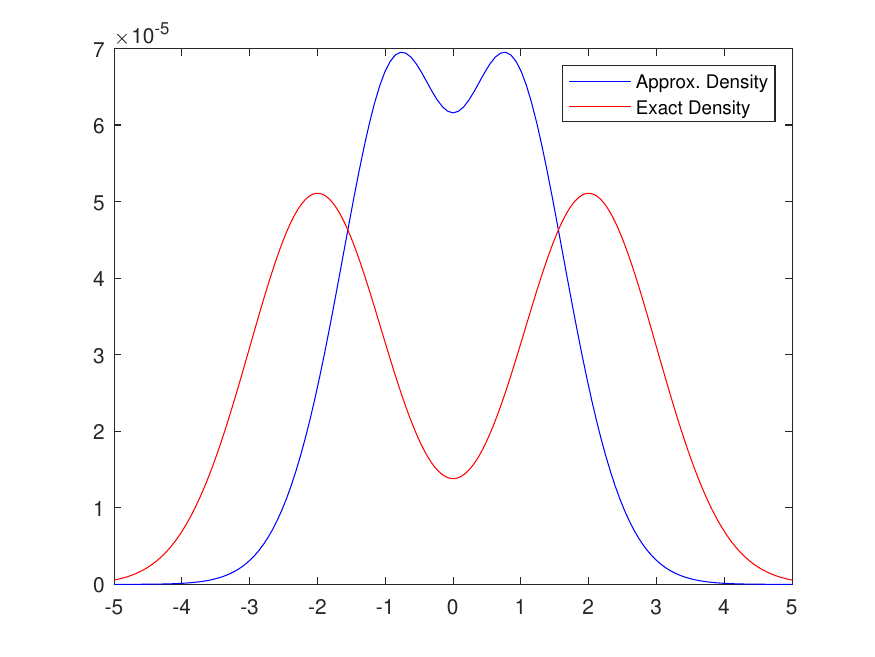}
        \caption{$h=0.01$, $20$ iterations.}
    \end{subfigure}
    \begin{subfigure}[t]{.3\linewidth}
        \includegraphics[width=\linewidth]{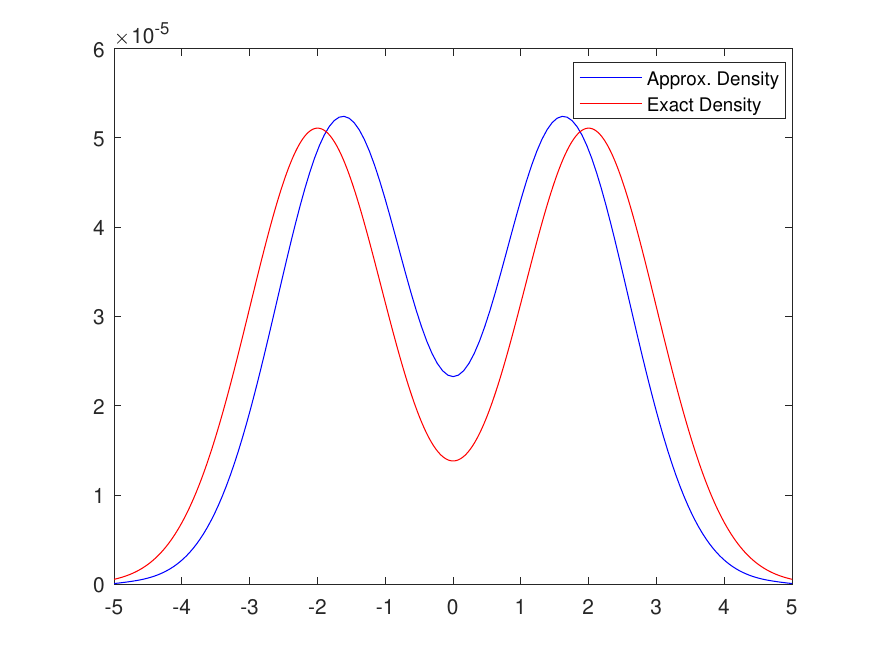}
        \caption{$h=0.01$, $100$ iterations.}
    \end{subfigure}
    \begin{subfigure}[t]{.3\linewidth}
        \includegraphics[width=\linewidth]{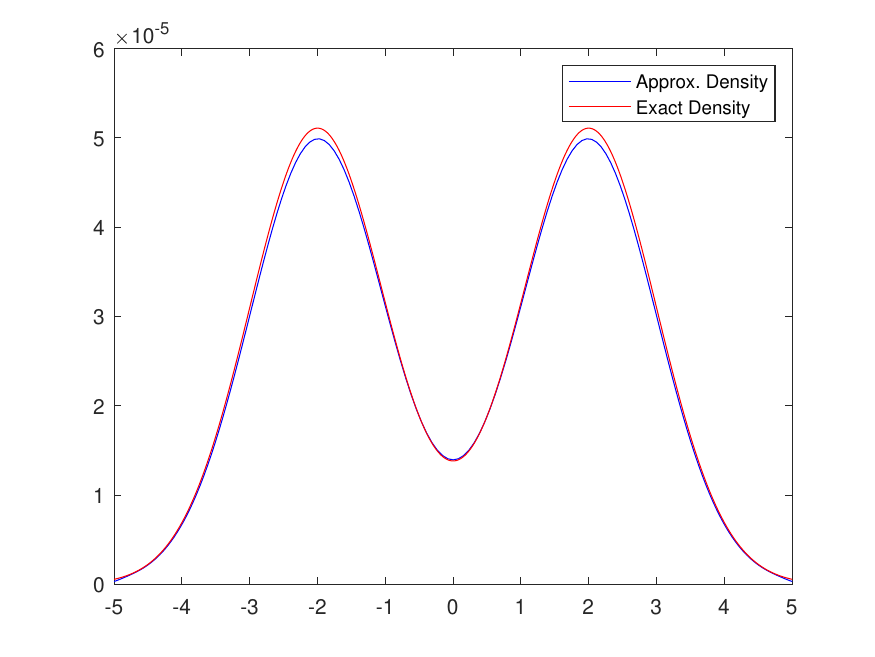}
        \caption{$h=0.01$, $400$ iterations.}
    \end{subfigure}\\
    \begin{subfigure}[t]{.3\linewidth}
        \includegraphics[width=\linewidth]{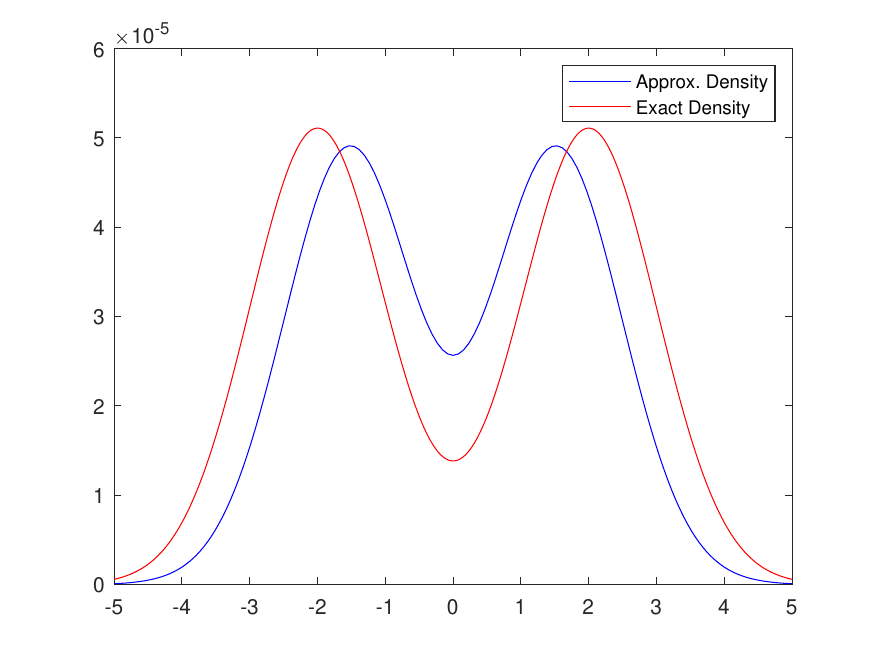}
        \caption{$h=0.05$, $20$ iterations.}
    \end{subfigure}
    \begin{subfigure}[t]{.3\linewidth}
        \includegraphics[width=\linewidth]{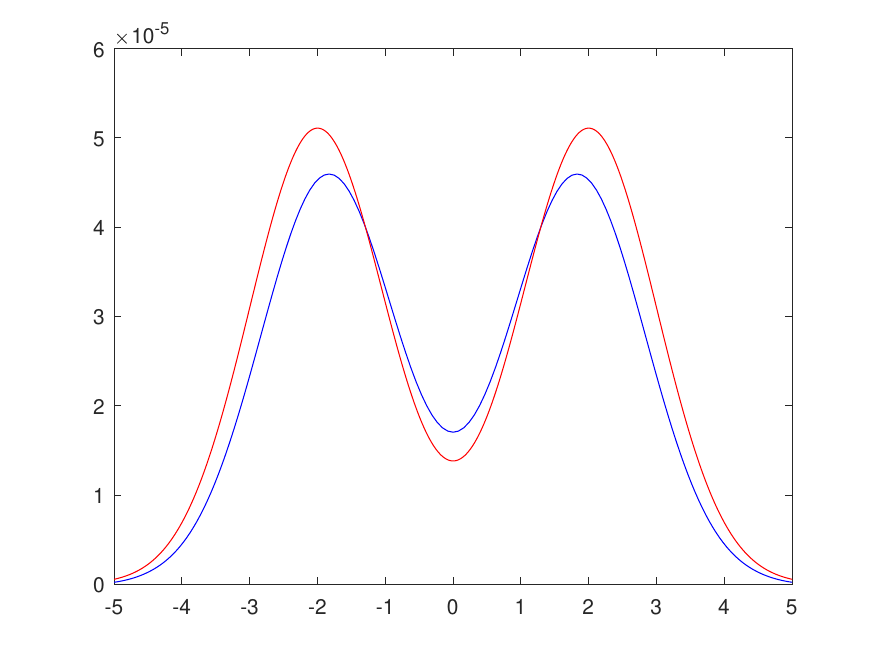}
        \caption{$h=0.05$, $40$ iterations.}
    \end{subfigure}
    \begin{subfigure}[t]{.3\linewidth}
        \includegraphics[width=\linewidth]{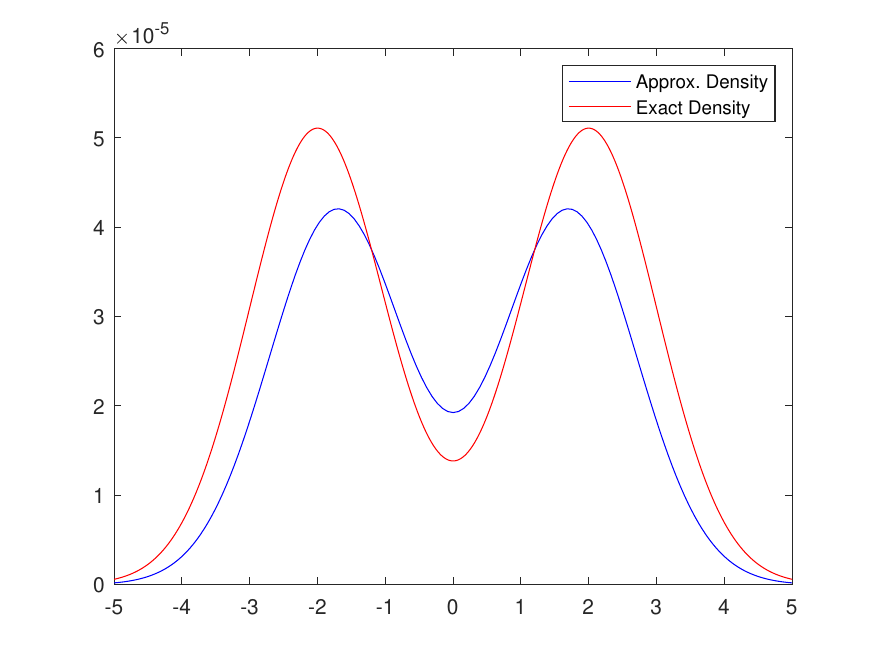}
        \caption{$h=0.1$, $15$ iterations.}
    \end{subfigure}
    \caption{Evolution of the density function (blue) with \eqref{rho_T_BRWP} for different stepsizes \( h \) for the first dimension. The initial density is \( \mathcal{N}(0,2) \). The target density is a Gaussian mixture (red).} \label{Fig:mixed_Gaussian_1}
\end{figure}

From Fig.\,\ref{Fig:mixed_Gaussian_1}, we observe that for sufficiently small stepsizes, the generated density approaches the target density over the reported iterations. Furthermore, comparing the first and last plots, we note that using a larger time stepsize yields an approximation to the target in fewer iterations but with a larger approximation error.

Next, we consider a mixture involving the \( L_1 \) norm and the \( L_{1/2} \) quasi-norm, with target density
\[
\rho^*(x) = \frac{1}{Z}\left[\exp(-\|x+2\vec{e}_1\|_1) + \frac{1}{2} \exp(-\|x-2\vec{e}_1\|_{1/2}^{1/2})\right]\,,
\]
where \( \vec{e}_1 \) is the vector with the first entry equal to 1 and all other entries equal to 0, and \( Z \) is the normalization constant.

\begin{figure}[H]
    \centering
    \begin{subfigure}[t]{.3\linewidth}
        \includegraphics[width=\linewidth]{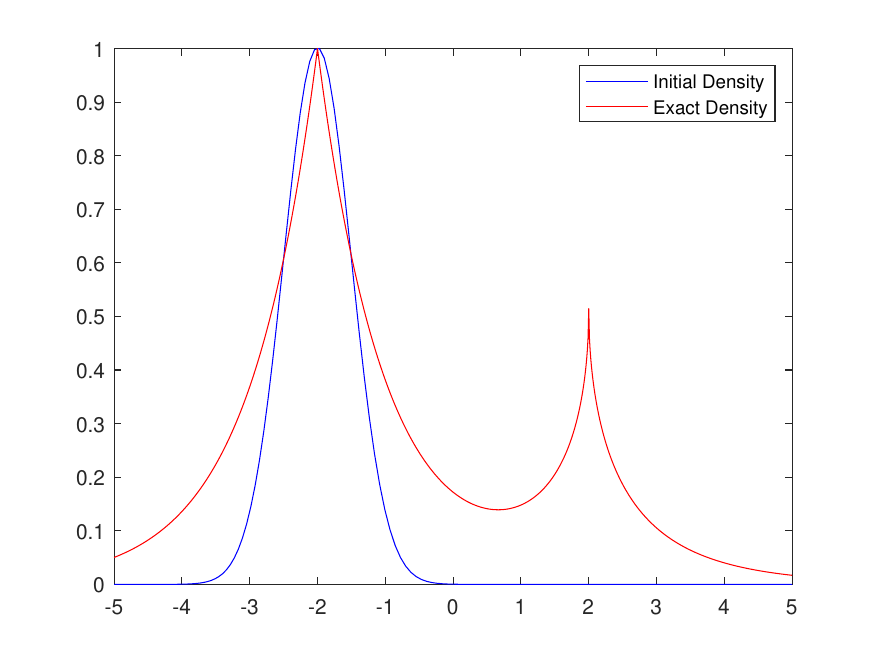}
        \caption{Initial distribution.}
    \end{subfigure}
    \begin{subfigure}[t]{.3\linewidth}
        \includegraphics[width=\linewidth]{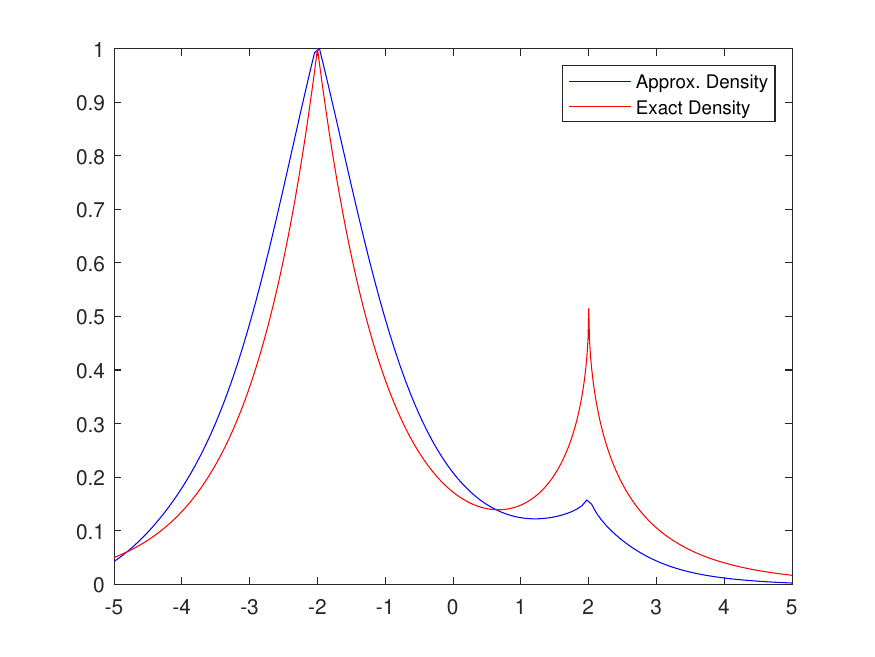}
        \caption{$h=0.05$, $20$ iterations.}
    \end{subfigure}
    \begin{subfigure}[t]{.3\linewidth}
        \includegraphics[width=\linewidth]{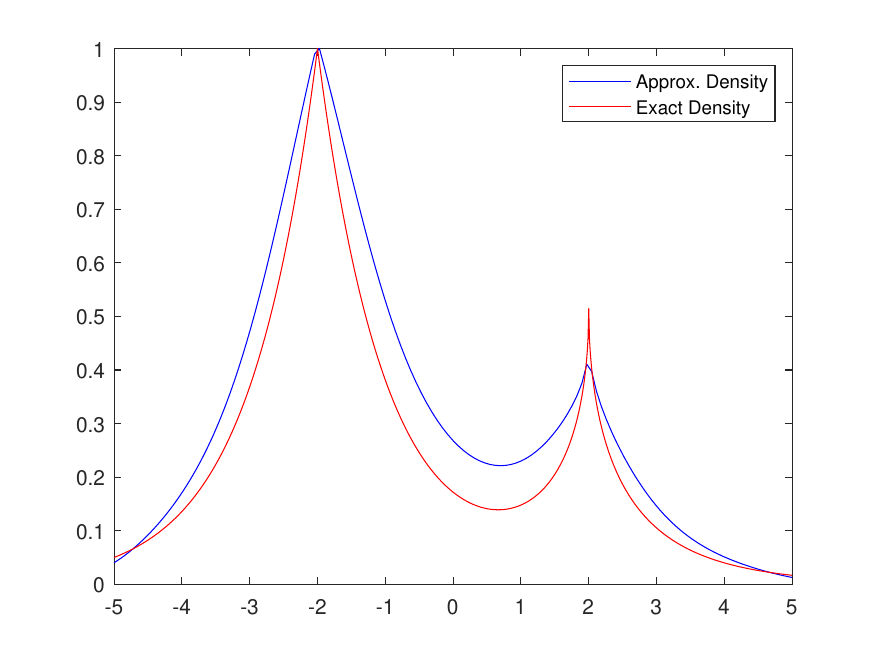}
        \caption{$h=0.05$, $50$ iterations.}
    \end{subfigure}
    \caption{Evolution of the density function (blue) for the first dimension under iteration of \eqref{rho_T_BRWP}, with the target defined above shown in red.} \label{Fig:L1L12}
\end{figure}

From Fig.\,\ref{Fig:L1L12}, we observe that even for the nonsmooth potential function, which does not satisfy the assumptions of Section~\ref{sec_kernel_approx}, the density approaches the target distribution over the reported iterations using the kernel formula.

\noindent\textbf{Example 2.} In the second example, we examine the BRWP algorithm by computing the score function based on the density function evolved using the regularized Wasserstein proximal operator, as outlined in Algorithm \ref{algo_BRWP_richardson}. We compare BRWP with the explicit Euler discretization of the probability flow ODE, the proximal Langevin algorithm in \citep{proximal_non_smooth}, and the ULA described in \eqref{def_ULA}. For explicit Euler, the score function at time \( t_k \) is approximated by a Gaussian KDE constructed from the particles.

The figure below shows the distribution of particles after 50 iterations for the Gaussian mixture distribution defined in \eqref{mixed_Gaussian}.

\begin{figure}[H]
    \centering
    \begin{subfigure}[t]{.24\linewidth}
        \includegraphics[width=\linewidth]{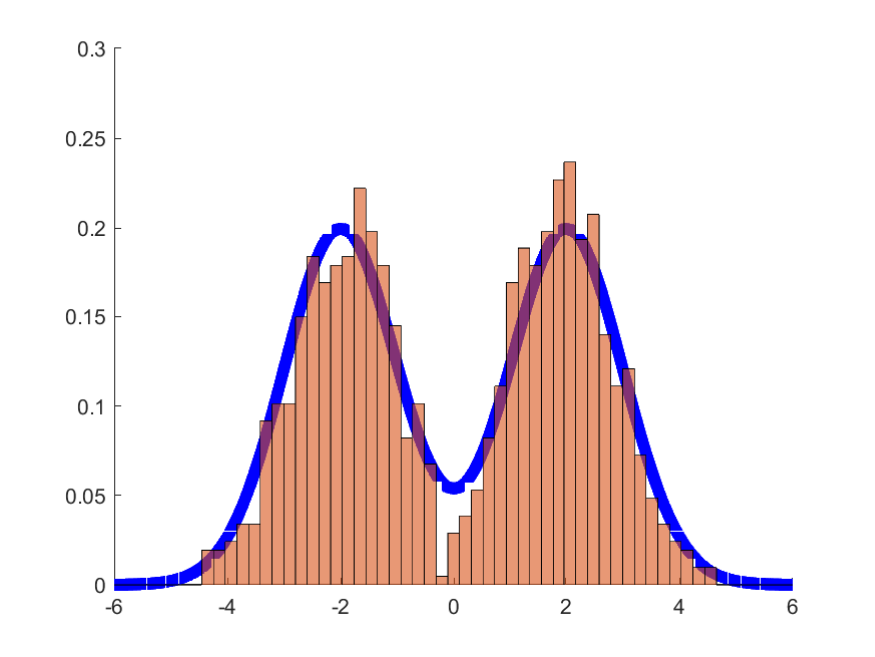}
        \caption{BRWP}
    \end{subfigure}
    \begin{subfigure}[t]{.24\linewidth}
        \includegraphics[width=\linewidth]{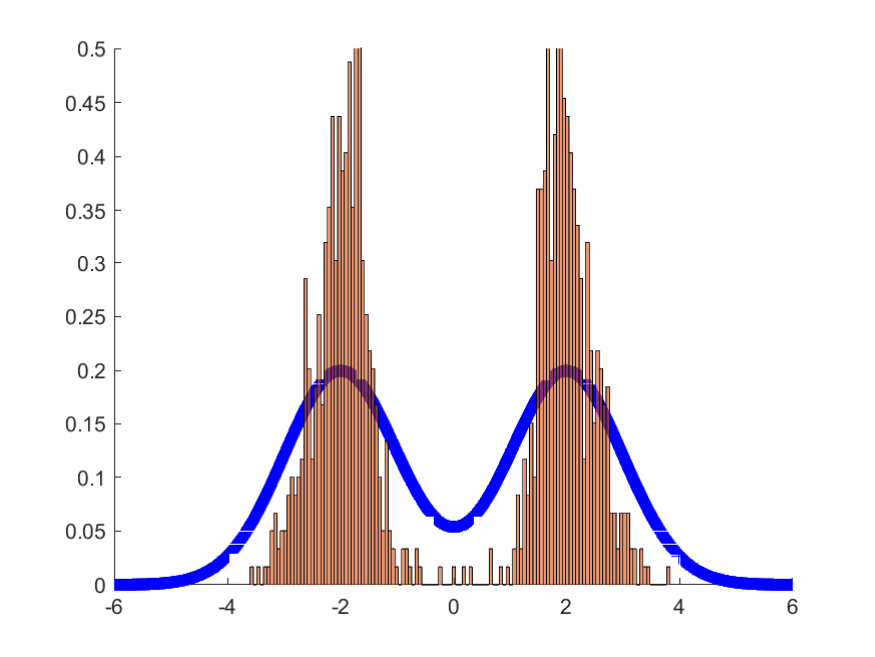}
        \caption{Explicit Euler}
    \end{subfigure}
        \begin{subfigure}[t]{.24\linewidth}
        \includegraphics[width=\linewidth]{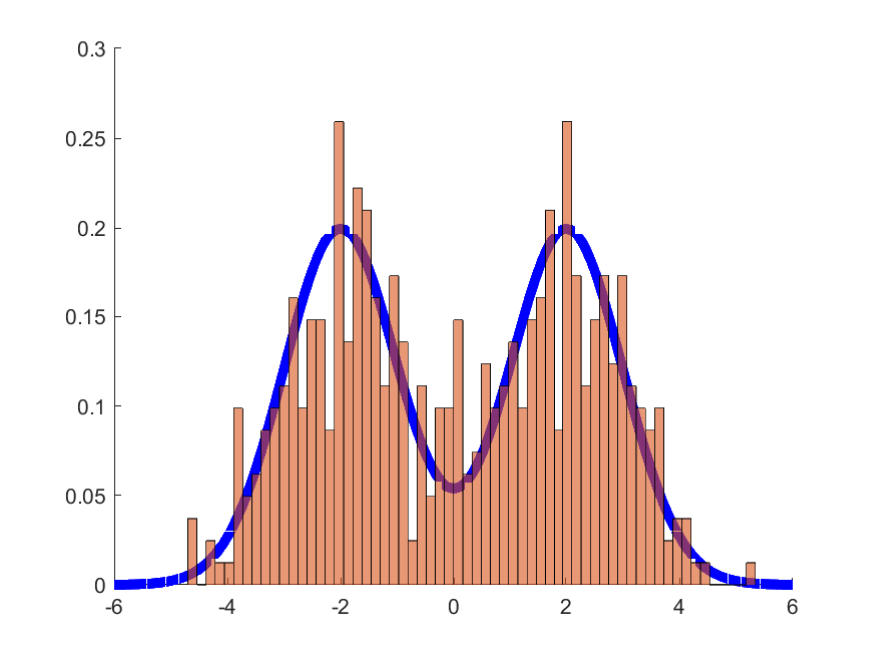}
        \caption{Proximal Langevin}
    \end{subfigure}
    \begin{subfigure}[t]{.24\linewidth}
        \includegraphics[width=\linewidth]{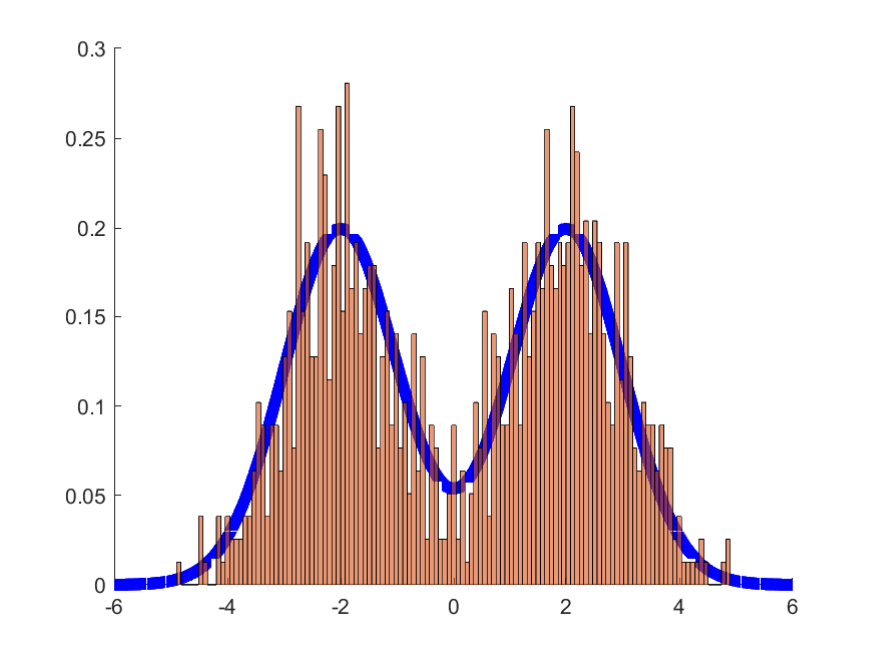}
        \caption{ULA}
    \end{subfigure}
    \caption{Histogram of 500 particles after 50 iterations in the first dimension for a Gaussian mixture distribution with $h  = 0.02$.} \label{Fig:mix_gaussian_sample}
\end{figure}

Comparing the first and second graphs in Fig.\,\ref{Fig:mix_gaussian_sample}, we observe that the BRWP histogram shows closer visual agreement with the target density than the explicit Euler histogram in this experiment.

In the second experiment, we consider a mixture of Gaussian and Laplace distributions defined as
\[
\rho^*(x) = \frac{1}{Z}\left[\exp\left(-\frac{\|x-2\|_2^2}{2\sigma^2}\right) + \exp\left(-\frac{\|x+2\|_1}{2b}\right)\right]\,,
\]
where \( Z \) is the normalization constant.

\begin{figure}[H]
    \centering
    \begin{subfigure}[t]{.24\linewidth}
        \includegraphics[width=\linewidth]{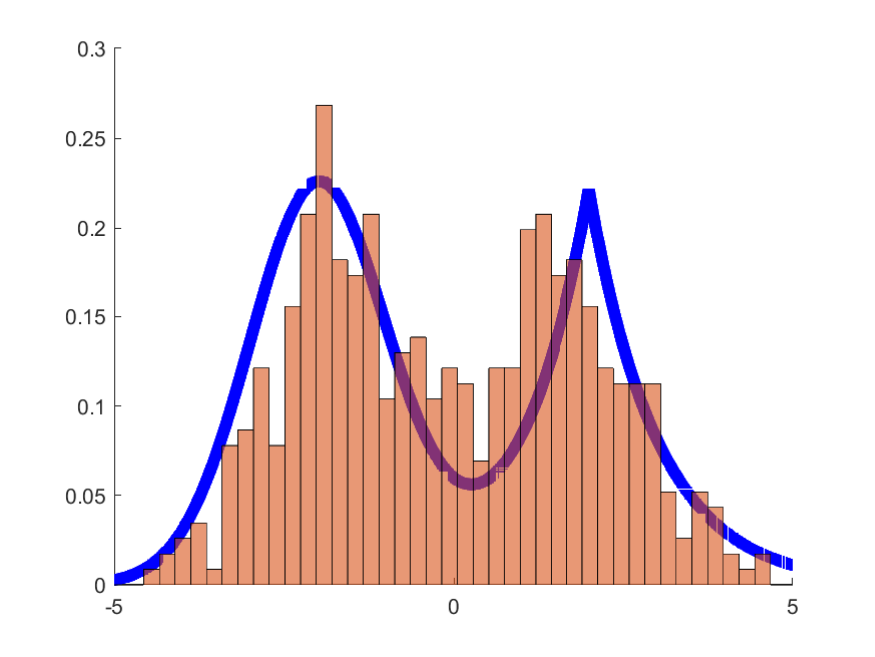}
        \caption{BRWP \\ 10 iterations}
    \end{subfigure}
    \begin{subfigure}[t]{.24\linewidth}
        \includegraphics[width=\linewidth]{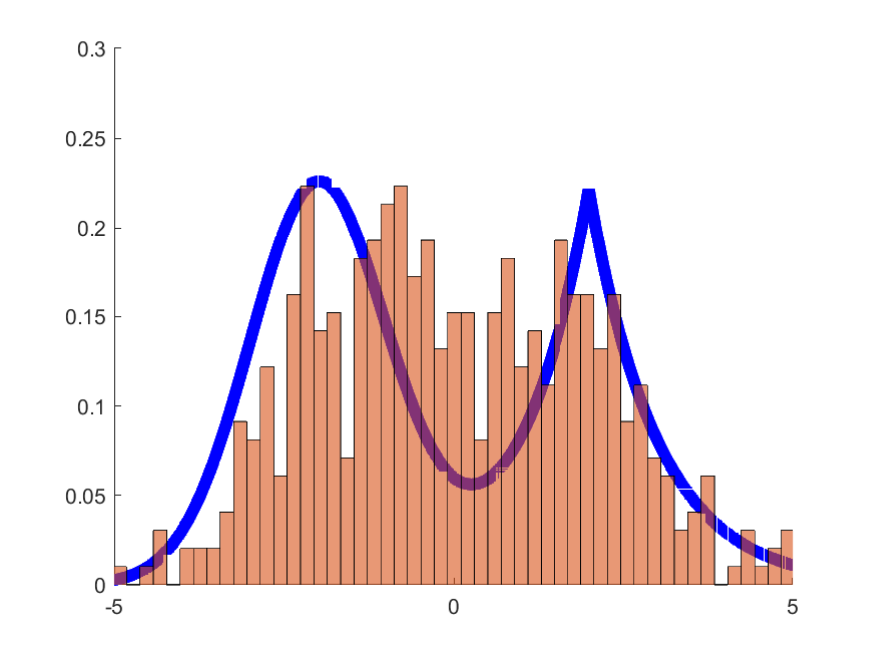}
        \caption{ULA \\ 10 iterations}
    \end{subfigure}
    \begin{subfigure}[t]{.24\linewidth}
        \includegraphics[width=\linewidth]{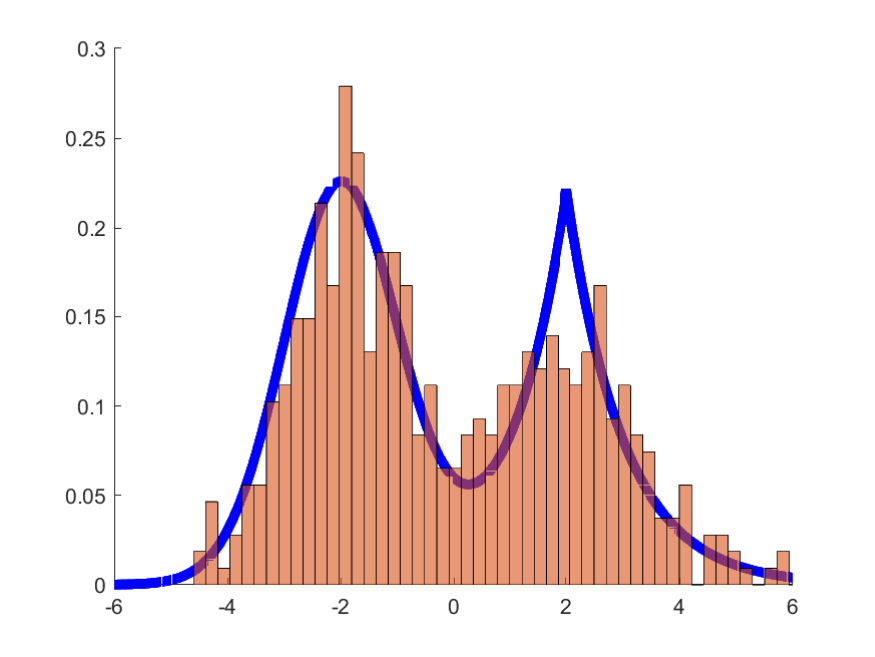}
        \caption{BRWP \\ 20 iterations}
    \end{subfigure}
    \begin{subfigure}[t]{.24\linewidth}
        \includegraphics[width=\linewidth]{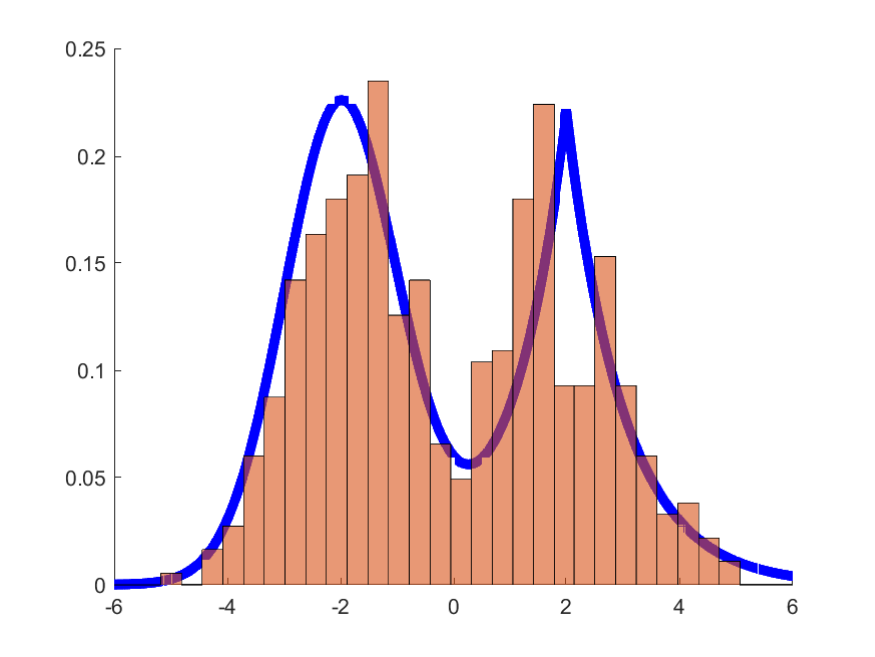}
        \caption{ULA \\ 20 iterations}
    \end{subfigure}
    \caption{Histogram of $500$ particles in the first dimension for the mixture of Gaussian and Laplace distributions with $h  = 0.02$.} \label{Fig:Gau_Lap_sampling}
\end{figure}

In Fig.\,\ref{Fig:Gau_Lap_sampling}, the BRWP algorithm shows closer visual agreement with the target than ULA after the displayed number of iterations.

\Needspace{8\baselineskip}
\section{Conclusion and discussion}

In this work, we present the convergence analysis of the BRWP algorithm, which is designed to sample from a distribution known up to a normalization constant. The algorithm relies on a semi-implicit time discretization of a deterministic probability flow ODE whose associated Liouville equation coincides with the Fokker--Planck equation of overdamped Langevin dynamics. To address the challenge of approximating the evolution of the density function, which usually involves intensive optimization, we apply the regularized Wasserstein proximal operator, whose solution is represented by a closed-form kernel formula. Using the weak formulation and Taylor expansion, we first demonstrate that the kernel formula approximates the evolution governed by the Fokker--Planck equation. We then conduct a weak second-order numerical analysis to study the convergence of the BRWP algorithm in KL divergence. \blue{Under the stated hypotheses, including strong convexity and Assumption~\ref{assump:SO-weak}, one-step entropy estimates give geometric decay up to an $O(h^2)$ error for sufficiently small stepsizes, using the score from the exact kernel formula.}

Our BRWP sampling method can be viewed as an interacting particle system based on a kernel derived from the regularized Wasserstein proximal operator. This concept shares similar motivations with sampling algorithms from interacting particle systems proposed in \citep{stuart_Bayersian, carrillo2022consensus, reich2021fokker, carrillo2019blob,liu2016stein}. In BRWP, the kernel serves as a linear approximation of the Fokker--Planck semigroup. A major practical challenge lies in accurately approximating the evolving density and efficiently evaluating the kernel formula, especially in high dimensions.

There are several interesting future directions to explore from this work. \blue{First, verifying Assumption~\ref{assump:SO-weak} from general initial distributions remains an open problem.} It would also be valuable to extend the discussion to more general non-log-concave distributions and investigate the convergence of the algorithm. Second, the current analysis and numerical framework can be extended and generalized to different schemes, addressing challenges such as constrained sampling problems, sampling from group-symmetric distributions, and sampling under different metrics. Lastly, we are also applying this approach to broader fields, including global optimization, time-reversible diffusion, and solving high-dimensional Hamilton-Jacobi equations.

\begin{appendix}

\section{Proofs and Auxiliary Estimates for Section~\ref{sec_kernel_approx}}
\subsection{Proofs of the approximation estimates}
\label{sec_appendix_kernel}
In this appendix, we provide proofs of the main lemmas and theorems in Section~\ref{sec_kernel_approx}.

\begin{proof}[Proof of Lemma~\ref{lem:PV-weak-local}]
{We first prove the weak estimate for $u\in C_c^{0,1}(U_0)$ and $\varphi\in C_c^{2,1}(U)$, both extended by zero. We then extend the estimate to the full input density using localization. Integrations by parts in $x$ are over $\R^d$, so no boundary terms occur; Lipschitz inputs follow by mollification. }\textit{Step 1. Kernel representation and Gaussian tails.} For $x\in\R^d$, write $\G_h[f(x,y)] = \int_{\R^d}G_h(y)f(x,y)dy$. Then
 \begin{equation}{\label{eq:PT-split}
P_V^h\varphi(x) =\varphi(x) +\frac{N(x)}{q_h(x)} \,, \qquad N(x):=\G_h\!\big[\psi(x+y)\,(\varphi(x+y)-\varphi(x))\big] \,,\quad q_h(x) =\G_h\!\big[\psi(x+y)\big] \,, }\end{equation}
where the Gaussian convolution is in the $y$-variable.

For all $x,y\in\R^d$,
\[
\varphi(x+y)-\varphi(x) =\int_0^1 \nabla\varphi(x+\theta y)\!\cdot y\,d\theta \,,
\]
hence
\[
N(x)=\int_0^1 \G_h\!\big[\psi(x+y)\,(\nabla\varphi(x+\theta y)\!\cdot y)\big]\,d\theta \,.
\]

Let
\[
r_h:=\min\!\left\{\frac{r_0}{2},\,\sqrt{\frac{12}{\beta}h\log\!\frac{1}{h}}\right\}.
\]
{Choose $h_0\le e^{-1}$ so that $12h\log(1/h)/\beta\le r_0^2/4$ for $0<h\le h_0$. }A standard change of variables gives
 \begin{equation}
{\label{eqn:Gaussian_tail} \int_{\{\|y\|>r_h\}} h^{-d/2}e^{-\beta\|y\|^2/(4h)}\,dy \;\le\; C_{d,\beta}\Bigl(\log\!\frac{1}{h}\Bigr)^{\frac{d}{2}-1}\,h^3 \;\le\; C_{d,h_0}\,h^2 \,. }\end{equation}
Since $h(\log\tfrac1h)^{d/2-1}$ is bounded for $ h \in (0,h_0]$,
\[
\G_h\big[\mathbf 1_{\{\|y\|>r_h\}}\big]\;\le\;C\,h^2 \,.
\]
Thus, tail contributions are $\mathcal O(h^2)$ and can be absorbed in the remainder term. For  {sufficiently small $h$ and all $x\in U_0$, }we also have
\[
x+B(0,r_h)\subset U_r \,,
\]
so Gaussian integrals can be truncated to $B(0,r_h)$ with an $\mathcal O(h^2)$ error\,.

\noindent\textit{Step 2. Expansion and integral identities.} For  {$x\in U_0$ }and $y\in B(0,r_h)$,
\begin{equation}\label{eq:psi-Taylor}
\psi(x+y) =\psi(x)+\nabla\psi(x)\!\cdot y +\tfrac12 y^\top\nabla^2\psi(x)y +\tfrac16\nabla^3\psi(x)[y,y,y] +R_\psi(x,y) \,,
\end{equation}
with
\[
|R_\psi(x,y)| \le C\,\|\psi\|_{C^{3,1}(U_r)}\,\|y\|^4 \,.
\]

Furthermore,
 \begin{equation}{\label{eq:D-expansion}
q_h(x)=\G_h[\psi(x+y)] =\psi(x)+h\beta^{-1}\Delta\psi(x) +\mathcal O\!\big(h^2\|\psi\|_{C^{3,1}(U_r)}\big) \,. }\end{equation}
Since $V$ is bounded on  {$U_0$}, the quantity  {$q_h(x)$ }is bounded below uniformly in $h$\,.

\noindent\textit{Step 3. Terms in the Taylor expansion of \(\psi\).} Recall the Gaussian identities for $y\sim\mathcal N(0,2h\beta^{-1}I_d)$:
\[
\G_h[y]=0\,,\qquad \G_h[yy^\top]=2h\beta^{-1}I_d\,,\qquad \G_h[y_iy_jy_k]=0\,.
\]
For smooth $f$,
\begin{equation}
\label{eqn:stein}
\G_h[y\!\cdot\!f(y)] =2h\beta^{-1}\,\G_h[\nabla_y\!\cdot f(y)] \,.
\end{equation}

 {\color{black}Define $R_\psi$ outside the truncated ball by the exact Taylor-remainder identity as well. There $\psi$ is bounded and the Taylor polynomial has degree three; the Gaussian integrals of these tail terms, including a factor $|y|$, are $O(h^2)$ for the chosen $h_0$. Thus the following full-space decomposition and its remainder bounds are justified. Substituting~}\eqref{eq:psi-Taylor} into $N(x)$ gives
\[
N(x)=I_0(x)+I_1(x)+I_2(x)+I_3(x)+I_{\mathrm{rem}}(x) \,,
\]
where
\begin{align*} I_0(x)&:=\psi(x)\int_0^1\G_h[\nabla\varphi(x+\theta y)\!\cdot y]\,d\theta\,,\\
I_1(x)&:=\int_0^1\G_h[(\nabla\psi(x)\!\cdot y)(\nabla\varphi(x+\theta y)\!\cdot y)]\,d\theta\,,\\
I_2(x)&:=\tfrac12\int_0^1\G_h[(y^\top\nabla^2\psi(x)y)(\nabla\varphi(x+\theta y)\!\cdot y)]\,d\theta\,,\\
I_3(x)&:=\tfrac16\int_0^1\G_h[\nabla^3\psi(x)[y,y,y](\nabla\varphi(x+\theta y)\!\cdot y)]\,d\theta\,,\\
I_{\mathrm{rem}}(x)&:=\int_0^1\G_h[R_\psi(x,y)(\nabla\varphi(x+\theta y)\!\cdot y)]\,d\theta\,. \end{align*}

Since $\G_h[\|y\|^4]=\mathcal O(h^2)$ and $\G_h[\|y\|^5]=\mathcal O(h^{5/2})$,
\[
|I_3(x)|+|I_{\mathrm{rem}}(x)| \le C\,h^2\,\|\psi\|_{C^{3,1}(U_r)}\,\|\nabla\varphi\|_{L^\infty(U)} \,.
\]

Set  {$p(x):=\frac{u(x)}{q_h(x)}$. Then
}\begin{equation}
{\langle I_2,p\rangle =\frac{1}{2}\!\int_0^1\!\!\int_{\R^d}\!\int_{\R^d} G_h(y)\,(y^\top\nabla^2\psi(x)y)\, \big(\nabla\varphi(x+\theta y)\!\cdot y\big)\,dy\,p(x)\,dx\,d\theta \,. }\end{equation}

Integration by parts in $x$ gives
 \[
{\langle I_2,p\rangle =-\frac{1}{2}\!\int_0^1\!\!\int_{\R^d}\!\int_{\R^d} G_h(y)\,(y\otimes y\otimes y) :\!\big(\varphi(x+\theta y)\,\nabla(\nabla^2\psi(x)p(x))\big)\, dy\,dx\,d\theta \,.
}\]

Using the evenness of $G_h$ and $\int G_h\|y\|^4\,dy=\mathcal O(h^2)$,
 \[
{|\langle I_2,p\rangle| \le C\,h^2\, \|\psi\|_{C^{3,1}(U_r)}\, \|\nabla\varphi\|_{L^\infty(U)}\, \|u\|_{C^{0,1}(U_{0})} \,.
}\]

\noindent\textit{Step 4. Main terms \(I_0\) and \(I_1\).} Define
 \[
{v(x):=\frac{\psi(x)u(x)}{q_h(x)}\,,\qquad w(x):=\frac{\nabla\psi(x)\,u(x)}{q_h(x)} \,.
}\]

\underline{$I_0$.} Using \eqref{eqn:stein} with $F_\theta(y)=\nabla\varphi(x+\theta y)$,
\[
\G_h[\nabla\varphi(x+\theta y)\!\cdot y] =2h\beta^{-1}\theta\,\G_h[\Delta\varphi(x+\theta y)] \,.
\]
Hence
 \begin{align}
{\langle I_0,p\rangle =}& {2h\beta^{-1}\int_0^1\theta \int_{\R^d}\G_h[\Delta\varphi(x+\theta y)]v(x)dx d\theta \notag}\\
{= }&{-2h\beta^{-1}\!\int_0^1\!\theta\, \G_h\left[\int_{\R^d} \nabla\varphi(x+\theta y)\cdot\nabla v(x) dx\right]\,d\theta \,. }\end{align}

\underline{$I_1$.} {Put $\sigma^2=2h/\beta$. Componentwise Stein integration by parts gives
}\[
{\G_h[(\nabla\varphi(x+\theta y)\cdot y)y_j] =\sigma^2\G_h[\partial_j\varphi(x+\theta y) +\theta y_j\Delta\varphi(x+\theta y)].
}\]
{The first term, paired with $w_j$, contributes $\sigma^2\langle\partial_j\varphi,w_j\rangle+O(h^2\|\varphi\|_{C^{2,1}}\|w_j\|_{L^1})$: expand $\nabla\varphi(x+\theta y)$ to first order and cancel its odd part. For the second term, integrate once in $x$ and then use Stein again:
}\begin{align*}
{\int w_j(x)\G_h[y_j\Delta\varphi(x+\theta y)]\,dx }&{=-\sum_i\int\partial_iw_j(x)\G_h[y_j\partial_i\varphi(x+\theta y)]\,dx}\\
&{=-\sigma^2\theta\sum_i\int\partial_iw_j(x) \G_h[\partial_{ji}\varphi(x+\theta y)]\,dx. }\end{align*}
{Its prefactor is another $\sigma^2$, so it is $O(h^2)$. For $I_0$, the same first-order expansion of $\nabla\varphi$ after the displayed integration by parts has an $O(h)$ Gaussian remainder. Therefore
}\begin{align*}
{\langle I_0+I_1,p\rangle }&{=-h\beta^{-1}\langle\nabla\varphi,\nabla v\rangle +2h\beta^{-1}\langle\nabla\varphi,w\rangle}\\
&{\quad+O(h^2\|\varphi\|_{C^{2,1}(U)}\|u\|_{C^{0,1}(U_0)}). }\end{align*}
{The bounds on $\nabla v$ and $\nabla w$ follow from the uniform lower bound for $q_h$ on $U_0$ and its uniform local derivative bounds. Also $\psi/q_h=1+O(h)$ in $C^1(U_0)$ by the heat expansion, so this replacement changes the remainder by $O(h^2)$. }

Substituting  {$v=\psi u/q_h$ and $w=\nabla\psi\,u/q_h$}, and using $\nabla\log\psi=-\frac{\beta}{2}\nabla V$ together with \eqref{eq:D-expansion}, gives
 \[
{\langle I_0+I_1,p\rangle =h\beta^{-1}\,\langle\Delta\varphi-\beta\,\nabla V\!\cdot\nabla\varphi,\,u\rangle +\mathcal O\!\big(h^2\|\varphi\|_{C^{2,1}(U)}\|u\|_{C^{0,1}(U_{0})}\big) \,.
}\]

Therefore
 \[
{\big|\langle P_V^h\varphi-\varphi-h\mathcal L\varphi,\,u\rangle\big| \le C\,h^2\,\|\varphi\|_{C^{2,1}(U)}\,\|u\|_{C^{0,1}(U_{0})} \,,
}\]
which  {proves the compact-input estimate. }

 \noindent{\textit{Step 5. Removing the compact-support restriction.} }Choose  {$\chi\in C_c^\infty(U_0)$ equal to one on a neighborhood of $\overline U$. Decompose $u=\chi u+(1-\chi)u$. The first term is covered above with $\|\chi u\|_{C^{0,1}}\le C\|u\|_{C^{0,1}(U_r)}$. For the second term, only $P_V^h\varphi$ contributes. The Hessian upper bound gives
}\begin{equation}\blue{\label{eq:qh-lower-localproof}
q_h(y)\ge(1+Lh)^{-d/2}\psi(y). }\end{equation}
{Indeed, the upper Taylor bound $V(y+z)\le V(y)+\nabla V(y)\cdot z+L|z|^2/2$ and Gaussian integration give this inequality, with an additional exponential factor }that is at least one {. For $x$ in the fixed compact support of $\varphi$, the upper Taylor bound for $V(y)$ at $x$ now yields, for small $h$,
}\[
\blue{\frac{\psi(x)G_h(x-y)}{q_h(y)} \le Ch^{-d/2}\exp\!\left(-\frac{\beta|x-y|^2}{16h}\right).
}\]
\blue{The distance between $\supp\varphi$ and $\supp(1-\chi)$ is positive. Consequently, for any fixed $N$, $|P_V^h\varphi(y)|\le C_Nh^N\|\varphi\|_\infty$ there. Since $u$ is a probability density, integration against $u(y)\,dy$ bounds the outer contribution by $Ch^2\|\varphi\|_\infty$. Combining the two contributions proves~\eqref{eq:local-PV-weak}. } \par\noindent{\emph{Step 6. Local density and positivity bounds for $u\in\mathcal{A}_{m,M}$.} For $u\in\mathcal{A}_{m,M}$, take a cutoff $\chi$ supported in $U_r$ and equal to one on a larger neighborhood of $\overline U$. Step 5 bounds the complementary input by $O(h^n)$ for every fixed $n$. On interior sets the localized heat expansion gives}
\begin{gather*}
\blue{q_h=\psi+\frac h\beta\Delta\psi+O_{C^0}(h^2),\qquad q_h-\psi=O_{C^2}(h),}\\
\blue{\G_h(\chi u/q_h)=\chi u/q_h+\frac h\beta\Delta(\chi u/q_h)+O_{C^0}(h^2).}
\end{gather*}
{Here $q_h$ is bounded below and has a uniformly bounded $C^{3,1}$ norm on the cutoff support. The fourth weak derivatives of $u$ and $\chi u/q_h$ are uniformly bounded, so the heat remainder follows by mollification and the Gaussian fourth moment. Since $u/q_h-u/\psi=O_{C^2}(h)$ on the interior, multiplication by $\psi$ yields}
\[
\blue{\K_V^hu=u+\frac h\beta \left\{\psi\Delta(u/\psi)-u\frac{\Delta\psi}{\psi}\right\} +O_{C^0(U)}(h^2)=u+h\mathcal L^*u+O_{C^0(U)}(h^2).}
\]
{The local bounds on $u,V$ control $\mathcal L^*u$, giving $\inf_U\K_V^hu\ge m/2$ for small $h$. Heat-semigroup contraction applied to derivatives of $\chi u/q_h$ through order two controls the $C^2$ norm of the localized contribution. Differentiating the normalized kernel in Step 5 through order two introduces only polynomial factors in $|x-y|$ and $h^{-1}$, which are absorbed by the Gaussian tail because the sets are separated. This proves the uniform $C^2$ bound. The argument is also the $C^0$ case of the derivative expansion recorded in Theorem~\ref{Thm_rho_T}.}
\end{proof}

 \medskip 

\paragraph{Gaussian tails and the stepsize restriction.}
The following calculation illustrates how Gaussian tails restrict the stepsize in the local estimate. Fix $R>0$ and assume that
\[
U \subset B(0,R) \subset B(0,2R) \subset U_r\,.
\]
For any $x\in U$, we have $\|x\|\le R$, and therefore $\|x-y\|\ge R$ for all $y\in U_r^{\,c}$. It follows that
 \[
{\int_{U_r^{\,c}} G_h(x-y)\,dy \;\le\; \int_{\{\|z\|\ge R\}} G_h(z)\,dz \;\le\; C\,\exp\!\Big(-\frac{\beta R^{2}}{8h}\Big)\,,
}\]
where the last inequality is the standard Gaussian tail bound. The estimate holds uniformly for all $x\in U$. Consequently, there exists $h_0(R)\in(0,1]$ such that
 \[
{C\,\exp\!\Big(-\frac{\beta R^{2}}{8h}\Big)\le h^{2}\,, \qquad 0<h\le h_0(R)\,,
}\]
which ensures that $ \int_{U_r^{\,c}} G_h(x-y)\,dy \;\le\; h^2 $ uniformly for $x\in U$. {This calculation bounds the Gaussian tail before normalization. Step 5 of the preceding proof includes the normalization factor when bounding the contribution of input densities outside $U_r$.}

 \begin{proof}[Proof of Theorem~\ref{thm:PV-weak}]
\noindent\emph{Step  {1. Conditional moments}.}  {Let $X$ have conditional density $\gamma_y(x)=\psi(x)G_h(x-y)/q_h(y)$ and put $Z=X-y$. The mode $a_y$ of $\gamma_y$ satisfies $a_y+h\nabla V(a_y)=y$, so $|a_y-y|\le h|\nabla V(y)|/(1+\alpha h)$. The Hessian of $-\log\gamma_y$ lies between $\beta(1+\alpha h)/(2h)$ and $\beta(1+Lh)/(2h)$. Gaussian upper and lower bounds around $a_y$ therefore give, for $1\le p\le4$ }and  {$h\le1$,
}\[
 \blue{\mathbb E_y|Z|^p\le C\{h^{p/2}+h^p|y-x_*|^p\}.
}\]
\noindent\emph{Step  {2. Signed moments by integration by parts}.}  {Write $g=\nabla V$. Integration by parts against $\gamma_y$ gives
}\[
 \blue{\mathbb E_y[Z_i F(X)] =\frac{2h}{\beta}\mathbb E_y[\partial_iF(X)] -h\mathbb E_y[g_i(X)F(X)].
}\]
{With $F=1$ this first implies $|\mathbb E_yZ|\le Ch(1+|y-x_*|)$. Expanding $g(y+Z)=g(y)+\nabla^2V(y)Z+O(|Z|^2)$ and using the bounded third derivative then gives
}\[
 \blue{\mathbb E_yZ=-hg(y)+O\bigl(h^2(1+|y-x_*|^2)\bigr).
}\]
{Apply the same integration-by-parts identity with $F(X)=Z_j$ and $F(X)=Z_jZ_k$. The Lipschitz bound for $g$, the moment bounds above}, and  {the signed first moment imply
}\begin{align*}
 \blue{\mathbb E_y[Z_iZ_j] }&\blue{=\frac{2h}{\beta}\delta_{ij} +O\bigl(h^2(1+|y-x_*|^2)\bigr),}\\
 \blue{|\mathbb E_y[Z_iZ_jZ_k]| }&\blue{\le Ch^2(1+|y-x_*|^3). }\end{align*}
\noindent {\emph{Step 3. Expansion for global test functions.} Taylor-expand $\varphi(y+Z)$ through degree three. Its remainder is bounded by $C\norm{D^4\varphi}_\infty|Z|^4$. Consequently
}\[
 \blue{|P_V^h\varphi(y)-\varphi(y)-h\cL\varphi(y)| \le Ch^2(1+|y-x_*|^4)\norm\varphi_{C_b^4}.
}\]
{Integration against $\dd|u|$ proves the result. This argument also gives~\eqref{eq:global-PV-weak-fixedU-final} for arbitrary finite signed inputs with a finite fourth moment, without positivity or local density bounds. This broader estimate is used in Corollary~\ref{cor:weak_global_BRWP}; the statements in Section~\ref{sec_kernel_approx} use $\mathcal{A}_{m,M}$}. \end{proof}
\paragraph{{\color{black}An auxiliary tail estimate.}}
{We also record the following weighted Gaussian integral estimate. For $p>q_V$, let $c_p=p-q_V$ and
}\[
{W(x,y)=(1+\|x\|^{q_V}+\|y\|^{q_V})(1+\|x\|)^{-p}(1+\|y\|)^{-p}.
}\]
Let  {$R=\max(\|x\|,\|y\|)$. Bounding the numerator by $C(1+R)^{q_V}$ and one denominator factor by $(1+R)^{-p}$ gives
}\[
{W(x,y)\le C(1+R)^{-c_p}\le C'(1+\|x-y\|)^{-c_p}.
}\]
{With $r_h^2=(8+2d+2c_p)h\log(1/h)/\beta$, direct substitution in $G_h$ gives
}\[
{\sup_{\|z\|\ge r_h}G_h(z)(1+\|z\|)^{-c_p}\le Ch^2.
}\]
{Thus if $\int(1+\|x\|)^p|u(x)|dx$ and $\int(1+\|y\|)^p|\varphi(y)|dy$ are at most $M_p$, the corresponding unnormalized tail integral is at most $CM_p^2h^2$, by Fubini. The normalized-kernel tail is controlled separately in Step 5 of the proof of Lemma~\ref{lem:PV-weak-local}}.

 \medskip 

 \begin{proof}[Proof of Theorem~\ref{thm:uniform-C31}]
{Translate $x_*$ to zero. We obtain the regularity of the two density sequences and the exact flow from bounds on their Gaussian parameters. }

 \noindent\emph{Step 1. Uniform Gaussian parameter bounds for BRWP.} {Let $\rho_k=\mathcal N(m_k,C_k)$. With $J=(I+hA)^{-1}$ and $Q_k=JC_kJ+2h\beta^{-1}J$, the kernel output is $\mathcal N(Jm_k,Q_k)$, and the BRWP map is
}\[
\blue{F_k(x)=(I-hA+h\beta^{-1}Q_k^{-1})x-h\beta^{-1}Q_k^{-1}Jm_k.
}\]
{In particular, the mean and covariance of the BRWP output are obtained by the usual affine Gaussian transformation. Proposition~\ref{prop:gaussian} proves, by induction on a sublevel set of $D_h$, that for a sufficiently small $h_0$,
}\[
\blue{|m_k|\le M,\qquad cI\preceq C_k\preceq CI \quad (k\ge0,\ 0<h\le h_0).
}\]

 \noindent\emph{Step 2. The kernel iteration.} {The invariant Gaussian covariance is
}\[
\blue{C_h^*=\frac{2}{\beta}(I+hA)(2I+hA)^{-1}A^{-1}.
}\]
{It satisfies $C_h^*=JC_h^*J+2h\beta^{-1}J$. Hence
}\[
\blue{\widehat m_k=J^km_0,\qquad \widehat C_k=J^kC_0J^k+C_h^*-J^kC_h^*J^k.
}\]
{The term $C_h^*-J^kC_h^*J^k$ is positive semidefinite. Since $C_h^*$ commutes with $J$ and its eigenvalues have fixed positive upper and lower bounds, this formula gives uniform upper and lower bounds for $\widehat C_k$: for a unit vector $v$, combine a lower bound times $|J^kv|^2$ from the first summand with a lower bound times $1-|J^kv|^2$ from this term. The means are also bounded. }

\noindent{\emph{Step 3. The exact Gaussian flow.} After translation by $x_*$, the exact Fokker--Planck law is $\rho(t)=\mathcal N(m_t,C_t)$, with}
\[
\blue{m_t=e^{-At}m_0,\qquad C_t=e^{-At}C_0e^{-At}+\beta^{-1}A^{-1}(I-e^{-2At}).}
\]
{Set $c_-:=\min\{\lambda_{\min}(C_0),(\beta L)^{-1}\}$ and $c_+:=\max\{\lambda_{\max}(C_0),(\beta\alpha)^{-1}\}$. The two covariance terms are bounded below by $c_-e^{-2At}$ and $c_-(I-e^{-2At})$, respectively, and similarly above by the corresponding $c_+$ terms. Thus $c_-I\preceq C_t\preceq c_+I$ and $|m_t|\le|m_0|$ for all $t\ge0$.}

\noindent {\emph{Step 4. Density derivatives and comparison.}} {For $g_{m,C}(x)=(2\pi)^{-d/2}(\det C)^{-1/2}\exp[-(x-m)^{\mathsf T}C^{-1}(x-m)/2]$, each derivative $D^\gamma g_{m,C}$ is a polynomial in $C^{-1}(x-m)$ and the entries of $C^{-1}$, multiplied by $g_{m,C}$. On the compact parameter ranges above these derivatives are uniformly bounded, on any fixed bounded ball, through every fixed order. One additional derivative bounds the Lipschitz seminorm. Therefore both $\widehat\rho_k$ and $\rho_k$ have }uniform $C^{3,1}$  {\color{black}bounds. If $\epsilon_k=\rho_k-\widehat\rho_k$, then
}\[
\blue{\|\epsilon_k\|_{C^{3,1}(U)}\le\|\rho_k\|_{C^{3,1}(U)}+\|\widehat\rho_k\|_{C^{3,1}(U)}\le C.
}\]
{Apply the same derivative bounds on the fixed ball $U_r$ to all three Gaussian families. The Gaussian density is continuous and strictly positive on the compact set of parameters and $x\in\overline U_r$, so it has a common positive minimum $m$. For $X\sim\mathcal N(\mu,C)$,
}\[
\blue{\mathbb E|X|^4=|\mu|^4+2|\mu|^2\operatorname{tr}C +4\mu^{\mathsf T}C\mu+(\operatorname{tr}C)^2+2\operatorname{tr}(C^2).}
\]
{The compact parameter bounds therefore give a uniform fourth moment. Choosing $M$ to bound the sum of this moment and the local $C^{3,1}$ norm places $\rho_k,\widehat\rho_k,\rho(t_k)$ in $\mathcal{A}_{m,M}$. The quadratic potential has $D^3V=0$, and exact inputs have zero perturbation, proving Assumption~\ref{assump:kernel-hypotheses} for all time. }\end{proof}

\begin{proof}[Proof of Theorem~\ref{thm:weak-score-drift}]
Fix  {$k$ on the stated horizon and $w\in C_c^{1,1}(U;\R^d)$. Assumption~\ref{assump:kernel-hypotheses} and Lemma~\ref{lem:PV-weak-local} give uniform positive lower bounds and $C^2(U)$ upper bounds for $\rho_k$, $\K_V^h\rho_k$, and $\K_V^h\tilde\rho_k$. }

 For this proof, set
\[
r:=\rho_k,\qquad r_K:=\K_V^h\rho_k,\qquad r^{(1)}:=r+h\,\mathcal L^*r\,.
\]
Then we have
\begin{align}
&\int w\!\cdot\!\Bigl[\nabla\log(\K_V^h\tilde\rho_k) - \nabla\log\rho_k-h\nabla\!\Bigl(\tfrac{\mathcal L^*\rho_k}{\rho_k}\Bigr)\Bigr]\rho_k\,dx\\
 =&\int w\!\cdot\!\Bigl[\nabla\log r_K-\nabla\log r-h\nabla\!\Bigl(\tfrac{\mathcal L^*\rho_k}{\rho_k}\Bigr)\Bigr]\rho_k\,dx +\int w\!\cdot\bigl[\nabla\log(\K_V^h\tilde\rho_k)-\nabla\log(\K_V^h\rho_k)\bigr]\rho_k\,dx\,. \notag \label{eq:I-decomp-fixed}
\end{align}
Denote the first and second terms by $I^{0}$ and $I^{1}$, respectively.

\noindent\emph{Step 1. Error at the exact input, $I^{0}$.} Insert and subtract $\nabla\log r^{(1)}$:
\begin{align}
I^{0}&= \int w\!\cdot(\nabla\log r_K-\nabla\log r^{(1)})\,r\,dx\;+\; \int w\!\cdot \Bigl(\nabla\log r^{(1)}- \nabla\log r -h\nabla\!\Bigl(\tfrac{\mathcal L^*r}{r}\Bigr)\Bigr)\,r\,dx \\
&=: I^{0}_{a}+I^{0}_{b}\,.\notag \label{eq:I0-split-clean}
\end{align}

Using
\[
\nabla\log r_K-\nabla\log r^{(1)} = \frac{\nabla(r_K-r^{(1)})}{r^{(1)}} + \nabla r_K\,\frac{r^{(1)}-r_K}{r_Kr^{(1)}}\,,
\]
we write
\begin{align*}
I^{0}_{a} = &\int w\!\cdot\!\frac{\nabla(r_K-r^{(1)})}{r^{(1)}}\,r\,dx + \int w\!\cdot\nabla r_K \frac{r^{(1)}-r_K}{r_K r^{(1)}}\,r\,dx \\
=&-\!\int \nabla\!\cdot\!\Bigl(\frac{wr}{r^{(1)}}\Bigr)(r_K-r^{(1)})\,dx- \ \!\int\Bigl(w\!\cdot\nabla r_K\,\frac{r}{r_K r^{(1)}}\Bigr)(r_K-r^{(1)})\,dx\,.
\end{align*}

Since $r,r_K,r^{(1)}\in C^{1,1}(U)$ and $r^{(1)},r_K\ge c>0$, both coefficients
\[
\nabla\!\cdot\!\Bigl(\frac{wr}{r^{(1)}}\Bigr), \qquad w\!\cdot\nabla r_K\,\frac{r}{r_K r^{(1)}}
\]
lie in \(C^{0,1}(U)\) with norms
\[
\left\|\nabla\!\cdot\!\Bigl(\frac{wr}{r^{(1)}}\Bigr)\right\|_{C^{0,1}(U)} + \left\|w\!\cdot\nabla r_K\,\frac{r}{r_K r^{(1)}}\right\|_{C^{0,1}(U)} \le C\,\|w\|_{C^{1,1}(U)}\,.
\]

 {The local density estimate in Lemma~\ref{lem:PV-weak-local} supplies $\|r_K-r^{(1)}\|_{C^0(U)}\le Ch^2$. Pairing this bound with the displayed integrable coefficients }implies
\begin{equation}\label{eq:I0a-done}
|I^{0}_{a}| \le C h^2\,\|w\|_{C^{1,1}(U)}\,.
\end{equation}

For $I^{0}_{b}$, Taylor expansion of \( \log(r+\theta h\,\mathcal L^*r) \) in $\theta$ yields
\[
\nabla\log r^{(1)} - \nabla\log r = h\,\nabla\!\Bigl(\tfrac{\mathcal L^*r}{r}\Bigr) +h^2 R(x)\,, \qquad \|R\|_{C^0(U)}\le C\,.
\]
Thus
\begin{equation}\label{eq:I0b-done}
|I^{0}_{b}| \le C h^2 \|w\|_{L^\infty(U)} \le C h^2 \|w\|_{C^{1,1}(U)}\,.
\end{equation}

From \eqref{eq:I0a-done}–\eqref{eq:I0b-done},
\begin{equation}\label{eq:I0-final}
|I^{0}| \le C h^2\|w\|_{C^{1,1}(U)}\,.
\end{equation}

\noindent\emph{Step 2. Input perturbation $I^{1}$.} {For $g=\tilde\rho_k-\rho_k\in C_c^{2,1}(U_0)$, the zero extension of $g/q_h$ has $C^1$ norm bounded by $C\|g\|_{C^{2,1}(U_0)}$. Applying heat-semigroup contraction to $\K_V^hg=\psi\G_h(g/q_h)$ gives
}\begin{equation}{\label{eq:Kh-strong}
\|\K_V^h\tilde\rho_k-\K_V^h\rho_k\|_{C^1(U)} \le C\|\tilde\rho_k-\rho_k\|_{C^{2,1}(U_{0})}\,. }\end{equation}
{For positive $f,g$ with a common lower bound $c$ and bounded $C^1$ norms,
}\[
{\nabla\log f-\nabla\log g =\frac{\nabla(f-g)}f+\nabla g\frac{g-f}{fg};
}\]
{hence
}\begin{equation}{\label{eq:score-Lip}
\|\nabla\log(\K_V^h\tilde\rho_k) -\nabla\log(\K_V^h\rho_k)\|_{C^0(U)} \le C'\,\|\K_V^h\tilde\rho_k-\K_V^h\rho_k\|_{C^1(U)}\le C\,\|\tilde\rho_k-\rho_k\|_{C^{2,1}(U_{0})}\,. }\end{equation}

Thus
 \begin{equation}{\label{eq:I1-done}
|I^{1}| \le C\|w\|_{L^\infty(U)}\|\tilde\rho_k-\rho_k\|_{C^{2,1}(U_{0})} \le C\|w\|_{C^{1,1}(U)}\|\tilde\rho_k-\rho_k\|_{C^{2,1}(U_{0})}\,. }\end{equation}

Combining \eqref{eq:I0-final} and \eqref{eq:I1-done} gives the desired result.

 \end{proof}

\subsection{Auxiliary estimates}

We collect moment, heat-kernel, stability, positivity, and comparison estimates for the kernel analysis.

\subsubsection{Bounds on \texorpdfstring{$p$}{p}th moments}\label{subsec:p-moment}
\blue{Let $x_*=\argmin V$. For the BRWP-chain part of the next lemma, consider a chain of probability densities and assume the following bound on score moments over the stated horizon:
}\begin{equation}\blue{\label{eq:moment-score-condition}
\sup_{h,k}\int\left|\nabla\log\G_h(\rho_k/q_h)(x)\right|^p\rho_k(x)\,dx\le M_p^{\rm sc}<\infty. }\end{equation}
\blue{This condition is verified under the Gaussian conditions of Theorem~\ref{thm:uniform-C31} by the affine Gaussian score and compact parameter bounds. The kernel-iteration bound does not require~\eqref{eq:moment-score-condition}. All norms in the moment recurrence below are centered at $x_*$. }

 \oldresult{15}
\begin{lemma}[Uniform discrete $p$th-moment bound for the BRWP iteration]
\label{lem:moment-p-discrete}
Let $p \ge 2$\blue{, assume a finite initial $p$th moment, and impose~\eqref{eq:moment-score-condition} for the BRWP bounds}. Then there exist constants $h_0\in(0,1]$, $a_p,c_p\in(0,1)$, and $b_p,d_p<\infty$, depending only on  \blue{$(\alpha,\beta,L,d,p,M_p^{\rm sc})$}, such that for all $h\in(0,h_0]$ and all \blue{indices covered by~\eqref{eq:moment-score-condition} for the BRWP iteration (and all }$k\ge0$  \blue{for the kernel chain),
}\begin{equation}{\label{eq:disc-moment-p-transport}
\int_{\R^d}\|x-x_*\|^p\,\rho_{k+1}(x)\,dx \;\le\; (1-a_p h)\!\int_{\R^d}\|y-x_*\|^p\,\rho_k(y)\,dy\;+\;b_p h\,, }\end{equation}
and
 \begin{equation}{\label{eq:disc-moment-p-kernel}
\int_{\R^d}\|x-x_*\|^p\,(\K_V^h\widehat\rho_k)(x)\,dx \;\le\; (1-c_ph)\!\int_{\R^d}\|y-x_*\|^p\,\widehat\rho_k(y)\,dy\;+\;d_p h\,. }\end{equation}

Consequently,  \blue{if~\eqref{eq:moment-score-condition} holds for all iterations,
}\begin{equation}{\label{eq:moment-uniform-transport}
\sup_{k\ge0}\int_{\R^d}\|x-x_*\|^p\,\rho_k(x)\,dx \;\le\;\max\!\Big\{\int\|x-x_*\|^p\rho_0(x)\,dx,\,\frac{b_p}{a_p}\Big\} <\infty\,, }\end{equation}
and
 \begin{equation}
{\sup_{k\ge0}\int_{\R^d}\|x-x_*\|^p\,\widehat\rho_k(x)\,dx \;\le\;\max\!\Big\{\int\|x-x_*\|^p\rho_0(x)\,dx,\,\frac{d_p}{c_p}\Big\} <\infty\,. }\end{equation}
\end{lemma}

\begin{proof}
Fix $p\ge2$. Constants denoted by $C$ may vary from line to line and depend only on  {$(\alpha,\beta,L,d,p,M_p^{\rm sc})$}.

\noindent\emph{Step 1. Score decomposition and identities.} We begin with the identity
 \[
{\log(\K_V^h\rho_k) = \log \psi + \log\!\big(\mathcal G_h(q_h^{-1}\rho_k)\big)\,.
}\]
Since $\nabla\log \psi = -(\beta/2)\nabla V$, it follows that
 \[
{\nabla V + \beta^{-1}\nabla\log(\K_V^h\rho_k) = \tfrac12\,\nabla V + \beta^{-1}\nabla\log\!\big(\mathcal G_h(q_h^{-1}\rho_k)\big)\,,
}\]
and therefore
 \begin{equation}{\label{eq:Fk-decompose}
F_k(x) = x - \frac{h}{2}\,\nabla V(x) - h\beta^{-1}\nabla\log\!\big(\mathcal G_h(q_h^{-1}\rho_k)\big)(x)\,. }\end{equation}

For nonzero $f\ge0$, suppose $\int_{\R^d}(1+\|y\|)G_h(z-y)f(y)\,dy<\infty$ for all $z$ in a neighborhood of $x$. Differentiating the Gaussian convolution then gives
\[
\nabla\log(\mathcal G_h f)(x) = \frac{-\beta}{2h}\, \frac{\displaystyle\int_{\R^d} (x-y)\,G_h(x-y)\,f(y)\,dy}{\displaystyle\int_{\R^d} G_h(x-y)\,f(y)\,dy} = -\beta\,\frac{x - m_f(x)}{2h}\,,
\]
where
\[
m_f(x) := \frac{\displaystyle\int_{\R^d} y\,G_h(x-y)\,f(y)\,dy}{\displaystyle\int_{\R^d} G_h(x-y)\,f(y)\,dy}
\]
is the mean weighted by the Gaussian kernel.  Applying this with  {$f=q_h^{-1}\rho_k$ gives
}\begin{equation}{\label{eq:Fk-explicit}
F_k(x) = x - \frac{h}{2}\nabla V(x) + \frac{1}{2}\big(x - m_{\rho_k/q_h}(x)\big)\,. }\end{equation}

Using Cauchy--Schwarz under the normalized weight proportional to {$G_h(x-y)\,q_h^{-1}(y)\rho_k(y)$,
}\[
{\|x - m_{\rho_k/q_h}(x)\|^2 \le \frac{\displaystyle\int_{\R^d} \|x-y\|^2 G_h(x-y)\,q_h^{-1}(y)\rho_k(y)\,dy}{\displaystyle\int_{\R^d} G_h(x-y)\,q_h^{-1}(y)\rho_k(y)\,dy}\,.
}\]

 {The score-moment bound~\eqref{eq:moment-score-condition} and the displayed identity give the }following averaged estimate:
 \begin{equation}
{\label{eq:barycenter-dev-avg} \int_{\R^d}\|x-m_{\rho_k/q_h}(x)\|^p\,\rho_k(x)\,dx \;\le\; C\,h^p\,, \qquad p\ge2\,, }\end{equation}
where $C$ depends  {also on $M_p^{\rm sc}$}.

\noindent\emph{Step 2. Drift contraction.} Set  {$\Lambda_k(x) := \beta(x - m_{\rho_k/q_h}(x))$. Combining \eqref{eq:Fk-explicit} and \eqref{eq:barycenter-dev-avg} gives
}\[
{F_k(x) = x - \frac{h}{2}\nabla V(x)
  + \frac{1}{2\beta}\Lambda_k(x)\,,\quad \int \|\Lambda_k(x)\|^p\rho_k(x)dx \leq Ch^p\,.}\]

Let  {$x_*$ }denote a minimizer of $V$, so  {$\nabla V(x_*)=0$}. By the mean-value theorem,
 \[
{\nabla V(x) - \nabla V(x_{*}) = H_\xi(x-x_{*})\,, \quad H_\xi:=\int_0^1 \nabla^2V(x_{*}+t(x-x_{*}))\,dt\,.
}\]
Thus,
 \[
{x - \tfrac{h}{2}\nabla V(x) - x_{*} = \big(I - \tfrac{h}{2}H_\xi\big)(x-x_{*})\,.
}\]

Since $\alpha I_d \preceq H_\xi \preceq L I_d$  {globally}, choose $h_0\le \min\{1,2/L\}$ so that
\[
\|I - \tfrac{h}{2}H_\xi\| \le 1 - \tfrac{\alpha h}{2}\,.
\]
Therefore,
 \begin{equation}{\label{eq:drift-contraction}
\big\|x - \tfrac{h}{2}\nabla V(x) - x_{*}\big\| \le (1 - \tfrac{\alpha h}{2})\,\|x-x_{*}\|\,. }\end{equation}

Raising the inequality to the $p$th power and using  {$(1-\theta)^p\le1-\theta$ for $0\le\theta\le1$ }shows that there exists $a_p>0$ such that
 \begin{equation}{\label{eq:drift-contraction-p}
\big\|x - \tfrac{h}{2}\nabla V(x) - x_{*}\big\|^p \le (1-a_p h)\|x-x_{*}\|^p\,. }\end{equation}

\noindent\emph{Step 3. Adding the fluctuation.} Write
 \[
{F_k(x) - x_{*} = U(x) + W(x)\,, \qquad U(x) := x - \tfrac{h}{2}\nabla V(x) - x_{*},\quad W(x) := \tfrac{1}{2\beta}\Lambda_k(x)\,.
}\]
For $p\ge 2$ and $\varepsilon\in(0,1]$, Young's inequality gives
\[
\|U + W\|^p \le (1+\varepsilon)\|U\|^p + C_p\varepsilon^{1-p}\|W\|^p\,.
\]
Choose  {$\varepsilon=a_ph/4$ and decrease $h_0$ if necessary. Integrating }against $\rho_k$ and using  {the drift estimate and~\eqref{eq:moment-score-condition} gives
}\begin{align*}
{\int|F_k(x)-x_*|^p\rho_k(x)dx }&{\le(1+a_ph/4)(1-a_ph)\int|x-x_*|^p\rho_k(x)dx +C_ph^{1-p}h^p}\\
&{\le(1-a_ph/2)\int|x-x_*|^p\rho_k(x)dx+C_ph. }\end{align*}
{Renaming $a_p/2$ as $a_p$ yields~\eqref{eq:disc-moment-p-transport}. Iterating the scalar inequality proves~\eqref{eq:moment-uniform-transport} on the stated horizon, and for all time when~\eqref{eq:moment-score-condition} holds for all iterations}.

\noindent\emph{Step 4. $p$th moment for the kernel iteration.} {Translate $x_*$ to zero and let $X$ have the conditional kernel law $\mu_y$. Its logarithmic derivative is $-\beta(\nabla V(X)+(X-y)/h)/2$. Integration by parts with the vector field $|x|^{p-2}x$ gives
}\[
{\E_y[|X|^{p-2}X\cdot(X-y+h\nabla V(X))] =\frac{2h}{\beta}(d+p-2)\E_y|X|^{p-2}.
}\]
{For $p=2$ this is immediate; for other $p\ge2$ it follows by a smooth approximation at zero and a cutoff at infinity. The conditional Gaussian tails justify the limits. Strong convexity and Young's inequality imply
}\[
{(1+p\alpha h)\E_y|X|^p \le |y|^p+\frac{2ph}{\beta}(d+p-2)\E_y|X|^{p-2}.
}\]
{Absorb the last moment using $|x|^{p-2}\le\epsilon|x|^p+C_{p,\epsilon}$, choosing $\epsilon$ so that its coefficient is at most $p\alpha h/2$. It follows that
}\[
{\E_y|X|^p\le\frac{|y|^p+C_ph}{1+p\alpha h/2} \le(1-c_ph)|y|^p+d_ph
}\]
for  {suitable positive $c_p,d_p$ on $0<h\le h_0$. Translate back, integrate against $\widehat\rho_k$, and use duality to obtain~\eqref{eq:disc-moment-p-kernel}. Iteration gives the asserted uniform kernel moments. }\end{proof}

 \paragraph{{\color{black}Second moments from entropy.}}
\blue{Every law with finite KL divergence satisfies
}\[
\blue{\int|x-x_*|^2\rho(dx)\le\frac{4}{\beta\alpha}\KL(\rho\|\rho^*)+\frac{2d}{\beta\alpha}.
}\]
\blue{Indeed, integration by parts under $\rho^*$ gives $\beta\alpha\int|x-x_*|^2\rho^*\le d$. Couple $\rho$ and $\rho^*$, use $|x-x_*|^2\le2|x-y|^2+2|y-x_*|^2$, and apply Talagrand's inequality. Theorem~\ref{Thm_KL_decay} therefore gives a uniform second-moment bound on its stated horizon. Under the Gaussian conditions of Theorem~\ref{thm:uniform-C31}, all fixed-order moments also follow directly from compact Gaussian parameters. }\subsubsection{{Gaussian heat-kernel estimates}}\label{subsec:Gaussian_estimation} {The following estimates concern interior derivatives and boundary decay. }\oldresult{16}
\begin{lemma}[\blue{Interior Dirichlet heat-kernel estimates}]
\label{lem:Dirichlet-Gaussian-derivative}
Let \blue{$U\Subset U_r\subset\R^d$ be bounded open sets, with $U_r$ having smooth boundary, and put $r_0=\dist(U,\partial U_r)>0$. Let }$G_h^D$ be the Dirichlet heat kernel on $U_r$ \blue{with diffusivity $\beta^{-1}$, and set $s=h/\beta$}. There exist  \blue{$C,h_0>0$, depending on $d,\beta,U,U_r$}, such that for \blue{$0<h\le h_0$, }$x\in U$, $y\in U_r$ \blue{, and every multi-index $\gamma$ with $0\le|\gamma|\le4$,
}\begin{align}
 \blue{\left|\partial_x^\gamma\{G_h^D(x,y)-G_h(x-y)\}\right| }&\blue{\le C s^{-(d+|\gamma|)/2} e^{-r_0^2/(Cs)}e^{-|x-y|^2/(Cs)},\label{eq:killed-vs-free}}\\
 \blue{|\partial_x^\gamma G_h^D(x,y)| }&\blue{\le C s^{-(d+|\gamma|)/2}e^{-|x-y|^2/(Cs)}. \label{eq:Dirichlet-pointwise}
}\end{align}
\blue{Consequently }every finite signed measure $\nu$ on $U_r$  \blue{satisfies
}\begin{equation}\blue{\label{eq:Dirichlet-derivative-measure}
 \int_{U_r}|\partial_x^\gamma G_h^D(x,y)|\,\dd|\nu|(y) \le C s^{-(d+|\gamma|)/2} \int_{U_r}e^{-|x-y|^2/(Cs)}\,\dd|\nu|(y). }\end{equation}
\end{lemma}
 \begin{proof}
 {Work first at diffusivity one and time $s$. Write $p_s$ and $p_s^D$ for the free and killed kernels. The Brownian-bridge representation is
}\[
 {0\le p_s(x-y)-p_s^D(x,y) =p_s(x-y)\,\mathbb P^{x,y}_s\{\text{the bridge exits }U_r\}.
}\]
{For $|x-y|<r_0/2$, the segment joining $x$ to $y$ lies in $B(x,r_0/2)\subset U_r$. Exit requires the centered bridge to have supremum norm at least $r_0/2$. The representation $\sqrt2(B_t-(t/s)B_s)$ of a centered bridge and the coordinatewise reflection bound for Brownian motion show that this probability is at most $C e^{-r_0^2/(Cs)}$. For $|x-y|\ge r_0/2$, use the probability bound one and split the free Gaussian exponent into two halves. In both cases,
}\[
 {0\le p_s(x-y)-p_s^D(x,y) \le C s^{-d/2}e^{-r_0^2/(Cs)}e^{-|x-y|^2/(Cs)}.
}\]
{This bounds the effect of the boundary at interior points; see also \citet{li2016heat}. }

 Fix  {$y$. Both the difference above and $p_s^D(x,y)$ solve the heat equation in $(s,x)$ for positive time. Apply the interior caloric derivative estimate on a cylinder of spatial radius $\sqrt{s}/8$ and time depth $s/64$ ending at $(s,x)$:
}\[
 {|\partial_x^\gamma w(s,x)| \le C_\gamma s^{-|\gamma|/2}\sup_{Q_s}|w|.
}\]
{Choose $h_0$ so that this cylinder stays at distance at least $r_0/2$ from the boundary. In it, times are comparable to $s$ }and {$|z-y|^2\ge\tfrac12|x-y|^2-|z-x|^2$; the preceding exponential bound is therefore uniform on the cylinder after changing }$C$ {. It yields \eqref{eq:killed-vs-free}. For $w=p^D$, use $0\le p^D\le p$ on the same cylinder to obtain \eqref{eq:Dirichlet-pointwise}. Restore $s=h/\beta$ and integrate against the total variation measure to prove the last assertion. }\end{proof}
\oldresult{18}
\subsubsection{Stability and contraction of the kernel formula}
\label{subsec:stability}
The next two lemmas bound the action of the kernel operators on test functions and input densities. The gradient estimate also gives the Wasserstein contraction used in Theorem~\ref{thm:oracle-C01-test}.
\begin{lemma}[Regularization and contraction of $P_V^h$]
\label{lem:PV-regularization}
 \blue{The first estimate holds for bounded functions and the second for globally Lipschitz functions, with global norms in both cases. For the third estimate, let $\varphi\in C_c^{3,1}(U_0)$, extended by zero; its constant is uniform for sufficiently small $h$ and depends on the fixed local potential norms and separated sets.
}\begin{align}
{\|P_V^h\varphi\|_{L^\infty(\R^d)}
}&{\le \|\varphi\|_{L^\infty(\R^d)}\,, \label{eq:PVT-Linf}}\\[3pt]
{\|\nabla P_V^h\varphi\|_{L^\infty(\R^d)} }&{\le \frac{1}{1+\alpha h}\,\|\nabla\varphi\|_{L^\infty(\R^d)} \, ,
\label{eq:PVT-grad-unified}}\\[3pt]
{\|P_V^h\varphi\|_{C^{3,1}(U)} }&{\le C\,\|\varphi\|_{C^{3,1}(U_{0})}\,. \label{eq:PVT-C31}
}\end{align}
\end{lemma}

\begin{proof}
\noindent\emph{Step 1. $L^\infty$ contraction and strong log-concavity.} For each fixed  {$y\in\R^d$}, \( P_V^h\varphi(y)=\E_{\mu_y}[\varphi(X)] \) is an average of $\varphi$ against a probability measure. Hence
 \[
{|P_V^h\varphi(y)| \le \|\varphi\|_{L^\infty(\R^d)}\,,\qquad y\in\R^d\,.
}\]

The density of $\mu_y$ is proportional to $e^{-U_y(x)}$, where
\[
\nabla_x^2 U_y(x) =\tfrac{\beta}{2}\big(\nabla^2 V(x)+h^{-1}I_d\big) \succeq \kappa I_d\,,\qquad \kappa:=\tfrac{\beta}{2}\Big(\alpha+\frac{1}{h}\Big)\,.
\]
Thus $\mu_y$ is $\kappa$-strongly log-concave. By the Poincaré inequality,
 \begin{equation}{\label{eq:Poincare}
\Var_{\mu_y}(f) \;\le\;\frac{1}{\kappa}\,\E_{\mu_y}\|\nabla f(X)\|^2\,, \qquad f,\nabla f\in L^2(\mu_y)\,. }\end{equation}
In particular, for any unit vector $u\in\R^d$,
\[
\Var_{\mu_y}(\langle X-y,u\rangle) \;\le\;\frac{1}{\kappa}\,, \qquad \kappa=\tfrac{\beta}{2}\Big(\alpha+\frac{1}{h}\Big)\,.
\]

\noindent\emph{Step 2. Score identity and gradient contraction.} Recall the score identity (differentiate under the integral):
 \begin{equation}{\label{eq:score_exp_identity}
\nabla_y \E_{\mu_y}[g(X)] =\tfrac{\beta}{2h}\,\Cov_{\mu_y}\!\big(g(X),X-y\big)\,, \qquad g\in C^1(\R^d)\text{ globally Lipschitz}\,. }\end{equation}
For any unit vector $u$,
\[
D_u \E_{\mu_y}[g(X)] =\tfrac{\beta}{2h}\,\Cov_{\mu_y}\big(g(X),\langle X-y,u\rangle\big)\,.
\]
Cauchy--Schwarz and the Poincaré inequality \eqref{eq:Poincare} give
 \[
{\big|\Cov_{\mu_y}(g,\langle X-y,u\rangle)\big| \le \sqrt{\Var_{\mu_y}(g)\,\Var_{\mu_y}(\langle X-y,u\rangle)} \le \kappa^{-1}\,\sqrt{\E_{\mu_y}\|\nabla g(X)\|^2}\le \kappa^{-1}\|\nabla g\|_{L^{\infty}(\R^d)}\,.
}\]
Thus
 \[
{|D_u \E_{\mu_y}[g(X)]| \le \tfrac{\beta}{2h}\,\kappa^{-1}\,\|\nabla g\|_{L^\infty(\R^d)} =\frac{1}{1+\alpha h}\,\|\nabla g\|_{L^\infty(\R^d)}\,.
}\]

Applying this to $g=\varphi$ and taking the supremum over all unit vectors $u$ gives precisely
 \[
{\|\nabla P_V^h\varphi\|_{L^\infty(\R^d)} =\|\nabla_y\E_{\mu_y}[\varphi(X)]\|_{L^\infty(\R^d)} \le \frac{1}{1+\alpha h}\,\|\nabla\varphi\|_{L^\infty(\R^d)}\,,
}\]
which is \eqref{eq:PVT-grad-unified}.

\noindent\emph{Step 3. Higher regularity.} {For the compactly supported test function $\varphi$, write $P_V^h\varphi=\G_h(\psi\varphi)/q_h$. The zero extension of $\psi\varphi$ has $C^{3,1}(\R^d)$ norm at most $C\|\varphi\|_{C^{3,1}(U_0)}$. Convolution with a Gaussian probability density contracts every derivative supremum norm and the Lipschitz seminorm of the highest derivative. On $U$, $q_h$ is uniformly positive and has a bounded $C^{3,1}$ norm, by localized heat convolution. Product and quotient rules therefore give
}\[
{\|P_V^h\varphi\|_{C^{3,1}(U)}\le C\|\varphi\|_{C^{3,1}(U_{0})},
}\]
{which is~\eqref{eq:PVT-C31}. }

 {Finally, the conditional mean has at most linear growth and conditional variances are bounded by the Poincar\'e estimate. Thus the kernel preserves finite first moments. Kantorovich--Rubinstein duality and the global gradient estimate imply
}\begin{equation}\blue{\label{eq:W1-kernel-contraction}
W_1(\K_V^h\mu,\K_V^h\nu)\le(1+\alpha h)^{-1}W_1(\mu,\nu). }\end{equation}
{The contraction acts on the Lipschitz seminorm; constants are preserved because $P_V^h1=1$.}
\end{proof}

\begin{lemma}[Local $C^{3,1}$ stability of $\K_V^h$]\label{lem:Kv-rhok-C31}
There exist $h_0\in(0,1]$ and $C<\infty$, depending only on
 \[
{(d,\beta,\alpha,\|V\|_{C^{4,1}(U_r)},U,U_0,U_r)\,,
}\]
such that for all $h\in(0,h_0]$ and all  \blue{$f\in C_c^{3,1}(U_0)$, extended by zero,
}\begin{equation}{\label{eq:KV-smoothing-C31}
\|\K_V^h f\|_{C^{3,1}(U)} \le C\,\|f\|_{C^{3,1}(U_{0})}\,. }\end{equation}
\end{lemma}

\begin{proof}
By the kernel representation \eqref{eq:KVT-kernel},
 \[
{\K_V^h f(x) = \psi(x) \int_{\R^d} G_h(x-y)\,\frac{f(y)}{q_h(y)}\,dy\,, \qquad x\in U_r \, .
}\]
Assumption~\ref{assump:Regularity} gives $\psi^{\pm1}\in C^{4,1}(U_r)$. {To control $q_h$ on the support of $f$, choose $U_0\Subset U_1\Subset U_r$ and $\chi\in C_c^\infty(U_r)$ with $0\le\chi\le1$ and $\chi=1$ on $U_1$. Then
}\[
{q_h=\G_h(\chi\psi)+\G_h((1-\chi)\psi).
}\]
{The zero extension of $\chi\psi$ has a bounded $C^4(\R^d)$ norm, so its heat convolution has uniformly bounded derivatives through order four. Strong convexity bounds $\|\psi\|_\infty$ using $V$ and $\nabla V$ at any fixed point in $U_r$: $\inf V\ge V(y)-|\nabla V(y)|^2/(2\alpha)$. Since $U_0$ is separated from the support of $1-\chi$, differentiation of the Gaussian gives, for $0\le m\le4$,
}\[
{\|D^m\G_h((1-\chi)\psi)\|_{L^\infty(U_0)} \le C_m h^{-N_m}e^{-c/h},
}\]
{where the constants depend }only on the listed parameters  {and the fixed cutoff. Finally, a fixed-radius ball about each point of $U_0$ lies in $U_1$; integrating the positive local lower bound for $\psi$ over that ball gives $q_h\ge c_0>0$ on $U_0$ for $h\le h_0$. Thus $q_h^{\pm1}$ have uniform $C^{3,1}(U_0)$ bounds. }

{The quotient $f/q_h$ has compact support in $U_0$ and its zero extension has the stated global regularity. }Let $|\lambda|\le3$.  Using $\partial_{x_i}G_h(x-y)=-\partial_{y_i}G_h(x-y)$, we repeatedly integrate by parts in $y$ to transfer all $x$-derivatives from $G_h$ onto  {$f/q_h$}.  This yields the representation
 \[
{\partial_x^\lambda \K_V^h f(x) = \sum_{\mu\le\lambda} a_{\lambda,\mu}(x) \int_{\R^d} G_h(x-y)\, D_y^{\lambda-\mu}\!\Bigl(\frac{f(y)}{q_h(y)}\Bigr)\,dy\,,
}\]
where each coefficient $a_{\lambda,\mu}$ is a linear combination of $\partial_x^\nu \psi(x)$ for $|\nu|\le|\lambda|$.  Hence $\|a_{\lambda,\mu}\|_{C^{0,1}(U)}\le C$ since $\psi\in C^{4,1}(U_r)$.

Because $\int_{\R^d} G_h(x-y)\,dy=1$ and $x\in U\Subset U_r$, we obtain
 \[
{\sup_{x\in U}\bigl|\partial_x^\lambda \K_V^h f(x)\bigr| \le C\,\sup_{y\in U_{0}}\max_{|\gamma|\le|\lambda|} \bigl|D_y^\gamma (f(y)/q_h(y))\bigr| \le C\,\|f\|_{C^{3,1}(U_{0})}\,.
}\]

The same formula, applied to difference quotients in $x$, yields uniform Lipschitz bounds for all derivatives up to order $3$ on $U$, again bounded by  {$C\,\|f\|_{C^{3,1}(U_0)}$}. Thus \eqref{eq:KV-smoothing-C31} follows.
\end{proof}

\subsubsection{Local positivity for Gaussian iterates}
\label{subsec:local-positivity}

{The next lemma gives uniform positive lower bounds for Gaussian iterates. Its proof also gives the Gaussian map comparison used in Corollary~\ref{cor:KV_non_refreshed}. The general score estimate in Theorem~\ref{thm:weak-score-drift} uses the positivity supplied by Assumption~\ref{assump:kernel-hypotheses} and Lemma~\ref{lem:PV-weak-local}.}

 \begin{lemma}
\label{lem:local-positivity}
 \blue{Under the same Gaussian conditions as Theorem~\ref{thm:uniform-C31}, there }exist $h_0\in(0,1]$ and $c_{U,h_0}>0$ such that, for all $0<h\le h_0$ and all $k\ge0$,
\begin{equation}\label{eq:disc-pos}
\inf_{x\in U}\widehat\rho_k(x)\;\ge\;c_{U,h_0}\,, \qquad \inf_{x\in U}\rho_k(x)\;\ge\;c_{U,h_0}\,.
\end{equation}
\end{lemma}

\begin{proof}
We first bound the exact Fokker--Planck density below, then transfer this bound to the discrete iterates using the Gaussian comparison estimate.

\noindent\emph{(i) Fokker--Planck flow.} By standard parabolic regularity and strict ellipticity of $\beta^{-1}\Delta-\nabla V\cdot\nabla$, the Fokker--Planck solution $\rho(t,\cdot)$ is $C^{2,1}$ on $U_r$ for each $t>0$, and the map $(t,x)\mapsto\rho(t,x)$ is continuous on $[0,T]\times\overline U$ for every $T>0$.

 {The initial Gaussian density is continuous and strictly positive on $\overline U$, so }there exists $c_0>0$ such that
 \[
{\inf_{x\in\overline U}\rho_0(x)\;\ge\;c_0\,.
}\]
Since $\rho(t,x)>0$ for all $t>0$ and $x\in U$ \citep{evans2022partial}, and $(t,x)\mapsto\rho(t,x)$ is continuous on $[0,h_0]\times\overline U$ for any fixed $h_0>0$, the infimum on this compact set is strictly positive:
\[
c_1 := \inf_{0\le t\le h_0}\,\inf_{x\in U}\rho(t,x)\;>\;0\,.
\]
 {For the remaining times, the continuous second-moment calculation gives
}\[
{\int|y-x_*|^2\rho(t,y)dy\le e^{-2\alpha t}\int|y-x_*|^2\rho_0(y)dy+\frac{d}{\beta\alpha}.
}\]
{Choose a fixed ball $B$ carrying at least half the mass at every time. The transition density $p_1(y,x)$ of this uniformly elliptic diffusion is positive and continuous, so its infimum $m_B$ on $\overline B\times\overline U$ is positive. For $t\ge1$ the semigroup property gives $\rho(t,x)\ge m_B/2$ on $U$. }Combining this with  {positivity and continuity on $[0,1]\times\overline U$ yields
}\begin{equation}{\label{eq:FP-pos}
\inf_{t\ge0}\inf_{x\in U}\rho(t,x)\ge c_U>0. }\end{equation}

\noindent\emph{(ii) Discrete iterates.} We now transfer {~}\eqref{eq:FP-pos} to the discrete chains  \blue{under the same Gaussian conditions. On their compact Gaussian parameter sets, the BRWP map, kernel map and exact Fokker--Planck map have bounded second derivatives in $h$ and identical first derivatives at $h=0$:
}\[
\blue{\dot m=-Am,\qquad \dot C=-AC-CA+2\beta^{-1}I.
}\]
\blue{The matrix square root is smooth on a compact subset of the positive definite cone. Coupling Gaussian laws with the same standard normal vector shows that each one-step map differs from the exact flow by $Ch^2$ in $W_1$, uniformly on these parameter sets. By synchronous coupling, the exact Ornstein--Uhlenbeck semigroup contracts $W_1$ }by the factor \blue{$e^{-\alpha h}$. Thus the errors satisfy $E_{k+1}\le e^{-\alpha h}E_k+Ch^2$, with $E_0=0$, for each discrete chain. Geometric summation yields $E_k\le Ch$ uniformly. In particular, for compactly supported test functions,
}\begin{equation}{\label{eq:dual-weak-error-local-pos}
|\langle\phi,\rho_k-\rho(kh)\rangle|+ |\langle\phi,\widehat\rho_k-\rho(kh)\rangle| \le Ch\|\phi\|_{C^{1,1}(U_{0})}. }\end{equation}
{Moreover, }Theorem~\ref{thm:uniform-C31}  {gives uniform }$C^{2,1}(U_r)$ {bounds for the two discrete families. The exact Ornstein--Uhlenbeck flow has means and covariance eigenvalues in a compact range and therefore has the same bounds}. In particular, all three families are uniformly Lipschitz on  {$U_0$:
}\[
{\sup_{t\ge0}\|\rho(t)\|_{C^{0,1}(U_{0})} +\sup_{k\ge0}\|\rho_k\|_{C^{0,1}(U_{0})} +\sup_{k\ge0}\|\widehat\rho_k\|_{C^{0,1}(U_{0})} \;\le\; C_2\,,
}\]
for some constant $C_2$ independent of $h$.

Fix $x\in U$ and let  {$\eta\in C_c^\infty(B_1(0))$ }be a standard bump function with $\int\eta=1$. For  {$0<r<\min\{1,\frac{1}{2}\dist(U,\partial U_0)\}$, so that $B_r(x)\subset U_0$ for every $x\in U$}, define
\[
\phi_{x,r}(y) := \frac{1}{r^d}\,\eta\Bigl(\frac{y-x}{r}\Bigr)\,, \qquad y\in\R^d\,.
\]
Then  {$\phi_{x,r}\in C_c^{1,1}(U_0)$}, and by scaling we have
 \[
{\|\phi_{x,r}\|_{C^{1,1}(U_{0})}\le C\,r^{-d-2}\,,
}\]
for some constant $C$ depending only on $\eta$ and $d$. Moreover, for any Lipschitz density $f$ on  {$U_0$,
}\[
{\Bigl|\int f(y)\,\phi_{x,r}(y)\,dy - f(x)\Bigr| \;\le\; C\,r\,\|f\|_{C^{0,1}(U_{0})}\,.
}\]

Applying this first to $f=\rho(kh,\cdot)$ and then to $f=\rho_k(\cdot)$ and $f=\widehat\rho_k(\cdot)$, and using \eqref{eq:dual-weak-error-local-pos}, we obtain
 \[
{\begin{aligned}
|\rho_k(x) -\rho(kh,x)| &\le \big|\langle\phi_{x,r},\rho_k-\rho(kh)\rangle\big| + C r\,\bigl(\|\rho_k\|_{C^{0,1}(U_{0})}+\|\rho(kh)\|_{C^{0,1}(U_{0})}\bigr)\,, \\
|\widehat\rho_k(x) -\rho(kh,x)| &\le \big|\langle\phi_{x,r},\widehat\rho_k-\rho(kh)\rangle\big| + C r\,\bigl(\|\widehat\rho_k\|_{C^{0,1}(U_{0})}+\|\rho(kh)\|_{C^{0,1}(U_{0})}\bigr)\,.
\end{aligned}
}\]
By \eqref{eq:dual-weak-error-local-pos} and the uniform Lipschitz bounds, this yields
\[
|\rho_k(x)-\rho(kh,x)| + |\widehat\rho_k(x)-\rho(kh,x)| \;\le\; C_3\bigl(h\,r^{-d-2} + r\bigr)\,,
\]
for some constant $C_3$ independent of $k$, $x\in U$ and $0<h\le1$.

Choosing, for example, $r=h^{1/(d+3)}$ gives
\[
\sup_{k\ge0}\|\rho_k-\rho(kh)\|_{L^{\infty}(U)} +\sup_{k\ge0}\|\widehat\rho_k-\rho(kh)\|_{L^{\infty}(U)} \;\le\; C_4\,h^{1/(d+3)}\,,
\]
for some $C_4>0$ independent of $h$.

Combining this with the lower bound \eqref{eq:FP-pos}, we obtain, for all $x\in U$, $k\ge0$ and  {sufficiently small $h$},
\[
\rho_k(x) \;\ge\;\rho(kh,x) - C_4 h^{1/(d+3)} \;\ge\; c_U - C_4 h^{1/(d+3)}\,,
\]
and the same inequality with $\rho_k$ replaced by $\widehat\rho_k$. Finally, choose $h_0\in(0,1]$ so small that $C_4 h_0^{1/(d+3)}\le c_U/2$, and set
\[
c_{U,h_0} := \tfrac12 c_U>0\,.
\]
Then for all $0<h\le h_0$, all $k\ge0$ and all $x\in U$,
\[
\rho_k(x)\;\ge\;c_{U,h_0}\,, \qquad \widehat\rho_k(x)\;\ge\;c_{U,h_0}\,,
\]
which proves \eqref{eq:disc-pos}.
\end{proof}

\subsubsection{Corollaries for kernel approximation}
{We compare the BRWP density iteration with iteration of the kernel formula in Wasserstein-1 distance. Lemma~\ref{lem:weak_epsilon_k} gives a one-step recurrence, and Corollary~\ref{cor:KV_non_refreshed} sums it over the stated horizon. We then compare the kernel formula with the exact Fokker--Planck evolution in Corollary~\ref{cor:weak_global_BRWP}.}

 \begin{lemma}[One-step $W_1$ comparison]
\label{lem:weak_epsilon_k}
 \blue{For laws with finite first moments, define
}\[
\blue{\epsilon_k=\rho_k-\widehat\rho_k,\quad e_k^{0,1}=W_1(\rho_k,\widehat\rho_k), \quad\delta_k=W_1(\cR_h\rho_k,\K_V^h\rho_k),\quad q=(1+\alpha h)^{-1}.
}\]
Then
 \begin{equation}\blue{\label{eq:eps-rec-C1}
e_{k+1}^{0,1}\le q e_k^{0,1}+\delta_k. }\end{equation}
\end{lemma}

\begin{proof}
Fix  {a globally Lipschitz test function $\varphi$ with $\operatorname{Lip}(\varphi)\le1$. Equal mass makes additive constants irrelevant. }Then
\begin{equation}\label{eq:eps-split-C01}
\langle\varphi,\epsilon_{k+1}\rangle =\langle\varphi,\rho_{k+1}-\widehat\rho_{k+1}\rangle =\underbrace{\langle\varphi,(F_k)_\#\rho_k-\K_V^h\rho_k\rangle}_{\mathcal E_1} +\underbrace{\langle\varphi,\K_V^h(\rho_k-\widehat\rho_k)\rangle}_{\mathcal E_2}\,.
\end{equation}

By duality and the {global }gradient contraction in Lemma~\ref{lem:PV-regularization},
 \[
{|\mathcal E_2|=|\langle P_V^h\varphi,\epsilon_k\rangle| \le\operatorname{Lip}(P_V^h\varphi)e_k^{0,1}\le q e_k^{0,1}.
}\]
{Kantorovich duality also gives $|\mathcal E_1|\le\delta_k$. Adding the two bounds and taking the }supremum over all  {such test functions proves~}\eqref{eq:eps-rec-C1}.
\end{proof}

Iterating Lemma~\ref{lem:weak_epsilon_k} gives the following comparison of the two density sequences.
\begin{corollary}\label{cor:KV_non_refreshed}
 \blue{If $\delta_j\le C_{\rm loc}h^2$ on the indicated horizon, then
}\[
\blue{W_1(\rho_k,\widehat\rho_k)\le q^kW_1(\rho_0,\widehat\rho_0) +\frac{C_{\rm loc}(1+\alpha h)}{\alpha} h(1-q^k).
}\]
\blue{In particular, for the same initial law and a fixed globally Lipschitz $\varphi$, $|\langle\varphi,\widehat\rho_k-\rho_k\rangle|\le Ch\operatorname{Lip}(\varphi)$. The one-step $W_1$ hypothesis is verified under the Gaussian conditions of Theorem~\ref{thm:uniform-C31}. }\end{corollary}
\begin{proof}
 {Iterating the norm recurrence of }Lemma~\ref{lem:weak_epsilon_k}  {gives
}\[
{e_k^{0,1}\le q^ke_0^{0,1}+C_{\rm loc}h^2\sum_{j=0}^{k-1}q^j.
}\]
{Since $1-q=\alpha h/(1+\alpha h)$, the geometric sum gives the stated bound. Under these Gaussian conditions the Gaussian parameters of the kernel and BRWP outputs have identical first derivatives with respect to $h$ at $h=0$ and bounded second derivatives on the compact parameter set, as in the proof of Lemma~\ref{lem:local-positivity}. The Gaussian coupling used there yields $\delta_j\le Ch^2$}.
\end{proof}

The next result compares the density obtained from the kernel formula with the exact Fokker--Planck evolution from the same input.
\begin{corollary}[Weak one-step density approximation]
\label{cor:weak_global_BRWP}
 \blue{In addition to the global potential hypotheses of Theorem~\ref{thm:PV-weak}, assume that $\rho(t)\in C_b^4(\R^d)$ and $\partial_t^2\rho(t)\in L^1(\R^d)$. Assume that the corresponding norms and the fourth moment are bounded uniformly over the stated time horizon. Then
}\begin{equation}\blue{\label{eq:global-kernel-FP-weak}
\Big| \big\langle \rho(t),\;\K_V^h\rho(t)-\rho(t+h)\big\rangle \Big| \ \le\ C\,h^2\Big( \|\rho(t)\|_{C_b^4}M_4(\rho(t))+\|\rho(t)\|_\infty\sup_{s\in[t,t+h]}\|\partial_s^2\rho(s)\|_{L^1} \Big)\,. }\end{equation}
\end{corollary}

\begin{proof}
Fix $t\ge0$ and write $\rho:=\rho(t)$. We add and subtract the first-order term generated by the Fokker--Planck operator:
\[
\big\langle \rho,\K_V^h\rho-\rho(t+h)\big\rangle = \underbrace{\big\langle \rho,\K_V^h\rho-\rho-h\,\mathcal{L}^*\rho\big\rangle}_{\mathrm{(I)}} + \underbrace{\big\langle \rho,\rho+h\,\mathcal{L}^*\rho-\rho(t+h)\big\rangle}_{\mathrm{(II)}}\,.
\]

For (I), {\color{black}the global estimate for general inputs established in the proof of Theorem~\ref{thm:PV-weak}} and the duality $\langle P_V^h f,g\rangle=\langle f,\K_V^h g\rangle$ give
 \[
{|\mathrm{(I)}| = \big|\langle P_V^h\rho-\rho-h\,\mathcal L\rho,\,\rho\rangle\big| \ \le\ C\,h^2\Big(\|\rho\|_{C_b^4}M_4(\rho)\Big).
}\]

For (II),  {Taylor's formula in $L^1$ }gives
 \[
{\rho(t+h)-\rho-h\mathcal L^*\rho =\int_t^{t+h}(t+h-s)\partial_s^2\rho(s)\,ds.
}\]
{Its pairing with $\rho(t)$ is bounded by $(h^2/2)\|\rho(t)\|_\infty\sup_{s\in[t,t+h]}\|\partial_s^2\rho(s)\|_{L^1}$. }Combining the bounds for (I) and (II) yields \eqref{eq:global-kernel-FP-weak}.
\end{proof}

\section{Proofs and Auxiliary Estimates for Section~\ref{sec_conv_KL}}
\label{sec:app_KL}
{This appendix proves the identities along the Fokker--Planck flow and the one-step density and entropy expansions in Lemmas~\ref{Lemma_rho_k_iteration}--\ref{Lemma_KL_decay_pre}. It also gives an additional entropy recurrence for Gaussian iterates. The one-step inequalities used in the general convergence theorem are proved in Appendix~\ref{sec:analytic-convergence}.}

\subsection{Entropy identities along the Fokker--Planck flow}
\label{App_sec_cont}

{\blue{For Lemmas 24--27, assume a smooth positive Fokker--Planck solution, with all displayed integrals finite, differentiation under the integral justified, and vanishing boundary terms in every integration by parts.} These conditions hold for the Gaussian solutions under the conditions of Theorem~\ref{thm:uniform-C31}. }\oldresult{24} Let $\rho^*(x) = \frac{1}{Z} \exp(-\beta V(x))$. When $\nabla^2 V \succeq \alpha I$, it is known that $\rho^*$ satisfies the logarithmic Sobolev inequality. Specifically, for any smooth function $g$ with  {$g\in H^1(\rho^*)$ and $0<\int g^2\rho^*dx<\infty$}, we have
\begin{equation}
\label{Log_sobolev}
        \int g^2 \log g^2 \rho^* \, dx - \int g^2 \rho^* \, dx \log \int g^2 \rho^* \, dx \leq \frac{2}{\beta \alpha} \int \|\nabla g\|^2 \rho^* \, dx \,.
\end{equation}
Note that the factor of $\beta$ arises from the definition $\rho^* = \frac{1}{Z} \exp(-\beta V)$, and our assumption applies to the function $V$.

Using the logarithmic Sobolev inequality, we can derive the following dissipation result for the Fokker--Planck equation \eqref{eqn_FP} with solution $\rho$.

\begin{lemma}
\label{Lemma_KL_derivative}
    If $\rho^*$ satisfies the logarithmic Sobolev inequality \eqref{Log_sobolev} and $\rho$ is the solution to the Fokker--Planck equation \eqref{eqn_FP}, then
    \[
    \frac{d}{dt}\KL(\rho \|\rho^*) = -\beta^{-1} \FI(\rho\|\rho^*) \leq -2 \alpha \KL(\rho \|\rho^*) \,.
    \]
\end{lemma}

\begin{proof}
Differentiate the KL divergence and integrate by parts under the stated integrability and boundary conditions:
\[
\frac{d}{dt}\KL(\rho \|\rho^*) = \int \frac{\partial}{\partial t} \rho \log \frac{\rho}{\rho^*} \, dx = \beta^{-1} \int \nabla \cdot \left(\nabla \log \frac{\rho}{\rho^*}\rho\right) \log \frac{\rho}{\rho^*} \, dx = -\beta^{-1} \int \left\|\nabla \log \frac{\rho}{\rho^*}\right\|_2^2 \rho \, dx \,.
\]
Substituting $g^2 = \frac{\rho}{\rho^*}$ into~\eqref{Log_sobolev} gives
\[
\KL(\rho \|\rho^*) = \int \log \frac{\rho}{\rho^*} \rho\, dx \leq \frac{1}{2 \beta \alpha} \int \left\|\nabla \log \frac{\rho}{\rho^*}\right\|_2^2 \rho \, dx = \frac{1}{2 \beta \alpha} \FI(\rho \|\rho^*) \,.
\]
\end{proof}\hspace*{\fill}

Differentiating once more gives the second derivative of the KL divergence and decay of the Fisher information. The resulting spatial identity also gives the coefficients in Lemma~\ref{Lemma_KL_decay_pre}.

\begin{lemma}
    {Assume $\nabla^2V\succeq\alpha I$. }When $\rho^*$ satisfies the logarithmic Sobolev inequality and $\rho$ is the solution to the Fokker--Planck equation, the second-order time derivative of the KL divergence satisfies \label{Lemma_FI_derivative}
    \begin{enumerate}
        \item[(i)]
        \[
         \frac{d^2}{dt^2} \KL(\rho \|\rho^*) = 2 \beta^{-2} \int \left\|\nabla^2 \log \frac{\rho}{\rho^*} \right\|_{\mathrm{F}}^2 \rho \, dx + 2 \beta^{-1} \int \left\langle \nabla \log \frac{\rho}{\rho^*}, \nabla^2 V \nabla \log \frac{\rho}{\rho^*} \right\rangle \rho \, dx \,.
        \]
        \item[(ii)]
        \[
        \frac{d^2}{dt^2} \KL(\rho \|\rho^*) = -\beta^{-1} \frac{d}{dt} \FI(\rho \|\rho^*) \geq 2 \beta^{-1} \alpha \FI(\rho \|\rho^*) \geq 4  \alpha^2 \KL(\rho \|\rho^*) \,.
        \]
    \end{enumerate}
\end{lemma}

\begin{proof}
Recall the weighted operator $\mathcal{D}^{\beta}$ and the Fokker--Planck generator $\mathcal{L}$:
\begin{equation}
\label{recall_D_beta}
    \mathcal{D}^{\beta}(u) := \frac{\beta^{-1}}{\rho} \nabla \cdot (\nabla u \, \rho) \,, \quad \mathcal{L}(v) = \beta^{-1} \Delta v - \nabla V \cdot \nabla v \,.
\end{equation}

Differentiating along the Fokker--Planck flow and integrating by parts gives
\begin{align}
\label{eqn:dt2KL_mid}
    &\frac{d^2}{dt^2} \KL(\rho \|\rho^*) = \frac{d}{dt} \int \frac{\partial}{\partial t} \rho \log \frac{\rho}{\rho^*} \, dx = \int \frac{\partial^2 \rho}{\partial t^2} \log \frac{\rho}{\rho^*} \, dx + \int \left| \frac{\partial \rho}{\partial t} \right|^2 \frac{1}{\rho} \, dx \\
    =& \int \left[ \beta^{-1} \Delta \frac{\partial \rho}{\partial t} + \nabla \cdot \left( \nabla V \frac{\partial \rho}{\partial t} \right) \right] \log \frac{\rho}{\rho^*} \, dx + \int \left| \mathcal{D}^{\beta} \left( \log \frac{\rho}{\rho^*} \right) \right|^2 \rho \, dx \notag \\
    =& \int \mathcal{D}^{\beta} \left( \log \frac{\rho}{\rho^*} \right) \mathcal{L} \left( \log \frac{\rho}{\rho^*} \right) \rho \, dx + \int \left| \mathcal{D}^{\beta} \left( \log \frac{\rho}{\rho^*} \right) \right|^2 \rho \, dx \notag \\
    =& 2 \int \mathcal{D}^{\beta} \left( \log \frac{\rho}{\rho^*} \right) \mathcal{L} \left( \log \frac{\rho}{\rho^*} \right) \rho \, dx + \int \mathcal{D}^{\beta} \left( \log \frac{\rho}{\rho^*} \right) \left[ \mathcal{D}^{\beta} \left( \log \frac{\rho}{\rho^*} \right) - \mathcal{L} \left( \log \frac{\rho}{\rho^*} \right) \right] \rho \, dx \notag \\
    =& -2 \beta^{-1} \int \nabla \log \frac{\rho}{\rho^*} \cdot \nabla \mathcal{L} \left( \log \frac{\rho}{\rho^*} \right) \rho \, dx + \beta^{-1} \int \mathcal{D}^{\beta} \left( \log \frac{\rho}{\rho^*} \right) \left\| \nabla \log \frac{\rho}{\rho^*} \right\|_2^2 \rho \, dx \notag \\
    =& -2 \beta^{-1} \int \nabla \log \frac{\rho}{\rho^*} \cdot \nabla \mathcal{L} \left( \log \frac{\rho}{\rho^*} \right) \rho \, dx + \beta^{-1} \int \mathcal{L} \left\| \nabla \log \frac{\rho}{\rho^*} \right\|_2^2 \rho \, dx \,, \notag
\end{align}
where the last equality follows from $\mathcal D^{\beta}\left(\logs\right)\rho = \beta^{-1}\Delta\rho + \nabla\cdot(\nabla V\rho)$.

The commutator between $\mathcal{L}$ and $\nabla$ for a smooth function $f$ can be written as
\[
\nabla \mathcal{L} f - \mathcal{L} \nabla f = \nabla \left( \beta^{-1} \Delta f - \nabla V \cdot \nabla f \right) - \beta^{-1} \Delta (\nabla f) + \nabla V \cdot \nabla^2 f = -\nabla^2 V \nabla f \,.
\]
Using Bochner's formula
\[
\frac{1}{2} \Delta \left( \|\nabla f\|_2^2 \right) = \Delta \nabla f \cdot \nabla f + \|\nabla^2 f\|_{\mathrm{F}}^2 \,,
\]
we note that
\begin{align}
\label{Gamma2_identity}
&\frac{1}{2} \mathcal{L} \|\nabla f\|_2^2 -\nabla f \cdot \nabla \mathcal{L}(f)  \\
= &\frac{1}{2} \mathcal{L} \|\nabla f\|_2^2  - \nabla f \cdot \mathcal{L} \nabla f + \left\langle \nabla f, \nabla^2 V \nabla f \right\rangle \notag \\
= &\beta^{-1} \left( \frac{1}{2} \Delta \|\nabla f\|_2^2 - \nabla f \cdot \Delta \nabla f \right) - \frac{1}{2} \nabla V \cdot \nabla \|\nabla f\|_2^2 + \langle \nabla f, \nabla^2 f \nabla V \rangle + \left\langle \nabla f, \nabla^2 V \nabla f \right\rangle \notag \\
= &\beta^{-1} \|\nabla^2 f\|_{\mathrm{F}}^2 + \left\langle \nabla f, \nabla^2 V \nabla f \right\rangle\,.\notag
\end{align}
Combining the final equality in \eqref{Gamma2_identity} with \eqref{eqn:dt2KL_mid}, we obtain the first statement.

Substituting $f = \log \frac{\rho}{\rho^*}$ into~\eqref{Gamma2_identity} and using~\eqref{eqn:dt2KL_mid} gives
\begin{align*}
&\frac{d^2}{dt^2} \KL(\rho \|\rho^*) = 2 \beta^{-2} \int \left\|\nabla^2 \log \frac{\rho}{\rho^*} \right\|_{\mathrm{F}}^2 \rho \, dx + 2 \beta^{-1} \int \left\langle \nabla \log \frac{\rho}{\rho^*}, \nabla^2 V \nabla \log \frac{\rho}{\rho^*} \right\rangle \rho \, dx \\
\geq &2 \beta^{-1} \alpha \int \left\|\nabla \log \frac{\rho}{\rho^*} \right\|_2^2 \rho \, dx = 2 \beta^{-1} \alpha \FI(\rho \|\rho^*) \geq 4 \alpha^2 \KL(\rho \|\rho^*) \,,
\end{align*}
where we used $\nabla^2 V \succeq \alpha I_{d}$ and the final inequality follows from Lemma \ref{Lemma_KL_derivative}.
\end{proof}\hspace*{\fill}

\begin{lemma}
\label{lemma_second_order}
Let $\rho$ satisfy the Fokker--Planck equation \eqref{eqn_FP}. Then, the second-order time derivative of the KL divergence satisfies
\begin{align*}
\frac{1}{2} \frac{d^2}{dt^2} \KL(\rho \| \rho^*) = & \int \left|\mathcal{D}^{\beta} \left(\log \frac{\rho}{\rho^*}\right)\right|^2 \rho \, dx   + \beta^{-2} \int \left\langle \nabla \log \frac{\rho}{\rho^*}, \, \nabla^2 \log \frac{\rho}{\rho^*} \nabla \log \frac{\rho}{\rho^*} \right\rangle \rho \, dx \,,
\end{align*}
where $\mathcal D^{\beta}$ is defined in \eqref{recall_D_beta}.
\end{lemma}

\begin{proof}
Equation~\eqref{eqn:dt2KL_mid} gives
\begin{align*}
\frac{d^2}{dt^2} \KL(\rho \| \rho^*) &= \int \mathcal{D}^{\beta} \left( \log \frac{\rho}{\rho^*} \right) \mathcal{L} \left( \log \frac{\rho}{\rho^*} \right) \rho \, dx + \int \left| \mathcal{D}^{\beta} \left( \log \frac{\rho}{\rho^*} \right) \right|^2 \rho \, dx \,.
\end{align*}
The two operators satisfy
\[
\mathcal{D}^{\beta} \left( \log \frac{\rho}{\rho^*} \right) - \mathcal{L} \left( \log \frac{\rho}{\rho^*} \right) = \beta^{-1} \left\| \nabla \log \frac{\rho}{\rho^*} \right\|_2^2 \,.
\]
Consequently,
\begin{align*}
\int \left| \mathcal{D}^{\beta} \left( \log \frac{\rho}{\rho^*} \right) \right|^2 \rho \, dx &= \frac{1}{2} \left[ \int \mathcal{D}^{\beta} \left( \log \frac{\rho}{\rho^*} \right) \mathcal{L} \left( \log \frac{\rho}{\rho^*} \right) \rho \, dx + \int \left| \mathcal{D}^{\beta} \left( \log \frac{\rho}{\rho^*} \right) \right|^2 \rho \, dx \right] \\
&\quad + \frac{1}{2} \int \left( \mathcal{D}^{\beta} \left( \log \frac{\rho}{\rho^*} \right) - \mathcal{L} \left( \log \frac{\rho}{\rho^*} \right) \right) \mathcal{D}^{\beta} \left( \log \frac{\rho}{\rho^*} \right) \rho \, dx \\
&= \frac{1}{2} \frac{d^2}{dt^2} \KL(\rho \| \rho^*) + \frac{\beta^{-1}}{2} \int \left\| \nabla \log \frac{\rho}{\rho^*} \right\|_2^2 \mathcal{D}^{\beta} \left( \log \frac{\rho}{\rho^*} \right) \rho \, dx \\
&= \frac{1}{2} \frac{d^2}{dt^2} \KL(\rho \| \rho^*) - \beta^{-2} \int \left\langle \nabla \log \frac{\rho}{\rho^*}, \, \nabla^2 \log \frac{\rho}{\rho^*} \nabla \log \frac{\rho}{\rho^*} \right\rangle \rho \, dx \,,
\end{align*}
where the last equality uses the following identity, valid for any smooth function $u$:
\begin{align*}
&\int \|\nabla u\|_2^2 \mathcal{D}^{\beta}(u) \rho \, dx = \beta^{-1} \int \|\nabla u\|_2^2 \nabla \cdot (\nabla u \rho) \, dx \\
=& -\beta^{-1} \int \nabla \|\nabla u\|_2^2 \cdot \nabla u \rho \, dx = -2 \beta^{-1} \int \left\langle \nabla u, \, \nabla^2 u \nabla u \right\rangle \rho \, dx \,.
\end{align*}
\end{proof}\hspace*{\fill}

\subsubsection{Decay of the fourth moment of the relative score}
We first prove decay of the fourth moment of the relative score along the exact flow, then derive a discrete bound for Gaussian BRWP iterates.
\begin{lemma}
\label{Lemma_4th_order}
When $\rho^*$ is strongly log-concave with $ \nabla^2 V \succeq \alpha I$ and $\rho$ is the solution to the Fokker--Planck equation, we have
\[
    \frac{\partial}{\partial t}\int \left\|\nabla\logs\right\|_2^4 \rho dx   \leq -4\alpha  \int\left\|\nabla\logs\right\|_2^4 \rho dx \,.
    \]
\end{lemma}
\begin{proof}
Differentiating in time gives
\begin{align}
    \label{eqn:4th_order_2}
     & \frac{\partial}{\partial t}\int \left\|\nabla\logs\right\|_2^4 \rho dx  = 4\int \left\|\nabla\logs\right\|_2^2\left\langle \nabla\logs, \nabla \frac{\partial \rho/\partial t}{\rho}\right\rangle \rho dx  + \int \left\|\nabla\logs\right\|_2^4 \frac{\partial\rho}{\partial t}dx\, .
\end{align}
Using the identity $\frac{\partial\rho}{\partial t} = \beta^{-1}\nabla\cdot\left(\nabla\logs\rho\right)$ and integrating by parts, we rewrite the first term as
\begin{align}
\label{eqn:4th_order_3}
&4\int \left\|\nabla\logs\right\|_2^2\left\langle \nabla\logs,\, \nabla \frac{\partial \rho/\partial t}{\rho}\right\rangle \rho dx \\
=&-4\int \nabla\cdot \left(\left\|\nabla\logs\right\|_2^2\nabla\logs\right) \frac{\partial \rho/\partial t}{\rho}\rho dx  - 4\int \left\|\nabla\logs\right\|_2^2\left\langle  \nabla\logs , \frac{\nabla\rho}{\rho}\right\rangle \frac{\partial\rho}{\partial t}dx\notag \\
  =&-4\beta^{-1}\int \nabla\cdot \left(\left\|\nabla\logs\right\|_2^2\nabla\logs\right) \nabla\cdot\left(\nabla\logs\rho\right) dx\notag\\
  &- 4\beta^{-1}\int \left\|\nabla\logs\right\|_2^2\left\langle  \nabla\logs , \frac{\nabla\rho}{\rho}\right\rangle \nabla\cdot\left(\nabla\logs\rho\right)dx\notag\\
  =& 4\beta^{-1}\int\left\langle\nabla\nabla\cdot \left(\left\|\nabla\logs\right\|_2^2\nabla\logs\right)  \,,\nabla\logs\right\rangle \rho dx\notag \\
 & +4\beta^{-1}\int \left\langle  \bigg[\nabla\left\langle \left\|\nabla\logs\right\|_2^2 \nabla\logs , \,\frac{\nabla\rho}{\rho}\right\rangle\bigg] , \,\nabla\logs \right\rangle\rho  dx\, .\notag
\end{align}
The last term in~\eqref{eqn:4th_order_3} is
 \begin{align}
{\label{eqn:4th_order_4} }&{4\beta^{-1}\int \left\langle  \bigg[\nabla\left\langle \left\|\nabla\logs\right\|_2^2 \nabla\logs ,\, \frac{\nabla\rho}{\rho}\right\rangle\bigg] , \nabla\logs \right\rangle\rho  dx  }\\
{=}&{4\beta^{-1}\int\left\langle\left\langle\nabla\left(\left\|\nabla\logs\right\|_2^2\nabla\logs\right), \frac{\nabla\rho}{\rho}\right\rangle,\nabla\logs\right\rangle \rho  dx\notag }\\
&{+4\beta^{-1}\int\left\langle\left\langle\left\|\nabla\logs\right\|_2^2\nabla\logs,\nabla\bigg(\frac{\nabla\rho}{\rho}\bigg)\right\rangle ,\nabla\logs\right\rangle \rho  dx\notag}\\
 {=}&  {4\beta^{-1}\int\left\langle \bigg(\nabla\bigg(\left\|\nabla\logs\right\|_2^2\nabla\logs\bigg)\bigg)^T\nabla\logs,\nabla\rho\right\rangle dx \notag}\\&{+4\beta^{-1}\int \left\|\nabla\logs\right\|_2^2\left\langle  \nabla\logs, \nabla\bigg(\frac{\nabla\rho}{\rho}\bigg)\nabla\logs\right\rangle \rho dx \,,\notag }\end{align}
where, for $u = \left\|\nabla\logs\right\|_2^2 \nabla\logs$ and $v = \frac{\nabla\rho}{\rho}$, we used
\[
(\nabla (u\cdot v))_i = \sum_j\bigg[\frac{\partial u_j}{\partial x_i}v_j + u_j\frac{\partial v_j}{\partial x_i}\bigg] =  ((\nabla u) \cdot v)_i +   (u\cdot (\nabla v))_i \quad \text{where } (\nabla u)_{ij} = \frac{\partial u_j}{\partial x_i} \,,
\]
in the first equality. We also used
\[
 (A\cdot u) \cdot v = \sum_{ij}a_{ij}u_jv_i = \sum_j(\sum_i a_{ij} v_i)u_j = (A^T v) \cdot u\,, \quad (u\cdot A)\cdot v =  (u^TA)v = u\cdot (Av) \,,
\]
in the second equality.

Noting that $\nabla(\nabla \rho/\rho) = \nabla^2\log\rho$, the last term in \eqref{eqn:4th_order_4} equals
\begin{align}
\label{eqn:4th_order_5}
&4\beta^{-1}\int \left\|\nabla\logs\right\|_2^2\left\langle \nabla\logs, \nabla\bigg(\frac{\nabla\rho}{\rho}\bigg)\nabla\logs\right\rangle \rho dx\\ =& 4\beta^{-1}\int \left\|\nabla\logs\right\|_2^2\left\langle  \nabla\logs, \nabla^2\log\rho\nabla\logs\right\rangle \rho dx \,.\notag
\end{align}

Moreover, applying integration by parts with respect to $\nabla\rho$ to the first term in the last line of \eqref{eqn:4th_order_4}, we get
 \begin{align}
{\label{eqn:4th_order_6} }&{4\beta^{-1}\int\left\langle \bigg(\nabla\bigg(\left\|\nabla\logs\right\|_2^2\nabla\logs\bigg)\bigg)^T\nabla\logs,\,\nabla\rho\right\rangle dx}\\
    {=}& {-4\beta^{-1}\int \nabla\cdot\bigg[\bigg(\nabla\bigg(\left\|\nabla\logs\right\|_2^2\nabla\logs\bigg)\bigg)^T\nabla\logs\bigg] \rho dx  \notag}\\
    {=}& {-4\beta^{-1}\int\left\langle \nabla\cdot \bigg(\nabla\bigg(\left\|\nabla\logs\right\|_2^2\nabla\logs\bigg)\bigg) ,\,\nabla\logs\right\rangle\rho dx\notag}\\
    &{- 4\beta^{-1}\int \Tr\bigg\{\nabla\bigg(\left\|\nabla\logs\right\|_2^2\nabla\logs\bigg)\nabla^2\logs\bigg\}\rho dx \, ,\notag }\end{align}
where, for $A = \nabla\left(\left\|\nabla\logs\right\|_2^2\nabla\logs\right)$ and $b = \nabla\logs$, we used the identity
 \begin{align*}
&{\nabla\cdot (A^Tb) =   \sum_j \frac{\partial\sum_i a_{ij}b_i}{\partial x_j} = \sum_{ij}\bigg(\frac{\partial a_{ij}}{\partial x_j}b_i + a_{ij}\frac{\partial b_i}{\partial x_j}\bigg) = (\nabla\cdot A)\cdot b + \sum_{ij}a_{ij}(\nabla b)^T_{ij}}\\ {= }&{(\nabla\cdot A)\cdot b + \text{Tr}\{A(\nabla b)\}\,,\quad (A_{ij}) = a_{ij}\,,\quad (\nabla \cdot A)_i = \sum_j\frac{\partial a_{ij}}{\partial x_j}\,. }\end{align*}

For $u = \left\|\nabla\logs\right\|_2^2\nabla\logs$, we have
 \[
{(\nabla\cdot (\nabla u))_i = \sum_j \frac{\partial (\nabla u)_{ij}}{\partial x_j} =  \sum_j\frac{\partial}{\partial x_j} \frac{\partial u_{j}}{\partial x_i} = \frac{\partial}{\partial x_i}\sum_j \frac{\partial u_j}{\partial x_j} = (\nabla(\nabla\cdot u))_i\,,
}\]
hence the first terms in the final expressions of \eqref{eqn:4th_order_3} and \eqref{eqn:4th_order_6} cancel. It remains to evaluate the last term in \eqref{eqn:4th_order_6}. For the scalar function $\zeta = \|\nabla\logs\|_2^2$ and the vector-valued function $u = \nabla\logs$, we have
\[
(\nabla(\zeta u))_{ij} = \frac{\partial \zeta u_j}{\partial x_i} = \frac{\partial \zeta}{\partial x_i}u_j+\zeta\frac{\partial u_j}{\partial x_i}  = (\nabla\zeta u^T)_{ij}+\zeta(\nabla u)_{ij} \,.
\]
This implies
\begin{align}
\label{eqn:4th_order_7}
   & - 4\beta^{-1}\int \Tr\bigg\{\nabla\bigg(\left\|\nabla\logs\right\|_2^2\nabla\logs\bigg)\nabla^2\logs\bigg\}\rho dx\\
    =& - 4\beta^{-1}\int \Tr\bigg\{\bigg(\nabla\left\|\nabla\logs\right\|_2^2 \left(\nabla\logs\right)^T + \left\|\nabla\logs\right\|_2^2\nabla^2\logs\bigg)\nabla^2\logs\bigg\}\rho dx\notag\\
        =& - 8\beta^{-1} \int \text{Tr}\bigg\{ \left(\nabla^2\logs\cdot \nabla\logs\right)\left(\nabla^2\logs\cdot \nabla\logs\right)^T\bigg\}\rho dx \notag\\
        &-4\beta^{-1}\int \left\|\nabla\logs\right\|_2^2\left\|\nabla^2\logs\right\|_{\mathrm{F}}^2\rho dx \notag\\
    =&- 8\beta^{-1}\int \left\|\nabla^2\logs\nabla\logs\right\|_2^2\rho dx -4\beta^{-1}\int \left\|\nabla\logs\right\|_2^2\left\|\nabla^2\logs\right\|_{\mathrm{F}}^2\rho dx \,,\notag
\end{align}
where we also used the fact that $\nabla^2\logs$ is symmetric and $Abb^TA^T = Ab(Ab)^T$.

Finally, the second term in \eqref{eqn:4th_order_2} is
\begin{align}
\label{eqn:4th_order_8}
    \int \left\|\nabla\logs\right\|_2^4 \frac{\partial\rho}{\partial t}dx = -4\beta^{-1}\int \left\|\nabla\logs\right\|^2_2\left\langle\nabla\logs ,\nabla^2\logs\nabla\logs\right\rangle\rho dx\,.
\end{align}

Combining~\eqref{eqn:4th_order_2}--\eqref{eqn:4th_order_8} gives
\begin{align*}
    & \frac{\partial}{\partial t}\int \left\|\nabla\logs\right\|_2^4 \rho dx  = -4\beta^{-1}\int \left\|\nabla^2 \logs\right\|_{\mathrm{F}}^2\left\|\nabla\logs\right\|_2^2 \rho dx \\&\qquad - 4 \int \left\|\nabla\logs\right\|_2^2\left\langle \nabla\logs,\, \nabla^2 V \nabla\logs\right\rangle\rho dx - 8\beta^{-1} \int \left\|\nabla^2\logs\nabla\logs\right\|_2^2\rho dx\,.
\end{align*}

The bound $\nabla^2 V\succeq \alpha I$ proves the result.
\end{proof} \hspace*{\fill}

 \blue{For BRWP iterates under the same Gaussian conditions as Theorem~\ref{thm:uniform-C31}, the following discrete estimate holds.
}\begin{lemma}
\label{lemma_4th_order_discrete}
There exists $C<\infty$, independent of $k$ and $h$, such that for all sufficiently small $h>0$,
\begin{equation}
\label{eqn:4th_exp}
\int\left\|\nabla\log\frac{\rho_k}{\rho^*}\right\|_2^4 \rho_k \, dx \;\le\; e^{-4\alpha hk} \int\left\|\nabla\log\frac{\rho_0}{\rho^*}\right\|_2^4 \rho_0 \, dx \;+\; C h \, .
\end{equation}
In particular, the fourth moment of the relative score is uniformly bounded in $k$.
\end{lemma}

\begin{proof}
Let
\[
F(\rho) := \int\left\|\nabla\log\frac{\rho}{\rho^*}\right\|_2^4 \rho \, dx .
\]

By Lemma~\ref{Lemma_4th_order}, for the exact Fokker--Planck flow $\rho_t$ we have
\[
\frac{d}{dt} F(\rho_t) \le -4\alpha F(\rho_t) .
\]
Hence,
\begin{equation}
\label{eq:cont_decay}
F(\rho_{t+h}) \le e^{-4\alpha h} F(\rho_t) .
\end{equation}

 \blue{On the Gaussian parameter sets in Proposition~\ref{prop:gaussian}, $F(\mathcal N(m,C))$ is a smooth function of $m$ and $C\succ0$, with bounded derivatives. The BRWP and exact }Fokker--Planck  \blue{parameter maps have the same first derivatives with respect to $h$ at $h=0$ and bounded second derivatives. Taylor's theorem in this finite-dimensional parameter space therefore gives
}\[
\blue{F(\rho_{k+1})=F(S_h\rho_k)+O(h^2),
}\]
\blue{uniformly, where $S_h$ denotes the exact flow starting at }$\rho_k$. Combining this with \eqref{eq:cont_decay} yields
\[
F(\rho_{k+1}) \le e^{-4\alpha h} F(\rho_k) + C h^2 .
\]

Iterating the above inequality gives
\[
F(\rho_k) \le e^{-4\alpha hk} F(\rho_0) + C h^2 \sum_{j=0}^{k-1} e^{-4\alpha h j} \le e^{-4\alpha hk} F(\rho_0) + C h ,
\]
since  {$\sum_{j\ge0} e^{-4\alpha h j}=(1-e^{-4\alpha h})^{-1}\le C_\alpha/h$}.
\end{proof}

\subsection{Proofs of Lemmas~\ref{Lemma_rho_k_iteration}--\ref{Lemma_KL_decay_pre} and auxiliary estimates}
\label{App_sec_KL}

\begin{proof}[Proof of Lemma~\ref{Lemma_rho_k_iteration}]
 {We use the analytic conditions of Assumption~\ref{assump:SO-weak} and the bounded-test condition~\eqref{eq:weak-oracle-global} of Lemma~\ref{Lemma_rho_k_iteration}. For $u\in C_b^3$, the transport Taylor remainder is integrable. Estimate~\eqref{eq:weak-oracle-global} with constant vector-valued test functions, together with $\nabla(\widetilde\Phi-\Phi)\in L^2(\rho)$, gives $\nabla(\mathcal L^*\rho/\rho)\in L^1(\rho)$ for each $h>0$. All divergence pairings below are defined distributionally. }\medskip 
 
 \emph{Step 1. Taylor expansion.} From the BRWP update,
\[
\langle u,\rho_{k+1}\rangle =\int u\!\left(x-h\nabla\!\Bigl(\beta^{-1}\log\!\frac{\widetilde\rho_{k+1}}{\rho^*}\Bigr)(x)\right)\rho_k(x)\,dx \,.
\]
Taylor expansion gives
 \begin{gather*}
u(x-h\nabla\widetilde\Phi_k) =u-h\nabla\widetilde\Phi_k\cdot\nabla u +\frac{h^2}{2}(\nabla\widetilde\Phi_k)^{\!\top}\nabla^2u\,\nabla\widetilde\Phi_k+R_3,\\
|R_3(x)|\le Ch^3\|D^3u\|_\infty|\nabla\widetilde\Phi_k(x)|^3.
\end{gather*}Thus
 \begin{align}
{\label{eq:weak-Taylor-correct} \langle u,\rho_{k+1}-\rho_k\rangle }&{= -h\,\langle\nabla\widetilde\Phi_k\!\cdot\nabla u\,,\,\rho_k\rangle +\frac{h^2}{2}\,\langle(\nabla\widetilde\Phi_k)}{}{^{\!\top}\nabla^2u\,\nabla\widetilde\Phi_k\,,\,\rho_k\rangle +O(h^3\|u\|_{C_b^3}) \,. }\end{align}

\medskip \emph{Step 2. Expansion of the drift.} Write
\[
\nabla\!\Bigl(\beta^{-1}\log\!\frac{\widetilde\rho_{k+1}}{\rho^*}\Bigr) = \nabla\!\Bigl(\beta^{-1}\log\!\frac{\rho_k}{\rho^*}\Bigr) + h\beta^{-1}\nabla\!\Bigl(\tfrac{\mathcal L^{\!*}\rho_k}{\rho_k}\Bigr) + \beta^{-1}e_{k,h} \,,
\]
{Since $\nabla u\in C_b^{1,1}$ and $\norm{\nabla u}_{C_b^{1,1}}\le C\norm u_{C_b^3}$, the weak oracle estimate~\eqref{eq:weak-oracle-global} bounds the error in the linear term by $Ch^3\norm u_{C_b^3}$. The first-order difference bound in~\eqref{eq:strong-taylor-assumption} gives $\|\nabla(\widetilde\Phi_k-\Phi_k)\|_{L^2(\rho_k)}\le Ch$, and the two fields have bounded $L^2$ norms. Cauchy--Schwarz therefore bounds the change in the quadratic term by $Ch^3\|D^2u\|_\infty$. The cubic moment in~\eqref{eq:strong-taylor-assumption} bounds the integrated third-order remainder. }

\medskip \emph{Step 3. Collecting the first two terms.} Substituting into \eqref{eq:weak-Taylor-correct} yields
\begin{align}
\langle u,\rho_{k+1}-\rho_k\rangle &= -h\,\Big\langle \nabla\!\Bigl(\beta^{-1}\log\!\frac{\rho_k}{\rho^*}\Bigr)\!\cdot\nabla u \,,\, \rho_k\Bigr\rangle
\\[-2pt]
&\quad -h^2\beta^{-1} \Bigl\langle \nabla\!\Bigl(\tfrac{\mathcal L^{\!*}\rho_k}{\rho_k}\Bigr)\!\cdot\nabla u \,,\, \rho_k\Bigr\rangle + \frac{h^2}{2}\, \Bigl\langle (\nabla\Phi_k)^{\!\top}\nabla^2u\,\nabla\Phi_k \,,\,\rho_k\Bigr\rangle +O(h^3) \,,\notag
\end{align}

\medskip \emph{Hessian term.} As usual,
\[
(\nabla\Phi_k)^{\!\top}\nabla^2u\,\nabla\Phi_k =\sum_{i,j}\partial_{ij}u\,\partial_i\Phi_k\,\partial_j\Phi_k \,,
\]
and distributional integration by parts gives
\[
\big\langle (\nabla\Phi_k)^{\!\top}\nabla^2u\,\nabla\Phi_k \,,\,\rho_k \big\rangle = \big\langle u \,,\,\nabla\cdot(\rho_k\zeta_k)\big\rangle \,,
\]
where $\rho_k\zeta_k=\nabla\cdot(\rho_k(\nabla\Phi_k)(\nabla\Phi_k)^{\mathsf T})$. The integrable tensor defines the pairing against bounded $D^2u$; smooth cutoffs have vanishing error terms. Substitution completes the proof.
\end{proof}

 \begin{proof}[Proof of Lemma~\ref{Lemma_KL_decay_pre}]
{Under Assumption~\ref{assump:SO-weak}, take $h$ sufficiently small that $hD^2\widetilde\Phi\preceq I_d/2$, as justified in Lemma~\ref{lem:analytic-descent}. Put $f=\log(\rho/\rho^*)$ and $\Psi_h(x)=|x|^2/2-h\widetilde\Phi(x)$. Since $D^2\Psi_h\succeq I_d/2$, for each $y$, the function $\Psi_h(x)-y\cdot x$ is coercive and strictly convex, and has a unique minimizer. Hence $F=\nabla\Psi_h$ is bijective, and its positive definite Jacobian gives a global orientation-preserving diffeomorphism. With $W=-\log\rho^*$, a change of variables gives the exact identity
}\[
 \blue{\KL(F_\#\rho\|\rho^*)-\KL(\rho\|\rho^*) =\int[W(x-h\nabla\widetilde\Phi)-W(x)-\log\det(I-hD^2\widetilde\Phi)]\rho\dd x.
}\]
{For every $z\le1/2$, scalar Taylor's formula gives
}\[
\blue{\left|-\log(1-z)-z-\frac{z^2}{2}\right| \le\frac83|z|^3.}
\]
{Indeed the third derivative is $2/(1-z)^3$, bounded by $16$ on the segment between $0$ and $z$. Apply this to the eigenvalues of $hD^2\widetilde\Phi$. The $W^{1,3}(\rho)$ bound makes both the potential and logarithmic-Jacobian remainders integrable, with total error $Ch^3$. The cutoff argument in Lemma~\ref{lem:analytic-descent} justifies the integrations by parts and finiteness of the coefficients used here. The first-order terms combine into $-h\int \nabla\widetilde\Phi\cdot\nabla f\,\rho$. The relative-score pairing~\eqref{eq:weak-oracle-entropy} gives
}\[
 \blue{-h\int \nabla\widetilde\Phi\cdot\nabla f\,\rho =-\frac{h}{\beta}\int|\nabla f|^2\rho-h^2\beta^{-1}\int\nabla(\mathcal L^*\rho/\rho)\cdot\nabla f\,\rho+O(h^3).
}\]
{The second term equals $h^2\int \frac{|\mathcal L^*\rho|^2}{\rho}$ by that integration-by-parts identity. In the quadratic terms, the first-order $L^2$ bound permits replacing $\nabla\widetilde\Phi$ and $D^2\widetilde\Phi$ by $\nabla\Phi$ and $D^2\Phi$, with an $O(h^3)$ total error. Thus
}\begin{align*}
 \blue{\KL(F_\#\rho\|\rho^*)-\KL(\rho\|\rho^*) =}{}&\blue{-\frac{h}{\beta}\int|\nabla f|^2\rho+h^2\int \frac{|\mathcal L^*\rho|^2}{\rho}}\\
 &\blue{+\frac{h^2}{2}\int\left\{ \beta^{-1}(\nabla f)^{\mathsf T}\nabla^2V\nabla f +\beta^{-2}\norm{\nabla^2f}_{\rm F}^2\right\}\rho+O(h^3). }\end{align*}
{The spatial identity underlying Lemmas~\ref{Lemma_FI_derivative} and~\ref{lemma_second_order} is
}\[
 \blue{\int\{\beta^{-1}(\nabla f)^{\mathsf T}\nabla^2V\nabla f +\beta^{-2}\norm{\nabla^2f}_{\rm F}^2\}\rho =\int \frac{|\mathcal L^*\rho|^2}{\rho}+\beta^{-2}\int (\nabla f)^{\mathsf T}\nabla^2f\nabla f\,\rho.
}\]
{It follows directly by weighted }integration by parts{, so it applies to the present state by the cutoff argument in Lemma~\ref{lem:analytic-descent}. Substitution proves the formula. The potential remainder is bounded by $Ch^3\|D^3V\|_\infty\int|\nabla\widetilde\Phi|^3\rho$, and the matrix remainder by $Ch^3\int\|D^2\widetilde\Phi\|_F^3\rho$. The linear replacement uses~\eqref{eq:weak-oracle-entropy}. For the quadratic replacements, $\|\nabla\Phi\|_{H^1(\rho)}\le C M+Mh$ follows from~\eqref{eq:strong-taylor-assumption} and H\"older's inequality; Cauchy--Schwarz therefore gives an $O(h)$ change in each quadratic coefficient}.
\end{proof}

\paragraph{Auxiliary coefficient identities.}
 {Assume that the displayed integrals are finite and boundary terms vanish; derivatives are understood distributionally where necessary. With $\mathcal D=\mathcal D_k^\beta$, $f=\log(\rho_k/\rho^*)$ and $\Psi=\nabla f(\nabla f)^{\mathsf T}$,
}\begin{align*}
{-\int\mathcal D f\,(1+f)\rho_kdx }&{=\beta^{-1}\int|\nabla f|^2\rho_kdx,}\\
{\int\mathcal D^2f\,(1+f)\rho_kdx }&{=\int|\mathcal D f|^2\rho_kdx,}\\
{\rho_k\zeta_k}&{=\beta^{-2}\nabla\cdot(\rho_k\Psi),}\\
{\int\zeta_k\cdot\nabla f\,\rho_kdx }&{=-\beta^{-2}\int\nabla f^{\mathsf T}\nabla^2f\nabla f\,\rho_kdx. }\end{align*}
{These identities follow by weighted integration by parts; the change-of-variables formula connects them to the entropy increment. }

\begin{proof}[Derivation of an auxiliary Gaussian entropy estimate]
 {Assume the Gaussian conditions of Theorem~\ref{thm:uniform-C31}. We derive the additional polynomial recurrence below, with $M_0=\beta^{-2}\int|\nabla\log(\rho_0/\rho^*)|^4\rho_0$. Unless another evaluation time is displayed, time derivatives are evaluated at the start of the Fokker--Planck flow with initial density $\rho_k$. Here $\rho_{t_k+h}$ denotes this restarted flow after time $h$.} Substituting the expression \eqref{dt2_KL} into \eqref{eqn:decay_KL_1}, we obtain
\begin{align}
\label{eqn:decay_KL_2}
&\KL(\rho_k\|\rho^*) - \KL(\rho_{k+1}\|\rho^*) \\
\geq\ & h\beta^{-1}\FI(\rho_k\|\rho^*) - \frac{3h^2}{4}\frac{d^2}{dt^2}\KL(\rho_k\|\rho^*) \notag\\
&\quad - \frac{h^2}{2}\beta^{-2}\bigg[ \int\!\left\|\nabla^2\log\frac{\rho_k}{\rho^*}\right\|_{\mathrm{F}}^2 \rho_k \, dx + \int\!\left\|\nabla\log\frac{\rho_k}{\rho^*}\right\|_2^4 \rho_k \, dx \bigg] + \mathcal{O}(h^3)\notag\\
=\ & h\beta^{-1}\FI(\rho_k\|\rho^*) - h^2\frac{d^2}{dt^2}\KL(\rho_k\|\rho^*) \notag\\
&\quad + \frac{h^2\beta^{-1}}{2} \int\!\left\langle \nabla\log\frac{\rho_k}{\rho^*},\, \nabla^2V \nabla\log\frac{\rho_k}{\rho^*} \right\rangle \rho_k \, dx - \frac{h^2\beta^{-2}}{2} \int\!\left\|\nabla\log\frac{\rho_k}{\rho^*}\right\|_2^4 \rho_k \, dx + \mathcal{O}(h^3)\,,\notag
\end{align}
where we used the inequality
\[
\|A\|_{\mathrm{F}}^2 + \|x\|_2^4 \geq -2 \sum_{ij} x_i a_{ij} x_j = -2 x^T A x\,.
\]

For the fourth-order term, Lemma~\ref{Lemma_4th_order} shows that, along the continuous Fokker--Planck flow,
\[
\frac{\partial}{\partial t} \int \left\|\nabla\log\frac{\rho_t}{\rho^*}\right\|_2^4 \rho_t \, dx \leq -4\alpha \int \left\|\nabla\log\frac{\rho_t}{\rho^*}\right\|_2^4 \rho_t \, dx\,,
\]
which implies exponential decay.

 \blue{Under the same Gaussian conditions as Theorem~\ref{thm:uniform-C31}, the Fisher information and the KL divergence are smooth functions of Gaussian parameters on a compact positive-definite set. The exact flow has bounded parameter derivatives, and the BRWP and exact maps agree through first order. Consequently, the Taylor remainder below is uniform, and the Fisher information of the BRWP update and the exact solution after one step differ by $O(h^2)$}. Hence,
\begin{align}
\label{eqn:discrete_dt2KL}
\frac{d^2}{dt^2}\KL(\rho_k\|\rho^*) &= \frac{1}{h} \bigg[ \frac{d}{dt}\KL(\rho_{t_k+h}\|\rho^*) - \frac{d}{dt}\KL(\rho_k\|\rho^*) \bigg] + \mathcal{O}(h) \notag\\
&= \frac{\beta^{-1}}{h} \big[ \FI(\rho_k\|\rho^*) - \FI(\rho_{k+1}\|\rho^*) \big] + \mathcal{O}(h)\,,
\end{align}
where we used
\[
\frac{d}{dt}\KL(\rho_k\|\rho^*) = -\beta^{-1}\FI(\rho_k\|\rho^*)\,, \qquad \FI(\rho_{k+1}\|\rho^*) = \FI(\rho_{t_k+h}\|\rho^*) + \mathcal{O}(h^2)\,.
\]

 {Substitute this identity and~\eqref{eqn:4th_exp} into~\eqref{eqn:decay_KL_2}. The leading $h\beta^{-1}\FI(\rho_k\|\rho^*)$ cancels the same term from the negative second-derivative term in~\eqref{eqn:decay_KL_2}, leaving
}\begin{align}{\label{eqn:KL_decay_pre}
\KL(\rho_{k+1}\|\rho^*)+\frac{h}{\beta}\FI(\rho_{k+1}\|\rho^*) \le\KL(\rho_k\|\rho^*)-\frac{\alpha h^2}{2\beta}\FI(\rho_k\|\rho^*) +\frac{h^2}{2}M_0e^{-4\alpha kh}+Ch^3, }\end{align}
where  {$M_0=\beta^{-2}\int|\nabla\log(\rho_0/\rho^*)|^4\rho_0$}.

The logarithmic Sobolev inequality~\eqref{PL_ineq} applied to~\eqref{eqn:KL_decay_pre} gives
\begin{align*}
\left(1+2\alpha h\right)\KL(\rho_{k+1}\|\rho^*) \leq \big(1-\alpha^2 h^2\big)\KL(\rho_k\|\rho^*) + \frac{h^2}{2}M_0 \exp(-4\alpha hk) + \mathcal{O}(h^3)\,,
\end{align*}
 {and hence, for small $h<\alpha^{-1}$,
}\begin{align}
{\label{eqn:KL_one_step_final} \KL(\rho_{k+1}\|\rho^*) }&{\leq \frac{1-\alpha^2 h^2}{1+2\alpha h}\KL(\rho_k\|\rho^*) + \frac{h^2}{2(1+2\alpha h)}M_0 \exp(-4\alpha hk) + \mathcal{O}(h^3)\notag}\\
&{= \left(1-h\alpha\frac{2+\alpha h}{1+2\alpha h}\right)\KL(\rho_k\|\rho^*) + \frac{h^2}{2}M_0 \exp(-4\alpha hk) + \mathcal{O}(h^3) \notag}\\
&{\le \left[1-2\alpha h + 3\alpha^2 h^2\right]\KL(\rho_k\|\rho^*) + \frac{h^2}{2}M_0 \exp(-4\alpha hk) + \mathcal{O}(h^3)\,.\notag }\end{align}
{The exact coefficient is bounded above by $1-2\alpha h+3\alpha^2h^2$, since their difference is $6\alpha^3h^3/(1+2\alpha h)\ge0$. This proves the auxiliary Gaussian recurrence. The general KL bounds~\eqref{eqn:KL_one_step_sec_4} and~\eqref{eq:original-gaussian-transient} follow directly from Lemma~\ref{lem:analytic-descent} under Assumption~\ref{assump:SO-weak}. }\end{proof}

{The next lemma sums a recurrence with a geometric contraction factor and an exponentially decaying error. It applies to the Gaussian recurrence above.}
\begin{lemma}
\label{Lemma_sequence_converge}
For a sequence \blue{with $a_0\ge0$ }that satisfies
\[
a_{k+1} \le (1-c_1h)a_k + h^2c_2 e^{-c_3kh} + \mathcal{O}(h^3)\,,
\]
where $h,c_i>0$\blue{, }$c_1h<1$, \blue{and the one-step remainder is bounded above by $C_Rh^3$, $C_R\ge0$, uniformly in $k$, }we have for all $k\ge 1$
 \[
{a_k \le (1-c_1h)^k a_0 + h^2c_2\,k\,\max\!\big\{(1-c_1h)^{k-1},\;e^{-c_3(k-1)h}\big\} + \blue{\frac{C_R}{c_1}h^2}\,.
}\]
In particular, for all $k\ge 1$,
 \[
{a_k \le e^{-c_1kh}a_0 + h^2c_2\,k\,e^{-\min\{c_1,c_3\}(k-1)h} + \blue{\frac{C_R}{c_1}h^2}\,.
}\]
\end{lemma}

\begin{proof}
 {Write $q=1-c_1h$. Iterating the one-step inequality gives
}\[
{a_{k+1}\le q^{k+1}a_0+h^2c_2\sum_{j=0}^kq^je^{-c_3(k-j)h} +C_Rh^3\sum_{j=0}^kq^j.
}\]
{The final sum is at most $(c_1h)^{-1}$, so its contribution is bounded by $(C_R/c_1)h^2$. }

For the term under summation, set
\[
b_j := (1-c_1h)^j e^{-c_3(k-j)h} = e^{-c_3kh}\Big((1-c_1h)e^{c_3h}\Big)^j\,.
\]
Since $b_j\ge0$, we may bound
\[
\sum_{j=0}^k b_j \le (k+1)\max_{0\le j\le k} b_j\,.
\]
Because $b_j$ is geometric in $j$, its maximum on $\{0,\dots,k\}$ is attained at $j=0$ or $j=k$. Hence,
\[
\max_{0\le j\le k} b_j = \max\{b_0,b_k\} = \max\{e^{-c_3kh},(1-c_1h)^k\}\,.
\]
Therefore
\[
\sum_{j=0}^k(1-c_1h)^j e^{-c_3(k-j)h} \le (k+1)\max\!\big\{(1-c_1h)^k,\;e^{-c_3kh}\big\}\,.
\]
Substituting this into the inductive estimate yields
 \[
{a_{k+1}\le (1-c_1h)^{k+1}a_0 + h^2c_2\,(k+1)\max\!\big\{(1-c_1h)^k,\;e^{-c_3kh}\big\} + \frac{C_R}{c_1}h^2\,.
}\]
Replacing $k+1$ by $k$ gives the claim.

\end{proof}

\section{Laplace Approximation to the Kernel Formula}
\label{sec_Laplace}

We use Laplace's method to obtain pointwise expansions of the kernel formula \eqref{rho_T_BRWP} on fixed compact sets. These expansions describe its approximation to the Fokker--Planck evolution \eqref{eqn_FP} for small $h$ at fixed $\beta$. 

\paragraph{{\color{black}Hypotheses for the Laplace expansions.}}
{\blue{For the Laplace calculations below, use the global Hessian bounds~\eqref{eq:strong-convexity}, $V\in C^6$ on a neighborhood of the fixed compact domains, and sufficiently small $h$. For the numerator calculation, assume additionally that $\rho_0$ is a probability density that is $C^6$ and bounded below by a positive constant on $U_r$. The constants depend on these local bounds and the separation of the domains. All stated remainders are uniform on the fixed smaller compact set.} }

We use the two-term Laplace expansion~\citep{integral_eqn}:
\begin{equation}
\label{Laplace_general}
    \int_{\R^d} g(x) \exp\left(-\frac{f(x)}{h} \right) dx = (2\pi h)^{d/2}\exp\left(-\frac{f(x^*)}{h}\right)\left[ \frac{g(x^*)}{\det(\nabla^2 f(x^*))^{1/2}} + \frac{h}{2}H_1(x^*)  +\mathcal{O}(h^2)\right]\,.
\end{equation}
Here $x^*=\argmin_x f(x)$ is the unique minimizer of $f$, and $H_1$ gives the first-order correction. Section 8.3 of~\citet{integral_eqn} derives the formula below. Writing $f_{p} = \frac{\partial}{\partial x_p}f$ for $p = 1,\cdots,d$, we have
 \begin{align}
{\label{eqn_H1} H_1(x^*) =  \det(\nabla^2 f(x^*))^{-1/2}}&{\bigg\{-f_{srq}B_{sq}B_{rp}g_p + \text{Tr}(CB) }\\&{+g\left[f_{pqr}f_{stu}\left(\frac{1}{4}B_{ps}B_{qr}B_{tu}+\frac{1}{6}B_{ps}B_{qt}B_{ru}\right)-\frac{1}{4}f_{pqrs}B_{pr}B_{qs} \right] \bigg\}_{x=x^*}\,,\notag }\end{align}
where we use the summation convention to sum over all indices from $1$ to $d$. The matrices in the expression are defined as
\[
B=(\nabla^2 f(x^*))^{-1}\,,\qquad C=\nabla^2 g(x^*)\,.
\]
For the general formula, assume $f$ has a nondegenerate unique minimum, is $C^6$ near it, and $g$ is $C^4$ there; outside a fixed neighborhood assume a positive gap above the minimum and $\int|g|e^{-f/h_1}<\infty$ for some $h_1>0$. Taylor expansion about the minimum, Gaussian odd-moment cancellation, and the exponentially small tail give the displayed $O(h^2)$ bracket remainder. In the applications below, the quadratic term gives a uniformly positive definite Hessian. A cutoff localizes the amplitude where needed. 

Let $\rho(t)$ be the Fokker--Planck density with $\rho(t_0)=\rho_0$. Applying the Laplace expansion to the kernel formula~\eqref{rho_T_BRWP} gives the one-step approximation in Theorem~\ref{Thm_rho_T}:
\[
\K_V^h\rho_0(x) = \rho_0(x) + h\left.\frac{\partial\rho(t,x)}{\partial t}\right|_{t=t_0} + \mathcal{O}(h^2)\,.
\]
We first expand the normalization term, then the numerator. The quadratic part of each phase controls the location of its minimum and the remainder.

\subsection{Approximation to the normalization term}

Firstly, we derive the approximation result for the denominator in the kernel formula \eqref{rho_T_BRWP}, as shown in the following lemma.

\begin{lemma}
\label{lemma_lap_deno}
 {Uniformly for $y$ in a fixed compact set, under the hypotheses stated at the start of this appendix and for sufficiently small $h$}, we have
    \begin{align}
    \label{normal_approx}
& \left(\frac{\beta}{4\pi h}\right)^{d/2}\int_{\mathbb{R}^d} \exp\left(-\frac{\beta}{2} \left(V(z) + \frac{\|z - y\|_2^2}{2h}\right)\right) dz \\
 = &\frac{1}{1+h/2\Delta V(\ts_{y})} \exp\left(-\frac{\beta}{2}\bigg(V(\ts_{y})+\frac{\|\ts_{y}-y\|_2^2}{2h}\bigg)\right) + \mathcal{O}(h^2)\,,\notag
    \end{align}
    where
    \[
   s_y = y-h\nabla V(s_y)\,,\quad \ts_{y} = y - h\nabla V(y)  \,.
    \]
\end{lemma}

\begin{proof}
For the normalization term in \eqref{rho_T_BRWP} and fixed \(y\), we have
 \begin{align}
{\label{Lap_app_mid} }& {\left(\frac{\beta}{4\pi h}\right)^{d/2}\int_{\mathbb{R}^d} \exp\left(-\frac{\beta}{2} \left(V(z) + \frac{\|z - y\|_2^2}{2h}\right)\right) dz }\\
{= }&{\left(\frac{\beta}{4\pi h}\right)^{d/2}\int_{\mathbb{R}^d} \exp\left(-\frac{\beta}{2h}\left(hV(z)+\frac{\|z-y\|_2^2}{2}\right)  \right) dz\notag}\\
{= }&{\exp\left(-\frac{\beta}{2h}\left(hV(s_y)+\frac{\|s_y-y\|_2^2}{2}\right)\right) \left[ \frac{1}{\det(I+h\nabla^2 V(s_y))^{1/2}} + \frac{h}{\beta}H_1(s_y) + \mathcal{O}\!\left((2h/\beta)^2\right)\right]\,,\notag }\end{align}
where \(s_y\) minimizes the phase
\[
s_y := \argmin_{z}\left\{hV(z)  +\frac{\|y-z\|_2^2}{2}\right\}\,.
\]
To bound \(H_1(s_y)\), substitute the following quantities into~\eqref{eqn_H1}:
\[
g(z) = 1\,, \quad f(z) = hV(z) + \frac{\|z-y\|_2^2}{2}\,, \quad f_{pq}(z) = h\frac{\partial^2 V}{\partial z_p\partial z_q}(z) + \delta_{pq}\,, \quad f_{pqr}(z) =h\frac{\partial^3 V }{\partial z_p\partial z_q \partial z_r}(z) \,.
\]
The derivatives of $f$ of order three and four are $O(h)$, and $B=(I+h\nabla^2V)^{-1}$ is uniformly bounded.

Equation~\eqref{eqn_H1} therefore gives
\[
H_1(s_y) = \mathcal{O}(h)\,.
\]
Substituting into~\eqref{Lap_app_mid} gives
\begin{align}
\label{Lap_app_mid_2}
& \left(\frac{\beta}{4\pi h}\right)^{d/2}\int_{\mathbb{R}^d} \exp\left(-\frac{\beta}{2} \left(V(z) + \frac{\|z - y\|_2^2}{2h}\right)\right) dz \\
= &\frac{1}{\det(I+h\nabla^2 V(s_y))^{1/2}} \exp\left(-\frac{\beta}{2h}\left(hV(s_y)+\frac{\|s_y-y\|_2^2}{2}\right)\right) + \mathcal{O}(h^2)\,.\notag
\end{align}

The strictly convex phase has a unique minimizer \(s_y\), characterized by
\[
 h\nabla V(s_y )  +s_y = y\,.
\]

Approximate \(s_y\) by
\begin{equation}
    \label{def_tildes}
    \ts_{y} := y - h\nabla V(y)  \,.
\end{equation}

By definition of $\ts_y$, we have
\[
\|\ts_{y}-s_y\|_2 = h\|\nabla V(s_y)-\nabla V(y)\|_2 \leq L h\|s_y-y\|_2 =L h^2 \|\nabla V(s_y)\|_2 = \mathcal{O}(h^2)\,.
\]
The Hessian bound makes $\nabla V$ globally $L$-Lipschitz; the compact-set bounds give the final $O(h^2)$ estimate.

Taylor expansion of the determinant and its square root gives
\begin{align}
\label{eqn_Taylor_det}
&\det(I+h\nabla^2 V(s_y)) = 1+ h\Delta V(s_y) + \mathcal{O}(h^2)\,,\\
&\det(I+h\nabla^2 V(s_y))^{1/2} =1 + \frac{h}{2}\Delta V(s_y) + \mathcal{O}(h^2)\,, \notag
\end{align}
 {uniformly on the compact set; a restriction such as $h dL\le1$ suffices for these expansions}.

 {The phase is stationary at $s_y$, so replacing $s_y$ by $\tilde s_y$ changes $hV(s_y)+|s_y-y|^2/2$ by $O(h^4)$ and the exponent after division by $h$ by $O(h^3)$. The Hessian factor changes by $O(h^3)$. Combining~\eqref{Lap_app_mid}--\eqref{eqn_Taylor_det} proves the assertion. An independent heat-kernel calculation gives the same first-order term:
}\begin{equation}{\label{eq:denominator-heat}
q_h(y)=\psi(y)\left[1+h\left(\frac{\beta}{4}|\nabla V(y)|^2-\frac{1}{2}\Delta V(y)\right)\right]+O(h^2). }\end{equation}
{Indeed, localize $\psi$ and apply $\G_hg=g+(h/\beta)\Delta g+O(h^2)$; the bounded discarded part is exponentially small. This last calculation needs only local $C^4$ potential regularity. }\end{proof}\hspace*{\fill}

\subsection{Approximation to the kernel formula}

Letting
\[
V_0(x) = -\beta^{-1}\log\rho_0(x)\,,
\]
we use an approach similar to that in the proof of Lemma \ref{lemma_lap_deno} to derive the asymptotic approximation of \(\K_V^h\rho_0\).

\begin{theorem}
\label{Thm_rho_T}
 {Under the numerator hypotheses stated at the start of this appendix, uniformly for $x$ in any fixed compact subset of $U$, }we have
\begin{equation}
  \label{rho_T_asym}
    \K_V^h\rho_0(x) = \rho_0(x)\big[1-\beta h\nabla (V-V_0)\cdot \nabla V_0(x)+h\Delta(V-V_0)(x)\big] + \mathcal{O}(h^2)\,.
\end{equation}

If $\rho$ solves the Fokker--Planck equation with $\rho(t_0,\cdot)=\rho_0$, then
\begin{equation}
\label{rho_T_FP}
     \K_V^h\rho_0(x)  = \rho_0 + h\beta^{-1}\nabla\cdot\left(\nabla\log\frac{\rho_0}{\rho^*}\rho_0\right) + \mathcal{O}(h^2) = \rho_0 + h\frac{\partial\rho}{\partial t}\bigg|_{t=t_0} + \mathcal{O}(h^2)\,.
\end{equation}
\end{theorem}

\begin{proof}
 \noindent{\emph{Step 1. Localization and the stationary point.} Choose $\chi\in C_c^\infty(U_0)$ equal to one on a neighborhood of $\overline U$. The normalized-kernel tail estimate in the proof of Lemma~\ref{lem:PV-weak-local}, integrated against the total mass of $\rho_0$, permits replacing the input by $\chi\rho_0$ with an $O(h^n)$ error on $U$ for every fixed $n$. We first expand at the exact proximal point $s_y$. The normalization calculation gives, uniformly on $\supp\chi$,
}\[
{q_h(y)=\frac{\exp[-\frac{\beta}2\{V(s_y)+|s_y-y|^2/(2h)\}]}{1+(h/2)\Delta V(s_y)}+O(h^2).
}\]
{All these denominators are bounded away from zero on the compact set, so inversion preserves an $O(h^2)$ remainder. After Gaussian convolution, that remainder is still $O(h^2)$. Consequently, the kernel formula becomes
}\begin{align}{\label{rho_T_mid}
\K_V^h\rho_0(x) }&{=\frac{e^{-\beta V(x)/2}}{(4\pi h/\beta)^{d/2}} \int\chi(y)\left[1+\frac{h}{2}\Delta V(s_y)\right] e^{-\beta f_h(y)/(2h)}\,dy+O(h^2), }\end{align}
where  {$V_0=-\beta^{-1}\log\rho_0$ is used only on $U_r$, and
}\begin{equation}{\label{def_f_rho_T}
f_h(y)=\frac{|x-y|^2}{2}+2hV_0(y)-hV(s_y)-\frac{|y-s_y|^2}{2}. }\end{equation}
{The amplitude is $g_h(y)=\chi(y)[1+(h/2)\Delta V(s_y)]$. Near the stationary point it is $1+O(h)$ in $C^4$; the phase is $|x-y|^2/2+O(h)$ in $C^6$. These uniform bounds follow by differentiating $s_y+h\nabla V(s_y)=y$; the inverse matrix $I+h\nabla^2V(s_y)$ is uniformly invertible. Equivalently $V(s_y)+|y-s_y|^2/(2h)$ is the smooth Moreau envelope, whose gradient is $\nabla V(s_y)$. }

 {For small $h$, the Hessian of $f_h$ is positive definite on $U_0$, and there is a unique stationary point }$r_x$  {near $x$, inside the region where $\chi=1$. The complement of a fixed neighborhood of $x$ in $\supp\chi$ has a uniform phase gap. Differentiating the phase gives the exact relation
}\[
{r_x=x+h\{\nabla V(s_{r_x})-2\nabla V_0(r_x)\} =x+h\nabla(V-2V_0)(r_x)+O(h^2).
}\]
{Here the cancellation of the $D s_y$ terms uses $y-s_y=h\nabla V(s_y)$. Replacing $\nabla V(s_{r_x})$ by $\nabla V(r_x)$ introduces an $O(h)$ error. Define the linear approximation
}\[
{\tilde r_x=x+h\nabla(V-2V_0)(x),\qquad |\tilde r_x-r_x|=O(h^2).
}\]
{Differentiating once more gives
}\begin{align*}
{\nabla^2f_h(r_x) }&{=I+2h\nabla^2V_0(r_x) -h\nabla^2V(s_{r_x})(I+h\nabla^2V(s_{r_x}))^{-1}}\\
&{=I+h\nabla^2(2V_0-V)(r_x)+O(h^2). }\end{align*}

 \noindent{\emph{Step 2. Laplace expansion and replacement of the stationary points.} The third and fourth derivatives of the phase are $O(h)$, and the first and second derivatives of the amplitude are $O(h)$ near the stationary point. Thus its $H_1$ term is $O(h)$. Laplace expansion yields
}\begin{align}{\label{rho_T_mid_2}
\K_V^h\rho_0(x) }&{=\frac{e^{-\beta V(x)/2}[1+(h/2)\Delta V(s_{r_x})]}{\det(I+h\nabla^2(2V_0-V)(r_x)+O(h^2))^{1/2}} e^{-\beta f_h(r_x)/(2h)}+O(h^2)\notag}\\
&{=\frac{e^{-\beta V(x)/2}[1+(h/2)\Delta V(\tilde s_{\tilde r_x})]}{1+(h/2)\Delta(2V_0-V)(\tilde r_x)}\notag}\\
&{\quad\cdot\exp\left[-\frac{\beta}{2}\left\{ \frac{|x-\tilde r_x|^2}{2h}-\frac{|\tilde r_x-\tilde s_{\tilde r_x}|^2}{2h} +2V_0(\tilde r_x)-V(\tilde s_{\tilde r_x})\right\}\right]+O(h^2). }\end{align}
{Replacing $r_x$ by $\tilde r_x$ changes the stationary phase by $O(h^4)$, hence the exponent after division by $h$ by $O(h^3)$. Replacing $s_y$ by $\tilde s_y=y-h\nabla V(y)$ has the same stationary-phase justification as in Lemma~\ref{lemma_lap_deno}. Thus the linearized expression has the stated remainder. }

 \noindent{\emph{Step 3. The first-order coefficient.} }Substitute the definitions of $\tilde r_x$ and $\tilde s_y$ into the exponent in~\eqref{rho_T_mid_2}. Omitting the factor $-\beta/2$, we obtain
\begin{align*}
&2V_0(\tr_x) - V(\ts_{\tr_x}) + \frac{h}{2}\|\nabla(V-2V_0)(x)\|_2^2-\frac{h}{2}\|\nabla V(\tr_x)\|_2^2 + \mathcal{O}(h^2)\\
=&2V_0\left(x+h\nabla (V-2V_0)(x)\right)- V\left(x-2h\nabla V_0(x)\right)  + \frac{h}{2}\|\nabla(V-2V_0)(x)\|_2^2-\frac{h}{2}\|\nabla V(x)\|_2^2 + \mathcal{O}(h^2)\\
=& 2V_0(x) +2 h\nabla V_0\cdot\nabla (V-2V_0)(x) - V(x) + 2h\nabla V\cdot\nabla V_0(x) \\&+ \frac{h}{2}\|\nabla(V-2V_0)(x)\|_2^2-\frac{h}{2}\|\nabla V(x)\|_2^2 + \mathcal{O}(h^2)\\
=&2V_0(x) - V(x) + 2h \nabla(  V- V_0)\cdot \nabla V_0(x) + \mathcal{O}(h^2)\,,
 \end{align*}
where we have used Taylor expansions of $V_0$ and $V$, and the relation
\[
\tilde{s}_{\tr_x} = \tr_x - h\nabla V(\tr_x) = x + h\nabla(V-2V_0)(x) - h\nabla V(x) + \mathcal{O}(h^2) = x-2h\nabla V_0(x) + \mathcal{O}(h^2)\,.
\]

Expanding the prefactor in~\eqref{rho_T_mid_2} gives
\[
\frac{1+ \frac{h}{2} \Delta V(\ts_{\tr_x})}{1 + \frac{h}{2} \Delta(2V_0-V)(\tr_{x})} = \bigg[1+\frac{h}{2}\Delta V(x)\bigg]\bigg[1-\frac{h}{2}\Delta(2V_0-V)(x)\bigg] + \mathcal{O}(h^2) = 1+h \Delta(V- V_0)(x) + \mathcal{O}(h^2)\,,
\]
 {for sufficiently small $h$, uniformly on the fixed compact set}. Combining the exponent and prefactor expansions proves Theorem~\ref{Thm_rho_T}.
\end{proof}\hspace*{\fill}

 \paragraph{{\color{black}Derivative version and a direct heat-kernel proof.}}
\blue{The preceding first-order coefficient also holds without positivity of the input: if $r\in\{0,1,2\}$ and $V$ and the probability density $\rho$ are $C^{r+4}$ on the fixed larger neighborhood, then
}\begin{equation}\blue{\label{eq:pointwise-derivative}
\|\K_V^h\rho-\rho-h\mathcal L^*\rho\|_{C^r(U)}\le Ch^2. }\end{equation}
\blue{For $r=0$, local $C^{3,1}$ bounds on $V$ and $\rho$ suffice: localize and mollify, use the uniform fourth weak-derivative bounds in the heat remainder, and pass to the limit in $C^0$ on the smaller domain. This derivative version is part of Theorem~\ref{Thm_rho_T}. Its $C^0$ case is }sufficient for the  \blue{weak score theorem; the $r=2$ case gives the additional $C^1$ score remainder under the stated $C^6$ bounds and a positive lower bound on the density.
}\begin{proof}[Proof of the derivative version]
{Choose nested neighborhoods of $\overline U$ inside $U_r$ and a smooth cutoff $\chi$ supported in $U_r$ that is one on the larger of these neighborhoods. Localized heat expansion as in Lemma~\ref{lemma_lap_deno} gives
}\[
 \blue{q_h=\psi+\tfrac h\beta\Delta\psi+O_{C^r}(h^2), \qquad q_h-\psi=O_{C^{r+2}}(h),
}\]
{on the needed interior sets. Localization also gives a uniform $C^{r+4}$ bound and a uniform positive lower bound for $q_h$ on $\operatorname{supp}\chi$ for small $h$. Thus $g_h=\chi\rho/q_h$, extended by zero, has a uniformly bounded $C^{r+4}$ norm. Apply the heat expansion to $g_h$:
}\[
 \blue{\G_hg_h=g_h+\tfrac h\beta\Delta g_h+O_{C^r}(h^2).
}\]

{We justify discarding the complementary input, which need not be smooth. The upper Hessian bound and Gaussian integration give
}\[
 \blue{q_h(y)\ge(1+Lh)^{-d/2}\psi(y).
}\]
{For $x$ in a fixed compact neighborhood of $\overline U$, $V(y)-V(x)\le\nabla V(x)\cdot(y-x)+(L/2)|y-x|^2$. Consequently, for small $h$, the normalized kernel and its $x$ derivatives through order $r$ obey bounds of the form
}\[
 \blue{\left|\partial_x^\gamma \frac{\psi(x)G_h(x-y)}{q_h(y)}\right| \le C h^{-d/2-r}(1+|x-y|)^r e^{-c|x-y|^2/h}, \qquad |\gamma|\le r.
}\]
{Where $1-\chi(y)\ne0$, the distance from $y$ to $U$ is bounded below. Integration against the probability density $\rho$ therefore makes this contribution $O(h^n)$ for every fixed $n$, including after $r$ derivatives. The same localization argument controls the distant contribution when differentiating $q_h$. }

{On $U$, where $\chi=1$, combine these expansions and use $\rho/q_h-\rho/\psi=O_{C^{r+2}}(h)$ to obtain
}\begin{align*}
 \blue{\K_V^h\rho }&\blue{=\frac{\psi\rho}{q_h} +\frac{h}{\beta}\psi\Delta(\rho/q_h)+O_{C^r(U)}(h^2)}\\
 &\blue{=\rho+\frac{h}{\beta}\left\{ \psi\Delta(\rho/\psi)-\rho\frac{\Delta\psi}{\psi}\right\} +O_{C^r(U)}(h^2)}\\
 &\blue{=\rho+h\{\nabla\cdot(\rho\nabla V)+\beta^{-1}\Delta\rho\} +O_{C^r(U)}(h^2). }\end{align*}
{The product rule proves the last equality and agrees with~\eqref{rho_T_asym} on the positive region. }\end{proof}

\section{{Exact Kernel Estimates and Verification of Assumption~\ref{assump:SO-weak}}}
\label{sec:certificates}
{This appendix supplies the entropy estimates used in Section~\ref{sec_conv_KL}.}
\subsection{{Entropy estimates for the kernel formula}}
\label{sec:kernel-entropy}
{Define the invariant density of the kernel formula and the relative entropy with respect to this density by
}\[
 {Z_h=\int\psi\mathcal G_h\psi\dd x,\qquad \rho_h^*=Z_h^{-1}\psi\mathcal G_h\psi,\qquad D_h(\rho)=\KL(\rho\|\rho_h^*).}
\]

{We work under the Hessian bounds~\eqref{eq:strong-convexity} and use the notation of Sections~\ref{sec_review} and~\ref{sec_conv_KL}. Recall that the relative Fisher information is $\FI(q\| r)=\int|\nabla\log(q/r)|^2q\dd x$. We use the classical curvature, logarithmic Sobolev, and transportation inequalities \citep{bakrygentilledoux2014,otto2000generalization}, and the Brascamp--Lieb covariance bound \citep{carlen2013covariance}. }\oldresult{32}
\begin{proposition}[\blue{Exact kernel estimates}]
\label{prop:kernel}
\blue{The kernel $\K_V^h$ is reversible with invariant density $\rho_h^*$, and, for every probability law $\rho$,
}\begin{equation}
 \blue{D_h(\K_V^h\rho)\le(1+\alpha h)^{-2}D_h(\rho). \label{eq:kernel-contraction}
}\end{equation}
\blue{Furthermore, with $B_0=(3d^2+4)L^2/12$,
}\begin{align}
 \blue{D_h(\rho)}&\blue{\le \KL(\rho\|\rho^*)+\frac{d}{2}\log(1+Lh), \label{eq:initial-comparison}}\\
 \blue{\KL(\rho\|\rho^*)}&\blue{\le2D_h(\rho)+B_0h^2. \label{eq:target-comparison}
}\end{align}
\blue{All inequalities hold in extended values. }\end{proposition}

\begin{proof}
{The symmetric law $\Gamma_h(\dd x\dd y)=Z_h^{-1}\psi(x)G_h(x-y)\psi(y)\dd x\dd y$ has both marginals $\rho_h^*$, proving reversibility.  Suppose first that $D_h(\rho)<\infty$ and that $f=\dd\rho/\dd\rho_h^*$ is bounded and bounded away from zero.  Put $q=\K_V^h\rho=a\rho_h^*$, $\gamma_y=\Gamma_h(\dd x\mid y)$, and $\nu_y=f\gamma_y/a(y)$.  Conditional entropy decomposition and Gaussian differentiation give
}\begin{align*}
 \blue{D_h(\rho)}&\blue{=D_h(q)+\int a(y)\KL(\nu_y\|\gamma_y)\rho_h^*(\dd y),}\\[-2pt]
 \blue{\nabla\log a(y)}&\blue{=\frac{\beta}{2h} \{\mathbb E_{\nu_y}X-\mathbb E_{\gamma_y}X\}. }\end{align*}
{Since $-\nabla^2\log\gamma_y\succeq\beta(1+\alpha h)I_d/(2h)$, Talagrand's inequality and the $W_2$ bound on the difference of means yield
}\begin{equation}
 \blue{D_h(\rho)\ge D_h(q) +\frac{h(1+\alpha h)}{\beta}\FI(q\|\rho_h^*). \label{eq:entropy-fisher-loss}
}\end{equation}
{The normalized truncations $c_n^{-1}\{\min\{f,n\}+n^{-1}\}$, where $c_n=\int[\min\{f,n\}+n^{-1}]\dd\rho_h^*$, and lower semicontinuity extend this to every finite-entropy $f$; if $D_h(\rho)=+\infty$, the contraction below is automatic.  Brascamp--Lieb for $\gamma_y$, inserted in the Gaussian convolution Hessian identity, gives
}\[
 \blue{-\nabla^2\log\rho_h^*\succeq \frac{\beta\alpha(2+\alpha h)}{2(1+\alpha h)}I_d.
}\]
{The Bakry--\'Emery logarithmic Sobolev inequality therefore bounds the Fisher information in \eqref{eq:entropy-fisher-loss} below by $\beta\alpha(2+\alpha h)D_h(q)/(1+\alpha h)$.  Hence $D_h(\rho)\ge(1+\alpha h)^2D_h(q)$, proving \eqref{eq:kernel-contraction}. }

{For the entropy comparison, set $s(t)=\langle\psi,\G_t\psi\rangle$ and $\phi(t)=\log\{s(t)/s(0)\}$.  Spectral calculus for $-\beta^{-1}\Delta$ gives
}\begin{equation}
 \blue{\log\int\left(\frac{\rho_h^*}{\rho^*}\right)^2\rho^*\dd x =\phi(2h)-2\phi(h), \qquad \phi''(t)=\Var_{\nu_t}(\xi), \label{eq:spectral-reference}
}\end{equation}
{where $\xi\ge0$ is the spectral variable and $\nu_t$ its exponentially tilted law.  The second moment of $\xi$ decreases under the exponential tilt (the covariance of the increasing functions $\xi^2$ and $\xi$ is nonnegative). Thus $\phi''(t)\le\int\xi^2\dd\nu_0$.  With $g=\nabla V$ and $b=|g|^2$, integration by parts gives
}\[
 \blue{\int\xi^2\dd\nu_0 =\mathbb E_{\rho^*}\left(\frac{\Delta V}{2}-\frac{\beta b}{4}\right)^2, \qquad \beta\mathbb E_{\rho^*} b^2 =\mathbb E_{\rho^*}\{2g^{\mathsf T}(\nabla^2V)g+(\Delta V)b\}.
}\]
{To make the constant explicit, let $m_1=\mathbb E_{\rho^*} b$. Expanding the square and using the preceding integration-by-parts identity, $0\preceq\nabla^2V\preceq LI_d$, and $\mathbb E b^2\ge m_1^2$ gives
}\[
 \blue{\int\xi^2\dd\nu_0 \le \frac{d^2L^2}{4}+\frac{\beta Lm_1}{2} -\frac{3\beta^2m_1^2}{16} \le\frac{d^2L^2}{4}+\frac{L^2}{3}=B_0.
}\]
{Integrating \eqref{eq:spectral-reference} gives
}\begin{equation}
 \blue{\log\int\left(\frac{\rho_h^*}{\rho^*}\right)^2\rho^*\dd x\le B_0h^2. \label{eq:renyi-reference}
}\end{equation}
{Finally, $\mathcal{G}_h\psi/\psi\ge(1+Lh)^{-d/2}$ and $Z_h\le Z$ prove \eqref{eq:initial-comparison}; the entropy variational formula applied to $\log(\rho_h^*/\rho^*)$, followed by \eqref{eq:renyi-reference}, proves \eqref{eq:target-comparison}.  Smooth approximation and lower semicontinuity give the extended-value statements. }\end{proof}

\paragraph{Auxiliary entropy bounds for the BRWP iteration.}
{For finite $D_h(\rho)$, Proposition~\ref{prop:kernel} makes $D_h(\K_V^h\rho)$ finite and gives the kernel entropy decrease
}\begin{equation}\label{eq:kernel-loss}
 \ell_h(\rho)=D_h(\rho)-D_h(\K_V^h\rho)\ge0.
\end{equation}
\blue{For finite $D_h(\rho)$ and $D_h(\cR_h\rho)$, define
\begin{equation}\label{eq:spent-defect}
\varepsilon_h(\rho)=\bigl[D_h(\cR_h\rho)-D_h(\K_V^h\rho)-\tfrac12\ell_h(\rho)\bigr]_+,\qquad [a]_+=\max\{a,0\}.
\end{equation}
Set $\varepsilon_h(\rho)=+\infty$ otherwise. This measures the excess entropy after the BRWP step over the kernel step beyond the allowance $\tfrac12\ell_h(\rho)$.}
{Define
}\begin{equation}\blue{\label{eq:r-definition}
 r_h=\frac{1+(1+\alpha h)^{-2}}{2}. }\end{equation}
\blue{Under~\eqref{eq:strong-convexity}, the BRWP iteration satisfies
}\begin{equation}\blue{\label{eq:kernel-recurrence}
 D_h(\rho_{j+1})\le r_hD_h(\rho_j)+\varepsilon_h(\rho_j).
}\end{equation}
{For finite entropies this follows directly from~\eqref{eq:spent-defect} and~\eqref{eq:kernel-contraction}; if the right-hand side is infinite, the inequality is automatic. Iteration and the two entropy comparisons in Proposition~\ref{prop:kernel} give
}\begin{align}
 \blue{\KL(\rho_k\|\rho^*)\le{}}&\blue{2r_h^k\{\KL(\rho_0\|\rho^*)+\tfrac d2\log(1+Lh)\}+B_0h^2+2\sum_{j=0}^{k-1}r_h^{k-1-j}\varepsilon_h(\rho_j). \label{eq:master-bound}}
\end{align}
\blue{If Assumption~\ref{assump:SO-weak} holds uniformly for $0\le j<N(h)$, Lemma~\ref{lem:analytic-defect} bounds the last term, for sufficiently small $h$ and $1\le k\le N(h)$, by
}\begin{equation}\blue{\label{eq:floor}
 \frac{4C_\varepsilon(1+\alpha h)^2}{\alpha(2+\alpha h)}h^2. }\end{equation}
{Indeed,
}\[
 1-r_h=\frac{\alpha h(2+\alpha h)}{2(1+\alpha h)^2},
\]
{so geometric summation gives~\eqref{eq:floor}. For $0<\alpha h\le1$, $(2+\alpha h)/(1+\alpha h)^2\ge3/4$ and hence $r_h^k\le e^{-3\alpha kh/8}$.}

\subsection{{Verification of Assumption~\ref{assump:SO-weak} for Gaussian iterates}}
\label{sec:gaussian}

\blue{Up to translation and an additive constant, let $V(x)=x^{\mathsf T}Ax/2$, where $\alpha I_d\preceq A\preceq LI_d$, and let $\rho=\mathcal N(m,C)$ with $C\succ0$.  Then
}\begin{equation}
 \blue{\rho_h^*=\mathcal N(0,C_h^*),\qquad C_h^*=\frac{2}{\beta}(I+hA)(2I+hA)^{-1}A^{-1}. \label{eq:gaussian-reference}
}\end{equation}

\begin{proposition}[\blue{Gaussian verification of Assumption~\ref{assump:SO-weak}}]
\label{prop:gaussian}
\blue{Fix $0<M<\infty$. On the Gaussian sublevel $D_h(\rho)\le M$, Assumption~\ref{assump:SO-weak} and the bounded-test condition~\eqref{eq:weak-oracle-global} hold uniformly for sufficiently small $h$. Consequently, by Lemma~\ref{lem:analytic-defect}, there exist $h_M,C_M>0$, depending only on $M,d,\alpha,L,\beta$, such that
}\begin{equation}
 \blue{D_h(\mathcal N(m,C))\le M,\qquad 0<h\le h_M \quad\Longrightarrow\quad \varepsilon_h(\mathcal N(m,C)) \le C_Mh^3. \label{eq:gaussian-h3}
}\end{equation}
\blue{For every nondegenerate Gaussian initial law, these assumption bounds hold uniformly for all iterations: there exist $M_0,h_0,C>0$ such that, for $0<h\le h_0$,
}\begin{equation}
 \blue{D_h(\rho_k)\le M_0,\qquad \varepsilon_h(\rho_k)\le Ch^3,\qquad k\ge0. \label{eq:gaussian-all-time}
}\end{equation}
\blue{Thus all conclusions of Theorem~\ref{Thm_KL_decay} hold for every $k$ with constants independent of $k$. }\end{proposition}

\begin{proof}
{We first verify the assumption for Gaussian densities whose relative entropy is bounded, then show that this bound is preserved by the BRWP update. The Gaussian entropy formula and~\eqref{eq:gaussian-reference} place $m$ in a fixed ball and $C$ in a compact subset of the positive definite cone on $D_h(\rho)\le M$, uniformly for small $h$. Write $J_h=(I+hA)^{-1}$. Completing the square gives
}\begin{equation}
 \blue{\K_V^h\mathcal N(m,C)=\mathcal N(J_hm,Q_h), \qquad Q_h=J_hCJ_h+\frac{2h}{\beta}J_h. \label{eq:gaussian-kernel}
}\end{equation}
{The BRWP update map is affine,
}\begin{equation}
 \blue{U_h(x)=L_hx-\frac{h}{\beta} Q_h^{-1}J_hm, \qquad L_h=I-hA+\frac{h}{\beta} Q_h^{-1}. \label{eq:gaussian-raw-map}
}\end{equation}
{On this compact parameter set, $\Phi,\widetilde\Phi$ and $\mathcal L^*\rho/\rho$ are quadratic with bounded coefficients. Gaussian moments give the upper Hessian and weighted $W^{1,3}(\rho)$ bounds in~\eqref{eq:strong-taylor-assumption}, while $Q_h=C+O(h)$ and $Q_h^{-1}=C^{-1}+O(h)$ give the $O(h)$ weighted $H^1(\rho)$ bound. The potential is quadratic, so $D^3V=0$. A second-order expansion of the affine kernel score gives $\|e_{k,h}\|_{L^2(\rho)}\le Ch^2$. Cauchy--Schwarz, using the bounded relative Fisher information on the parameter set, proves \eqref{eq:weak-oracle-entropy}; bounded test functions also give \eqref{eq:weak-oracle-global}. This verifies Assumption~\ref{assump:SO-weak} uniformly on the sublevel, and Lemma~\ref{lem:analytic-defect} proves~\eqref{eq:gaussian-h3}.}

\blue{Choose $M$ strictly larger than the small-$h$ bound for $D_h(\rho_0)$. As long as $D_h(\rho_k)\le M$, recurrence~\eqref{eq:kernel-recurrence} and the just-verified one-step entropy estimate give
}\[
 \blue{D_h(\rho_{k+1})\le r_hD_h(\rho_k)+C_Mh^3<M.}
\]
\blue{For sufficiently small $h$ this follows from $1-r_h\asymp h$. The affine update preserves Gaussianity and nondegeneracy, since $L_h=I+O(h)$ uniformly on this compact parameter set. Induction therefore keeps every iterate in the sublevel and makes all assumption bounds uniform for all time.}
\end{proof}

\subsection{{Verification of the periodic version of Assumption~\ref{assump:SO-weak}}}
\label{sec:torus}

\blue{For periodic potentials, we verify the periodic version of Assumption~\ref{assump:SO-weak} using uniform positivity and smoothness bounds. The strong-convexity condition~\eqref{eq:strong-convexity} is replaced by the hypotheses below. Let $\T^d=(\R/2\pi\mathbb Z)^d$ carry normalized Haar measure.  Fix $\psi>0$ in $C^4(\T^d)$. Use the same heat-time scaling as in Section~\ref{sec_kernel_approx}, now with the periodic Laplacian, and define
}\[
 \blue{V=-2\beta^{-1}\log\psi,\qquad \G_t=e^{t\Delta/\beta},\qquad q_h=\mathcal{G}_h\psi,
}\]
\begin{equation}
 \blue{\K_V^h\rho=\psi \mathcal{G}_h(\rho/q_h),\qquad \rho_h^*=Z_h^{-1}\psi q_h,\qquad Z_h=\int_{\T^d}\psi q_h\dd x, \label{eq:torus-kernel}
}\end{equation}
\blue{Put $Z=\int_{\T^d}\psi^2\dd x$; $\cR_h$ is still defined by \eqref{eq:raw-map}.  Define $D_h=\KL(\,\cdot\,\|\rho_h^*)$ and $\varepsilon_h$ by \eqref{eq:spent-defect}; the map in \eqref{eq:raw-map} is understood modulo $2\pi\mathbb Z^d$. }

\oldresult{34}
\begin{proposition}[\blue{Periodic verification of Assumption~\ref{assump:SO-weak}}]
\label{prop:torus}
\blue{Fix $m>0$ and $M<\infty$.  There are $h_0,C,c>0$, depending only on $d,\beta,\psi,m,M$, such that, whenever the probability density $\rho$ satisfies
}\begin{equation}
 \blue{\rho\ge m,\qquad \norm\rho_{C^4}\le M,\qquad 0<h\le h_0, \label{eq:tube}
}\end{equation}
\blue{the periodic version of Assumption~\ref{assump:SO-weak} and the bounded-test condition of Lemma~\ref{Lemma_rho_k_iteration} hold uniformly. The map $F_{h,\rho}$ is an orientation-preserving diffeomorphism. In this periodic setting, the one-step and iterated entropy bounds are
}\begin{equation}
 \blue{\varepsilon_h(\rho)\le Ch^3. \label{eq:torus-h3}
}\end{equation}
\blue{If the BRWP iterates $\rho_j$ satisfy \eqref{eq:tube} for $0\le j<N$, then
}\begin{equation}
 \blue{D_h(\rho_k)\le e^{-ckh}D_h(\rho_0)+Ch^2, \qquad 0\le k\le N. \label{eq:torus-recurrence}
}\end{equation}
\blue{The corresponding bound for the KL divergence with respect to the target density $\rho^*=Z^{-1}\psi^2$ is
}\begin{equation}
 \blue{\KL(\rho_k\|\rho^*) \le C e^{-ckh}\KL(\rho_0\|\rho^*)+Ch^2, \qquad \rho^*=Z^{-1}\psi^2. \label{eq:torus-target-rate}
}\end{equation}
\end{proposition}

\begin{proof}
{\emph{Verification of the assumption.} The heat identity $\G_tg-g=\beta^{-1}\int_0^t\G_s\Delta g\dd s$, applied once in $C^2$ and twice in $C^0$, gives the bounds~\eqref{eq:torus-map-bounds} and the kernel expansion in~\eqref{eq:endpoint-consistency}. The positive $C^4$ bounds then give all analytic conditions in Assumption~\ref{assump:SO-weak}, including the $O(h)$ field and derivative bounds. To verify the pairings, put $\nu=\K_V^h\rho$ and
}\[
 \blue{\eta_h=\log\nu-\log\rho-h\frac{\cL^*\rho}{\rho}.}
\]
{The kernel expansion and positivity give $\|\eta_h\|_\infty\le Ch^2$ and $e_h=\nabla\eta_h$. Periodic integration by parts yields
}\[
 \blue{\left|\int w\cdot e_h\,\rho\dd x\right| =\left|\int\eta_h\nabla\cdot(\rho w)\dd x\right| \le Ch^2\norm w_{C^{1,1}(\T^d)}.}
\]
{For $w=\nabla\log(\rho/\rho^*)$, the identity $\nabla\cdot(\rho w)=\beta\cL^*\rho$ gives a uniform $L^1$ bound, proving the relative-score pairing as well.}

{\emph{Periodic entropy bounds.} We use the spectral gap of the periodic kernel operator to bound its entropy decrease. For density ratios relative to $\rho_h^*$, the corresponding Markov operator is
}\[
 \blue{P_V^hf=\frac{\mathcal{G}_h(\psi f)}{q_h}, \qquad \K_V^h(\rho_h^*f)=\rho_h^*P_V^hf.
}\]
{It is self-adjoint and Markov on $L^2(\rho_h^*)$, and its spectrum is nonnegative because
}\begin{equation}
 \blue{\langle f,P_V^hf\rangle_{\rho_h^*} =Z_h^{-1}\norm{\G_{h/2}(\psi f)}_2^2\ge0. \label{eq:positive-spectrum}
}\end{equation}
{Its Dirichlet form, compared with that of the periodic heat semigroup, gives the spectral gap
}\begin{equation}
 \blue{\gamma_h\ge \left(\frac{\min\psi}{\max\psi}\right)^2 (1-e^{-h/\beta})\ge c_{\psi,\beta}h. \label{eq:torus-gap}
}\end{equation}
{Condition \eqref{eq:tube} places $f=\rho/\rho_h^*$ in an interval $[a,b]\Subset(0,\infty)$ independent of small $h$.  For $\varphi(z)=z\log z-z+1$, uniform convexity on $[a,b]$ and the nonnegative spectrum of $P_V^h$ give
}\begin{align}
 \blue{D_h(\rho)-D_h(\K_V^h\rho) }&\blue{=\int\{P_V^h\varphi(f)-\varphi(P_V^hf)\}\dd\rho_h^*} \blue{\ge\frac{1}{2b}\langle f,(I-(P_V^h)^2)f\rangle_{\rho_h^*} \ge\frac{a}b\gamma_hD_h(\rho). \label{eq:warm-contraction}
}\end{align}
{Thus $\ell_h(\rho)\ge chD_h(\rho)\ge chD_h(\K_V^h\rho)$. }

{Set $\nu=\K_V^h\rho$, $\Phi=\beta^{-1}\log(\rho/\rho^*)$ and $\widetilde\Phi=\beta^{-1}\log(\nu/\rho^*)$.  The identity $\G_tg-g=\beta^{-1}\int_0^t\G_s\Delta g\dd s$ and \eqref{eq:tube} give, uniformly for small $h$,
}\begin{equation}
 \blue{\nu\ge m/2,\quad \norm{\nu-\rho}_{C^2}\le Ch,\quad \norm\nu_{C^3}+\norm{\nabla\widetilde\Phi}_{C^2}\le C,\quad \norm{\nabla(\widetilde\Phi-\Phi)}_{C^1}\le Ch. \label{eq:torus-map-bounds}
}\end{equation}
{\color{black}Thus $F_{h,\rho}=I-h\nabla\widetilde\Phi$ is an orientation-preserving local diffeomorphism. Its periodic displacement makes the induced torus map homotopic to the identity, hence of degree one; as a local diffeomorphism of the compact torus it is a one-sheeted covering and globally invertible.  The density formula and Taylor expansions of the heat operator and inverse map give
}\begin{align}
 \blue{\nu}&\blue{=\rho+\frac{h}{\beta}\left\{\psi\Delta(\rho/\psi) -\rho\frac{\Delta\psi}{\psi}\right\}+O_{C^0}(h^2) =\rho+h\cL^*\rho+O_{C^0}(h^2),\notag}\\
 \blue{\mu:=\cR_h\rho }&\blue{=\frac{\rho\circ F_{h,\rho}^{-1}}{\det DF_{h,\rho}\circ F_{h,\rho}^{-1}} =\rho+h\nabla\cdot(\rho \nabla\Phi)+O_{C^0}(h^2) =\rho+h\cL^*\rho+O_{C^0}(h^2), \label{eq:endpoint-consistency}
}\end{align}
{\color{black}where $\cL^*\rho=\nabla\cdot(\rho\nabla V)+\beta^{-1}\Delta\rho$. The remainders use at most four derivatives of each of $\rho$ and $\psi$; \eqref{eq:torus-map-bounds} controls the inverse expansion.  Consequently
}\begin{equation}
 \blue{X:=\chi^2(\mu\|\nu) =\int_{\T^d}\frac{(\mu-\nu)^2}{\nu}\dd x\le Ch^4. \label{eq:chi-square}
}\end{equation}

{\color{black}The entropy chain rule and Cauchy--Schwarz yield
}\begin{equation}
 \blue{D_h(\mu)-D_h(\nu) \le X+\sqrt{X\, \Var_\nu\!\left(\log\frac{\nu}{\rho_h^*}\right)}. \label{eq:entropy-chain}
}\end{equation}
{\color{black}All density ratios are uniformly bounded above and away from zero. Convexity on this fixed interval gives
}\[
 \blue{\Var_\nu\!\left(\log\frac{\nu}{\rho_h^*}\right) \le C D_h(\nu).
}\]
{\color{black}Combining this estimate with \eqref{eq:warm-contraction}, Young's inequality bounds the mixed term by $\tfrac12\ell_h(\rho)$ plus $C X/h$.  Equation \eqref{eq:chi-square} proves \eqref{eq:torus-h3}.  Iteration of the entropy contraction with \eqref{eq:torus-h3} proves \eqref{eq:torus-recurrence}. }

{Finally, $\log(\rho_h^*/\rho^*)=O_{C^0}(h)$ and both $\KL(\rho_h^*\|\rho^*)$ and $\KL(\rho^*\|\rho_h^*)$ are $O(h^2)$.  Entropy decomposition, Pinsker's inequality, and Young's inequality imply, for every fixed $\eta>0$,
}\[
 \blue{\KL(\rho\|\rho^*)\le(1+\eta)D_h(\rho)+C_{\eta,\psi,\beta}h^2,
}\]
{and the reverse comparison, completing the proof. }\end{proof}

{The bounds in Proposition~\ref{prop:torus} hold on any horizon for which~\eqref{eq:tube} is uniform. Preserving these density bounds for all iterations gives an all-time $O(h^2)$ error bound. The choice $h\asymp\sqrt\delta$ requires uniform bounds through $kh\asymp|\log h|$. }

\paragraph{Propagation on the torus.}
{We expect bounds of the form~\eqref{eq:tube} to hold uniformly for the BRWP iterates on each fixed interval $0\le kh\le T$, for all sufficiently small $h$, when the initial and target densities are strictly positive and have sufficiently high Sobolev regularity. The constants and stepsize threshold may depend on $T$. The proof strategy is to compare the iterates with the periodic Fokker--Planck solution in a Sobolev space continuously embedded in $C^4$. Positivity of the exact solution supplies the reference lower bound; the main additional step is to establish stability of the nonlinear update using centered heat-kernel estimates. This stability analysis is beyond the scope of the present paper. Uniform control over all iterations would require further estimates.}

\paragraph{{\color{black}Comparison with an $L^2$ score bound.}}
\blue{The following example distinguishes the paired remainder bounds from the stronger estimate $\norm{e_h}_{L^2(\rho)}\le Ch^2$. With $V=0$, $\beta=1$, $\rho_n(x)=1+n^{-4}\cos(nx)$ on $\T^1$, and $h_n=n^{-2}$, $n\ge2$, the densities have uniform positive $C^4$ bounds, but
}\[
 \blue{\norm{e_{h_n}}_{L^2(\rho_n)} \sim \frac{e^{-1}}{\sqrt2}\,n^{-3},\qquad h_n^2=n^{-4}.
}\]
{Indeed, $\K_V^{h_n}\rho_n=1+e^{-1}n^{-4}\cos(nx)$, and differentiation of the logarithms gives the displayed leading term. The weak pairing bounds hold, although the $L^2$ remainder is not $O(h_n^2)$. }

\subsection{Entropy estimates under Assumption~\ref{assump:SO-weak}}
\label{sec:analytic-convergence}
The Hessian bounds~\eqref{eq:strong-convexity} apply throughout this subsection. The constants below depend only on $d,\alpha,L,\beta$ and the uniform constants in Assumption~\ref{assump:SO-weak}.

\oldresult{35}
\begin{lemma}[\blue{KL divergence estimate for the BRWP update}]
\label{lem:analytic-descent}
\blue{Suppose that the current density $\rho$ satisfies Assumption~\ref{assump:SO-weak}. For sufficiently small $h>0$, the exact BRWP map is a global orientation-preserving diffeomorphism, and
}\begin{equation}
\blue{\KL(\cR_h\rho\|\rho^*)-\KL(\rho\|\rho^*) \le-\frac{h}{2\beta}\FI(\rho\|\rho^*)+Ch^3. \label{eq:analytic-descent}}
\end{equation}
\blue{Consequently, if the same bounds hold uniformly at every $\rho_j^{(h)}$, $0\le j<N(h)$, then, for sufficiently small $h$,
}\begin{equation}
\blue{\KL(\rho_k\|\rho^*) \le e^{-\alpha kh}\KL(\rho_0\|\rho^*)+\frac{C}{\alpha}h^2, \qquad 0\le k\le N(h). \label{eq:analytic-convergence}}
\end{equation}
\end{lemma}

\begin{proof}
{We first bound the spatial derivatives by the Fisher information. Substitution into the change-of-variables formula then bounds the increase in KL divergence. For small stepsizes, the leading entropy decrease controls the error terms.}

\emph{Spatial derivative estimates.} Use $\Phi,\widetilde\Phi$ from Assumption~\ref{assump:SO-weak} and put $f=\beta\Phi=\log(\rho/\rho^*)$ and $I=\int|\nabla f|^2\rho$. All norms without a subscript in this proof are in $L^2(\rho)$; matrix norms in integrals are Frobenius norms. The $W^{1,3}$ and $H^1$ bounds give $\nabla f,D^2f\in L^2(\rho)$. Since $\nabla(\widetilde\Phi-\Phi)=\beta^{-1}(h\nabla(\mathcal L^*\rho/\rho)+e_{k,h})$, the Lebesgue pairing in \eqref{eq:weak-oracle-entropy} implies $\nabla f\cdot\nabla(\mathcal L^*\rho/\rho)\in L^1(\rho)$.

We first justify integration by parts. Choose smooth cutoffs $0\le\chi_R\le1$ equal to one on $B_R$, supported in $B_{2R}$, with $\|\nabla\chi_R\|_\infty\le C/R$. Since $\nabla\cdot(\rho\nabla f)=\beta\mathcal L^*\rho$, compactly supported integration by parts gives
\[
\blue{\beta\int\chi_R^2\frac{|\mathcal L^*\rho|^2}{\rho} =-\int\chi_R^2\nabla f\cdot\nabla(\mathcal L^*\rho/\rho)\,\rho -2\int\chi_R\mathcal L^*\rho\,\nabla f\cdot\nabla\chi_R.}
\]
Young's inequality bounds the last term by half the left-hand side plus $C R^{-2}I$. Absorb the first part into the left-hand side. Fatou's lemma gives $\mathcal L^*\rho/\rho\in L^2(\rho)$. The last term now tends to zero by Cauchy--Schwarz. Thus $\beta^{-1}\int\nabla f\cdot\nabla(\mathcal L^*\rho/\rho)\,\rho=-\int \frac{|\mathcal L^*\rho|^2}{\rho}$, and the field difference bound and~\eqref{eq:weak-oracle-entropy} yield
\[
\blue{\int \frac{|\mathcal L^*\rho|^2}{\rho} =-h^{-1}\int\nabla f\cdot\nabla(\widetilde\Phi-\Phi)\rho +(\beta h)^{-1}\int \nabla f\cdot e_{k,h}\rho \le M\sqrt I+Ch.}
\]
The spatial Bochner identity in Lemma~\ref{lemma_second_order} is
\[
\int\left\{\beta^{-1}(\nabla f)^{\mathsf T}D^2V\nabla f +\beta^{-2}\|D^2f\|_F^2\right\}\rho =\int \frac{|\mathcal L^*\rho|^2}{\rho}+\beta^{-2}\int (\nabla f)^{\mathsf T}D^2f\nabla f\,\rho.
\]
To justify it here, set $a=\mathcal L^*\rho/\rho$, $H=D^2f$, and $Q=(\nabla f)^{\mathsf T}D^2\log\rho\,\nabla f$. Put $A=\|H\|_F^2+\beta\nabla f\cdot\nabla a$, which belongs to $L^1(\rho)$ by the bounds above. Direct differentiation gives $\nabla\cdot(\rho H\nabla f)=\rho(A-Q)$. Write $Q_+=\max\{Q,0\}$ and $Q_-=\max\{-Q,0\}$. Since $D^2\log\rho=\beta D^2\Phi-\beta D^2V\preceq\beta MI_d$, we have $Q_+\le\beta M|\nabla f|^2$. Integrating the divergence identity against $\chi_R$ and applying Cauchy--Schwarz gives
\begin{align*}
\blue{\int\chi_R Q_-\rho} &\blue{=\int\chi_R Q_+\rho-\int\chi_R A\rho-\int\nabla\chi_R\cdot H\nabla f\,\rho}\\
&\blue{\le\beta MI+\|A\|_{L^1(\rho)}+\frac{C}{R}\|H\|_{L^2(\rho)}\sqrt I.}
\end{align*}
For each fixed state and stepsize, the right-hand side is bounded uniformly for $R\ge1$. Since $\chi_R\to1$ pointwise, Fatou's lemma gives $Q_-\in L^1(\rho)$. Dominated convergence and the vanishing boundary term now justify passage to the limit and prove the Bochner identity. Since $D^2V\succeq0$, $D^2f=\beta D^2\Phi$ and $D^2\Phi\preceq MI_d$, this implies the interpolation estimate
\begin{equation}
\blue{\|D^2\Phi\|^2\le M\sqrt I+\frac M\beta I+Ch. \label{eq:analytic-interpolation}}
\end{equation}
Finite $I$ gives finite KL divergence by~\eqref{PL_ineq} and a finite second moment by the estimate following Lemma~\ref{lem:moment-p-discrete}.

\emph{Transport entropy estimate.} \blue{The upper bound on $D^2\widetilde\Phi$ follows from the kernel formula. Put $\psi=e^{-\beta V/2}$, $q_h=\G_h\psi$, $u_h=\rho/q_h$ and $c=M+L/2$. The convolution Hessian identity gives $D^2\log q_h\succeq-\beta LI_d/2$, hence $D^2\log u_h\preceq\beta cI_d$. For $2hc<1$, the conditional law $\mu_y(\dd x)\propto u_h(x)e^{-\beta|x-y|^2/(4h)}\dd x$ therefore has curvature at least $\beta(1-2hc)/(2h)$. Since $\K_V^h\rho=\psi\G_hu_h$, Brascamp--Lieb and Gaussian differentiation yield
}\[
\blue{D^2\widetilde\Phi(y)=\frac12D^2V(y)-\frac1{2h}I_d+\frac\beta{4h^2}\operatorname{Cov}_{\mu_y}(X)\preceq\left(\frac L2+\frac c{1-2hc}\right)I_d.}
\]
\blue{For $2hc\le1/2$, set $C_H=L/2+2c$ and take also $hC_H\le1/2$.} Then the potential $|x|^2/2-h\widetilde\Phi(x)$ has Hessian at least $I_d/2$. Subtracting any linear functional gives a coercive, strictly convex function with a unique minimizer. Thus $F_h=I_d-h\nabla\widetilde\Phi$ is bijective; its positive definite Jacobian gives a global orientation-preserving diffeomorphism by the inverse function theorem. With $W=-\log\rho^*$, the exact change-of-variables formula is
\[
\Delta D:=\KL(F_{h\#}\rho\|\rho^*)-\KL(\rho\|\rho^*) =\int[W(x-h\nabla\widetilde\Phi)-W(x)-\log\det(I_d-hD^2\widetilde\Phi)]\rho.
\]
The upper Hessian bound $D^2W\preceq\beta LI_d$ and the scalar inequality $-\log(1-z)\le z+z^2$ for $z\le1/2$ give
\begin{align*}
W(x-h\nabla\widetilde\Phi)-W(x)&\le-h\nabla W\cdot \nabla\widetilde\Phi +\tfrac12\beta Lh^2|\nabla\widetilde\Phi|^2,\\
-\log\det(I_d-hD^2\widetilde\Phi)&\le h\operatorname{tr}(D^2\widetilde\Phi)+h^2\|D^2\widetilde\Phi\|_F^2.
\end{align*}
For the second bound, apply the scalar inequality to the eigenvalues of the symmetric matrix $hD^2\widetilde\Phi$. For $z\le0$, the inequality follows by integrating $t/(1+t)\le t$ for $t\ge0$; for $0\le z\le1/2$, use the logarithm series. The displayed integrals are finite by the moment and derivative bounds. Indeed $\nabla\log\rho=\nabla f-\beta\nabla V\in L^2(\rho)$ by the second-moment bound; hence $\nabla\cdot(\rho \nabla\widetilde\Phi)$ is integrable. Testing against $\chi_R$ and letting $R\to\infty$ gives $\int\nabla\cdot(\rho \nabla\widetilde\Phi)=0$. Integration by parts therefore combines the linear terms into $-h\int \nabla f\cdot \nabla\widetilde\Phi\rho$, so
\begin{equation}
\blue{\Delta D\le-h\int \nabla f\cdot \nabla\widetilde\Phi\rho +h^2\left\{\frac{\beta L}{2}\|\nabla\widetilde\Phi\|^2+\|D^2\widetilde\Phi\|^2\right\}. \label{eq:analytic-map-bound}}
\end{equation}
The first-order field and derivative bounds, together with \eqref{eq:analytic-interpolation}, yield
\begin{align*}
-h\int \nabla f\cdot \nabla\widetilde\Phi\rho&\le-\frac h\beta I+Mh^2\sqrt I,\\
\|\nabla\widetilde\Phi\|^2&\le2\beta^{-2}I+2M^2h^2,\\
\|D^2\widetilde\Phi\|^2&\le2M\sqrt I+\frac{2M}{\beta}I+Ch+2M^2h^2.
\end{align*}
\emph{Absorbing the error terms.} Hence, for $h\le1$, $\Delta D\le-hI/\beta+Ch^2\sqrt I+Ch^2I+Ch^3$. For sufficiently small $h$, absorb $Ch^2I$ into $hI/(4\beta)$; Young's inequality bounds $Ch^2\sqrt I$ by $hI/(4\beta)+C'h^3$. This proves~\eqref{eq:analytic-descent}. The target logarithmic Sobolev inequality $I\ge2\beta\alpha\KL(\rho\|\rho^*)$ gives
\[
\KL(\rho_{j+1}\|\rho^*) \le(1-\alpha h)\KL(\rho_j\|\rho^*)+Ch^3.
\]
Taking also $\alpha h\le1$ and summing this recurrence proves \eqref{eq:analytic-convergence}.
\end{proof}

\begin{lemma}[\blue{One-step entropy comparison}]
\label{lem:analytic-defect}
\blue{Suppose that~\eqref{eq:strong-convexity} holds and that Assumption~\ref{assump:SO-weak} holds uniformly for $0\le j<N(h)$. There are constants $C_\varepsilon<\infty$ and $h_\varepsilon>0$, depending only on $d,\alpha,L,\beta$ and the uniform assumption bounds, such that
}\begin{equation}\blue{\label{eq:defect-assumption}
 \varepsilon_h(\rho_j)\le C_\varepsilon h^3, \qquad 0<h\le h_\varepsilon,\quad 0\le j<N(h). }\end{equation}
\end{lemma}

\begin{proof}
{Fix a current density $\rho$ on the stated horizon. We prove the uniform one-step bound
}\begin{equation}
\blue{\varepsilon_h(\rho)\le C h^3, \qquad 0<h\le h_\varepsilon. \label{eq:analytic-defect}}
\end{equation}
{Applying this bound to $\rho=\rho_j$ gives~\eqref{eq:defect-assumption}.} {We adapt the preceding transport estimate to $\rho_h^*$. This requires an upper bound on the kernel entropy decrease and a bound on the difference between the scores of $\rho_h^*$ and $\rho^*$.}

\emph{Kernel entropy decrease.} We first complement~\eqref{eq:entropy-fisher-loss} with an upper bound in terms of the input Fisher information $I_h=\FI(\rho\|\rho_h^*)$; write also $I=\FI(\rho\|\rho^*)$:
\begin{equation}
\blue{0\le\ell_h(\rho) \le\frac{h}{\beta(1+\alpha h)}I_h. \label{eq:kernel-loss-upper}}
\end{equation}
For $u=\rho/\rho_h^*$, the reversible joint law and conditional measures $\gamma_y$ in Proposition~\ref{prop:kernel} give
\[
\ell_h(\rho)=\int\operatorname{Ent}_{\gamma_y}(u)\,\rho_h^*(\dd y), \quad \operatorname{Ent}_{\gamma}(u) =\int u\log u\,\dd\gamma-\left(\int u\,\dd\gamma\right) \log\left(\int u\,\dd\gamma\right).
\]
The curvature of $\gamma_y$ is at least $\kappa_h=\beta(1+\alpha h)/(2h)$. Its logarithmic Sobolev inequality bounds this conditional entropy by $(2\kappa_h)^{-1}\int|\nabla u|^2/u\,\dd\gamma_y$. Integrating in $y$ and using the other marginal $\rho_h^*$ proves \eqref{eq:kernel-loss-upper}. Truncation and approximation extend the calculation from bounded positive density ratios to inputs with finite relative entropy and Fisher information, as used below.

\emph{Comparison of the scores of $\rho_h^*$ and $\rho^*$.} Let $x_*$ minimize $V$. We claim
\begin{equation}
\blue{|\nabla\log(\rho_h^*/\rho^*)(x)|\le Ch(1+|x-x_*|). \label{eq:reference-score-growth}}
\end{equation}
For $Y\sim\gamma_x$, write $m_x=\mathbb EY$ and $v_x=\mathbb E\nabla V(Y)$. Differentiating $q_h=\mathcal{G}_h\psi$ and integrating by parts in the conditional law give
\[
\nabla\log(\rho_h^*/\rho^*)(x)=-\frac\beta2\{v_x-\nabla V(x)\},\qquad m_x-x=-hv_x.
\]
Brascamp--Lieb gives $\mathbb E|Y-m_x|^2\le2dh/\beta$. Taylor expansion of $\nabla V$ about $m_x$, using the bounded third derivative and cancellation of the centered linear term, gives $|v_x-\nabla V(m_x)|\le Ch$. If $hL\le1/2$, then
\[
|v_x|\le L|x-x_*|+Lh|v_x|+Ch \quad\Longrightarrow\quad |v_x|\le2L|x-x_*|+Ch.
\]
Now $|v_x-\nabla V(x)|\le Lh|v_x|+Ch$ proves \eqref{eq:reference-score-growth}. The estimate $\sqrt I\le\beta(M+Mh)$ stated after Assumption~\ref{assump:SO-weak} bounds $I$ uniformly; the entropy and second-moment bounds used in the preceding proof are therefore uniform. Consequently,
\begin{equation}
\blue{\|\nabla\log(\rho_h^*/\rho^*)\|_{L^2(\rho)}\le Ch,\qquad \sqrt I\le\sqrt{I_h}+Ch,\qquad \|\nabla\widetilde\Phi-\beta^{-1}\nabla\log(\rho/\rho_h^*)\|_{L^2(\rho)}\le Ch. \label{eq:reference-score-comparison}}
\end{equation}
In particular $I_h<\infty$, while~\eqref{eq:initial-comparison} gives $D_h(\rho)<\infty$.

\emph{Transport estimate relative to $\rho_h^*$.} The potential $W_h=-\log\rho_h^*$ also satisfies $D^2W_h\preceq\beta LI_d$. Indeed, differentiation of the convolution gives
\[
-D^2\log q_h(x)=\frac\beta2\mathbb E_{\gamma_x}D^2V(Y) -\frac{\beta^2}{4}\operatorname{Cov}_{\gamma_x}(\nabla V(Y)) \preceq\frac{\beta L}{2}I_d,
\]
and $W_h=\beta V/2-\log q_h+\log Z_h$. Thus the exact transport estimate~\eqref{eq:analytic-map-bound} holds with reference $\rho_h^*$ and linear term $-h\int \nabla\log(\rho/\rho_h^*)\cdot \nabla\widetilde\Phi\rho$. The same bounds also ensure finite entropy of the updated density. Equations~\eqref{eq:analytic-interpolation} and \eqref{eq:reference-score-comparison}, with first-order difference bounds, imply
\[
\|D^2\widetilde\Phi\|_{L^2(\rho)}^2\le C\sqrt{I_h}+CI_h+Ch, \qquad \|\nabla\widetilde\Phi\|_{L^2(\rho)}^2\le CI_h+Ch^2.
\]
It follows that
\[
\blue{D_h(\cR_h\rho)-D_h(\rho) \le-\frac h\beta I_h+Ch^2\sqrt{I_h}+Ch^2I_h+Ch^3.}
\]
Adding $\tfrac12\ell_h(\rho)$ and using \eqref{eq:kernel-loss-upper} gives
\begin{align*}
\blue{D_h(\cR_h\rho)-D_h(\K_V^h\rho)-\tfrac12\ell_h(\rho)} &\blue{\le-\frac{h}{2\beta}I_h +Ch^2\sqrt{I_h}+Ch^2I_h+Ch^3.}
\end{align*}
For sufficiently small $h$, absorb $Ch^2I_h$ into half the negative Fisher term. Young's inequality absorbs $Ch^2\sqrt{I_h}$ into its remaining half and a remainder $C h^3$. Taking the positive part proves~\eqref{eq:analytic-defect}.
\end{proof}

The proofs also show that~\eqref{eq:weak-oracle-entropy} can be replaced by a uniform bound $\beta^{-1}\|\nabla(\mathcal L^*\rho/\rho)\|_{L^2(\rho)}\le M$: the same cutoff argument then gives $\int \frac{|\mathcal L^*\rho|^2}{\rho}\le M\sqrt I$. In that case neither weak pairing is used. For these two lemmas alone, the derivative bound can be weakened to $\|D^2(\widetilde\Phi-\Phi)\|_{L^2(\rho)}\le M\sqrt h$, with all other conditions unchanged: its squared contribution in~\eqref{eq:analytic-map-bound} is still $O(h^3)$. The stronger $H^1$ estimate in the analytic conditions of Assumption~\ref{assump:SO-weak} also supplies the remainder in Lemma~\ref{Lemma_KL_decay_pre}.

\end{appendix}
\bibliography{ref.bib}
\end{document}